\documentclass[10pt,aps,prb, preprint,reqno]{amsart}
\usepackage{amsmath}
\usepackage{amssymb}
\usepackage{yfonts}
\usepackage{hyperref}
\usepackage{mathrsfs}
\usepackage {latexsym}
\usepackage{graphics}
\usepackage{txfonts}
\usepackage{txfonts}	
\usepackage{breqn}
\usepackage{cite}
\usepackage{color}
\usepackage[shortlabels]{enumitem}

\newtheorem{theorem}{Theorem}[section]
\newtheorem{remark}{Remark}[section]
\newtheorem{lemma}[theorem]{Lemma}

\newtheorem{proposition}[theorem]{Proposition}

\newtheorem{define}{Definition}[section]  
 
\newtheorem{hypothesis}{Hypothesis}[section]

\newcommand{\joinR}{\hspace{-.1em}}
\newcommand{\RomanI}{I}
\newcommand{\RomanII}{\mbox{\RomanI\joinR\RomanI}}
\newcommand{\RomanIII}{\mbox{\RomanI\joinR\RomanII}}

\newcommand\pv{\mathop{\mathrm{P.V.}}}

\DeclareMathOperator*{\Tr}{Tr}
\DeclareMathOperator*{\Id}{Id}
\DeclareMathOperator*{\esssup}{esssup}
\DeclareMathOperator*{\supp}{supp}
\DeclareMathOperator*{\divergence}{div}

\allowdisplaybreaks

\begin{document}
\title[Non-uniqueness in law for the 2-d QG equations]{Non-uniqueness in law of the \\two-dimensional surface quasi-geostrophic equations forced by random noise}

 
\author[Kazuo Yamazaki]{Kazuo Yamazaki}  
\address{Texas Tech University, Department of Mathematics and Statistics, Lubbock, TX, 79409-1042, U.S.A.; Phone: 806-834-6112; Fax: 806-742-1112}
\email{kyamazak@ttu.edu}
\date{}

\begin{abstract}
Via probabilistic convex integration, we prove non-uniqueness in law of the two-dimensional surface quasi-geostrophic equations forced by random noise of additive type. In its proof we work on the equation of the momentum rather than the temperature, which is new in the study of the stochastic surface quasi-geostrophic equations. We also generalize the classical Calder$\acute{\mathrm{o}}$n commutator estimate to the case of fractional Laplacians.
\end{abstract}

\maketitle

\section{Introduction}\label{Section 1}
\subsection{Motivation from physics and mathematics}\label{Motivation from physics and mathematics}
Initially investigated in 1948 by Charney \cite{C48}, the two-dimensional (2-d) surface quasi-geostrophic (QG) equations  can be derived as the motion of potential temperature on the boundary upon considering Laplace's equation on $\mathbb{R}_{+}^{3}$ with the temperature as the partial derivative of the stream function with respect to (w.r.t.) the vertical variable (see \cite[equation (2)]{HPGS95} and \cite[equation (1.2)]{K10}). Its dissipative term in the form of a fractional Laplacian appears naturally and models Ekman pumping effect in the boundary layer (see \cite[equation (1)]{C02}). It has always attracted attention from physicists and engineers due to a wide breadth of applications in atmospheric sciences, meteorology, and oceanography, as well as mathematicians,  especially since Constantin, Majda, and Tabak \cite{CMT94} demonstrated similarities between the three-dimensional (3-d) Euler equations and the non-dissipative 2-d QG equations analytically and numerically. On the other hand, the study of partial differential equations (PDEs) in general hydrodynamics forced by random noise can be traced  back at least to \cite{LL56} by Landau and Lifschitz in 1957 (also \cite{BT73, N65}). The purpose of this manuscript is to study the QG equations forced by random noise of additive type and prove its non-uniqueness in law via probabilistic convex integration.

\subsection{Previous works}\label{Previous works}
We denote by $\mathbb{N} \triangleq \{1, 2, \hdots, \}$, $\mathbb{N}_{0} \triangleq \{0\} \cup \mathbb{N}$, $A \overset{(\cdot)}{\lesssim}_{a,b}B$ to imply the existence of $C(a,b) \geq 0$ such that $A\leq CB$ due to an equation $(\cdot)$ and $A \approx B$ if $A\lesssim B$ and $B\lesssim A$. For any vector $x$, we denote its $j$-th component by $x_{j}$ for $j \in \{1,\hdots, n\}$ and $x^{\bot} \triangleq (-x_{2}, x_{1})$ in case $n = 2$. For any square matrix $A$, we denote its $(m,l)$-entry by $A^{ml}$ and its trace-free part by $\mathring{A}$; in particular, we denote the trace-free part of a tensor product $A\otimes B$ by $A\mathring{\otimes}B$. We define $\partial_{t} \triangleq \frac{\partial}{\partial t}$ and $\partial_{i} \triangleq \frac{\partial}{\partial x_{i}}$ for $i \in \{1,\hdots, n\}$. The spatial domain of our primary interest is $\mathbb{T}^{2} = (\mathbb{R} /[-\pi,\pi])^{2}$. For full generality, let us introduce the generalized QG equations, which was introduced by Kiselev (see \cite[equation (1.3)]{K10}) and studied by many others (e.g., \cite{MX11}). We denote by $\theta: \mathbb{R}_{+} \times \mathbb{T}^{2} \mapsto \mathbb{R}$ the potential temperature, $\mathcal{R}$ the Riesz transform vector, $\Lambda^{r} \triangleq (-\Delta)^{\frac{r}{2}}$ for any $r \in \mathbb{R}$ to be the Fourier operator with its Fourier symbol $\lvert k \rvert^{r}$; i.e., $\Lambda^{r} f(x) \triangleq \sum_{k \in \mathbb{Z}^{2}} \lvert k \rvert^{r} \hat{f}(k) e^{ik\cdot x}$ where $\hat{f}$ denotes the Fourier transform of $f$ and $\check{f}$ will denote the Fourier inverse of $f$. Under these notations, with $\gamma_{1} \in (0, 2]$, $\gamma_{2} \in [1, 2]$, and $\nu \geq 0$, the generalized QG equations consist of 
\begin{equation}\label{est 1}
\partial_{t} \theta + u\cdot\nabla \theta + \nu \Lambda^{\gamma_{1}} \theta = 0, \hspace{3mm} u \triangleq \Lambda^{1-\gamma_{2}} \mathcal{R}^{\bot} \theta, 
\end{equation} 
given $\theta^{\text{in}}(x) \triangleq \theta(0,x)$ as initial data. Hereafter, for simplicity we assume $\nu = 1$ in case $\nu > 0$. We observe that if the initial data $\theta^{\text{in}}$ is mean-zero, which is a standard convention in the case of a spatial domain being a torus, the solution $\theta(t)$ remains mean-zero for all $t \geq 0$. The generalized QG equations reduce to the modified QG (MQG) equations introduced by Constantin, Iyer, and Wu \cite{CIW08} when $\gamma_{2} = 2-\gamma_{1}$ for $\gamma_{1} \in (0, 1]$, the classical QG equations when $\gamma_{2} = 1$ to which we refer  simply as the QG equations, and the Euler equations when $\gamma_{2} = 2, \nu = 0$. Due to the fact that the $L^{\infty}(\mathbb{T}^{2})$-norm seems to be the most useful bounded quantity in estimates that any smooth solution to \eqref{est 1} possesses and the rescaling property that $\lambda^{\gamma_{1} + \gamma_{2} - 2} \theta(\lambda^{\gamma_{1}} t, \lambda x)$ for any $\lambda \in \mathbb{R}_{+}$ solves the generalized QG equations if $\theta(t,x)$ does, we call the three cases $\gamma_{1} + \gamma_{2} \in (2, 4], \gamma_{1} + \gamma_{2} = 2,$ and $\gamma_{1} + \gamma_{2} \in (1, 2)$, sub-critical, critical, and super-critical cases w.r.t. $L^{\infty}(\mathbb{T}^{2})$-norm, respectively. The issue of global existence of the unique solution to the initial value problem of the QG equations has attracted much attention. Resnick in \cite[Theorem 2.1]{R95} was able to construct a global weak solution to the QG equations, even in case $\nu = 0$, starting from $\theta^{\text{in}} \in L^{2}(\mathbb{T}^{2})$ (see \cite[Theorem 1.1]{M08} for the case $\theta^{\text{in}} \in L^{p}$ for $p > \frac{4}{3}$). Within a class of smooth solutions, the global existence was settled in the sub-critical case (e.g., \cite{CW99}) and the critical case by breakthrough works \cite{CV10, KNV07} (see also \cite{CV12, KN10}). While this problem remains unsolved in the super-critical case, some properties of the solution were studied in various works (e.g., \cite{CW08, CW09, K11}). 

Many researchers have studied the following stochastic Navier-Stokes (NS) equations 
\begin{equation}\label{est 15} 
du + [(u\cdot\nabla) u + \nabla p -\nu \Delta u] dt = G(u) dB, \hspace{5mm} \nabla\cdot u = 0, 
\end{equation}  
where $u: \mathbb{R}_{+} \times \mathbb{T}^{n} \mapsto \mathbb{R}^{n}$, $p: \mathbb{R}_{+} \times \mathbb{T}^{n}\mapsto \mathbb{R}$, and $G(u)dB$ represent velocity field, pressure field, and random noise to be described subsequently in detail, respectively. The case $\nu  =0$ reduces \eqref{est 15} to the stochastic Euler equations. Flandoli and Romito \cite{FR08} proved the global existence of a Leray-Hopf weak solution to the NS equation forced by an additive noise; we call their solution Leray-Hopf because it satisfies the appropriate energy inequality (see \cite[MP3 on p. 421]{FR08}). However, path-wise uniqueness of their solution remains unknown, and their solution was not probabilistically strong, which requires that the solution is adapted to the canonical filtration generated by the noise. In general, upon any probabilistic Galerkin approximation (e.g., \cite[Appendix A]{FR08}), if one takes mathematical expectation to obtain uniform estimates, then the resulting solution becomes probabilistically weak. In case the noise is additive, one can consider an Ornstein-Uhlenbeck process and work on a random PDE satisfied by the difference between the solution and this Ornstein-Uhlenbeck process; nonetheless the solution obtained via this approach as a converging subsequence will depend on the fixed realization and hence not be probabilistically strong (see \cite[p. 84]{F08}). By Yamada-Watanabe theorem, path-wise uniqueness will lead to the solution becoming probabilistically strong. However, the general consensus was that the proof of the path-wise uniqueness seems to be no easier than its counterpart in the deterministic case and hence a significant amount of effort has been devoted to prove uniqueness in law (e.g., \cite[p. 878--879]{DD03}), which is weaker than path-wise uniqueness according to Yamada-Watanabe theorem and Tanaka's counterexample \cite[Example 2.2]{C03}. Due to Cherny's theorem, uniqueness in law and the existence of probabilistically strong solution will imply path-wise uniqueness; nonetheless, even the uniqueness in law of the Leray-Hopf weak solution constructed in \cite{FR08} has also remained open. 

For the QG equations forced by random noise, Zhu in \cite[Definition 4.2.1]{Z12b} defined its solution to require in particular the regularity of $L_{T}^{\infty} L_{x}^{2} \cap L_{T}^{2}H_{x}^{\gamma_{1}}$ and proved the global existence of a solution in case $\gamma_{1} > 0$ and the noise is additive in \cite[Theorem 4.2.4]{Z12b}. Its proof consisted of analysis on the random PDE satisfied by the difference between the solution to the QG equation forced by a certain random noise and an Ornstein-Uhlenbeck process so that the solution that was constructed therein was probabilistically weak. In \cite[Definition 4.3.1 and Theorem 4.3.2]{Z12b} Zhu also defined and constructed a martingale solution via the Galerkin approximation, and thus the solution is also probabilistically weak. The path-wise uniqueness of such a martingale solution in the sub-critical case with a multiplicative noise was also proven in \cite[Theorem 4.4.4]{Z12b} under the hypothesis that its initial data $\theta^{\text{in}}$ is in $L^{p}(\mathbb{T}^{2})$ for $p$ sufficiently large. We refer to large deviation principle in the sub-critical case in \cite[Sections 4.5 and 4.6]{Z12b}, the existence of Markov selections for any $\gamma_{1} > 0$,  strong Feller property, support theorem, the existence of a unique invariant measure, and exponential convergence if $\gamma_{1} > \frac{4}{3}$ respectively in  \cite[Theorems 4.2.5, 4.3.3, 4.3.8, 4.3.10, and 4.4.5]{Z12a} (also \cite{Y22a}) where the fourth result was improved to $\gamma_{1} > 1$ in \cite[Theorems 5.9 and 5.13]{RZZ15}, as well as regularization of multiplicative noise in \cite{BNSW18}. 

Next, we review recent developments on the convex integration that was fueled in the past decade with the goal to prove Onsager's conjecture \cite{O49}, the positive direction being that every weak solution $u \in C^{\alpha}(\mathbb{T}^{3})$ to the 3-d Euler equations for $\alpha > \frac{1}{3}$ conserves its energy and the negative direction being the existence of a weak solution $u \in C^{\alpha}(\mathbb{T}^{3})$ for $\alpha < \frac{1}{3}$ that fails to conserve its energy. While Constantin, E, and Titi \cite{CET94} and Eyink \cite{E94} in 1994 proved its positive direction, De Lellis and Sz$\acute{\mathrm{e}}$kelyhidi Jr. \cite{DS09}, by partially using ideas from \cite{MS03}, proved the existence of a solution $u \in L_{t,x}^{\infty}$ to the $n$-d Euler equations for $n \in \mathbb{N}\setminus \{1\}$ with compact support in space and time, extending previous works of Scheffer \cite{S93} and Shnirelman \cite{S97} that proved analogous results with regularity in $L_{t,x}^{2}$ in the 2-d case. After further extensions (e.g., \cite{DS10, DS13, BDIS15}), Isett \cite{I18} completely proved the negative direction of the Onsager's conjecture in any dimension $n \geq 3$. Via an introduction of intermittent Beltrami waves, Buckmaster and Vicol \cite{BV19a} proved the non-uniqueness of weak solutions to the 3-d NS equations, and it was followed by many more: \cite{BCV18, LQ20, CDD18, D19, LT20} on the NS equations; \cite{BMS21} on power-law model; \cite{LTZ20} on Boussinesq system;  \cite{BBV20, FLS21} on magnetohydrodynamics (MHD) system; \cite{CGSW15, MS20, MS18} on transport equation.

Concerning active scalar equations, C$\acute{\mathrm{o}}$rdoba, Faraco, and Gancedo \cite{CFG11} applied the convex integration approach of \cite{DS09} to the incompressible porous media equations. Subsequently, Isett and Vicol \cite{IV15} proved the existence of weak solutions with compact support for general active scalar equations $\partial_{t} \theta +u \cdot \nabla \theta = 0$ as long as $u = T[\theta]$ is divergence-free and $T$ is a Fourier operator with a Fourier symbol that is not an odd function of frequency; consequently, their result excluded the QG equations \eqref{est 1}. By defining $v \triangleq \Lambda^{-1} u$ as a potential velocity, Buckmaster, Shkoller, and Vicol in \cite{BSV19} worked on the QG momentum equations
\begin{subequations}\label{est 2} 
\begin{align}
&\partial_{t} v + (u \cdot \nabla) v - (\nabla v)^{T} \cdot u + \nabla p + \Lambda^{\gamma_{1}} v = 0, \hspace{3mm} \nabla\cdot v = 0, \\
&u = \Lambda v; 
\end{align}
\end{subequations} 
the relationship between $v$ and the solution $\theta$ to \eqref{est 1} is 
\begin{equation}\label{est 41}
\theta = - \nabla^{\bot} \cdot v. 
\end{equation} 
Considering that the solution to \eqref{est 2} starting from $v(0) \in \dot{H}^{\frac{1}{2}}(\mathbb{T}^{2})$ satisfies $\lVert v (t) \rVert_{\dot{H}_{x}^{\frac{1}{2}}}^{2} \leq \lVert v(0) \rVert_{\dot{H}_{x}^{\frac{1}{2}}}^{2}$, which can be seen from the identity of 
\begin{equation}\label{est 38}  
(u\cdot\nabla) v - (\nabla v)^{T} \cdot u = u^{\bot} (\nabla^{\bot} \cdot v)
\end{equation} 
(see \cite[Theorem 1.2 (ii)]{M08} and \cite[Section 2.2]{R95}), the authors in \cite{BSV19} were able to construct a solution with prescribed $\dot{H}^{\frac{1}{2}}(\mathbb{T}^{2})$-norm and thereby conclude non-uniqueness; let us describe its proof in detail in Subsection \ref{Subsection 2.2}. Subsequently, Isett and Ma \cite{IM21} provided a direct approach by working on the actual QG equations; moreover, Cheng, Kwon, and Li \cite{CKL20} proved non-uniqueness of steady-state weak solution to the QG equations. 

The impact of the development of convex integration has spilled over to the community of researchers on stochastic PDEs recently. First, Breit, Feireisl, and Hofmanov$\acute{\mathrm{a}}$ \cite{BFH20} and Chiodaroli, Feireisl, and Flandoli \cite{CFF19} proved path-wise non-uniqueness of certain compressible Euler equations using some ideas of convex integration from \cite{DS10} (see \cite{HZZ20} on the incompressible Euler equations). Second, Hofmanov$\acute{\mathrm{a}}$, Zhu, and Zhu proved non-uniqueness in law of the 3-d NS equations via approach similar to \cite[Section 7]{BV19b} where non-uniqueness was proven via opposite energy inequality. This result inspired more to follow: \cite{RS21, Y20a, Y20c, Y21a, Y21c, Y21d}. Furthermore, Hofmanov$\acute{\mathrm{a}}$, Zhu, and Zhu \cite{HZZ21a} were able to prove non-uniqueness in law of the 3-d NS equations with prescribed initial data or prescribed energy by adapting the approach of \cite{BMS21}; we emphasize that \cite[Theorem 1.1]{HZZ21a} does not imply \cite[Theorem 1.2]{HZZ19} due to the difference in the solutions' regularity. We also refer to \cite{HZZ21b} on the 3-d NS equations forced by space-time white noise and \cite{KY22} on transport equation forced by random noise of three types: additive; linear multiplicative; transport. 

\section{Statement of main results and new ideas to overcome difficulty}\label{Section 2}
\subsection{Statement of main results}\label{Subsection 2.1}
Following \eqref{est 2} we consider 
\begin{subequations}\label{est 30}
\begin{align}
&dv + [(u\cdot\nabla) v - (\nabla v)^{T} \cdot u + \nabla p + \Lambda^{\gamma_{1}} v]dt = G(v) dB, \hspace{3mm} \nabla\cdot v = 0,  \label{est 30a}\\
&u = \Lambda v, \label{est 30b}
\end{align}
\end{subequations}
where $G(v)$ becomes a $GG^{\ast}$-Wiener process that is independent of $v$ to be described in detail subsequently; here, we denoted an adjoint of $G$ by $G^{\ast}$. Given any deterministic initial data $v^{\text{in}} \in \dot{H}^{\frac{1}{2}}(\mathbb{T}^{2})$, via a Galerkin approximation under a suitable condition on $G$, one can construct a solution $v$ to \eqref{est 30} that satisfies
\begin{equation}\label{est 31}
\mathbb{E}^{\textbf{P}} [\lVert v(t) \rVert_{\dot{H}_{x}^{\frac{1}{2}}}^{2}] \leq \lVert v^{\text{in}} \rVert_{\dot{H}_{x}^{\frac{1}{2}}}^{2} + t \Tr ((-\Delta)^{\frac{1}{2}} GG^{\ast})
\end{equation} 
where $\mathbb{E}^{\textbf{P}}$ represents the mathematical expectation w.r.t. a probability measure $\textbf{P}$.
\begin{theorem}\label{Theorem 2.1} 
Suppose that  $\gamma_{1} \in (0, \frac{3}{2})$ and $G(v) dB$ in \eqref{est 30} is $dB$, a $GG^{\ast}$-Wiener process such that 
\begin{equation}\label{est 161} 
\Tr ((-\Delta)^{4- \frac{\gamma_{1}}{2} + 2 \sigma} GG^{\ast}) < \infty
\end{equation} 
for some $\sigma > 0$. Then, given any $T > 0, K > 1$, and $\kappa \in (0,1)$, there exist a sufficiently small $\iota \in \left(0, \frac{1}{2} - \frac{\gamma_{1}}{3}\right)$ and a $\textbf{P}$-almost surely (a.s.) strictly positive stopping time $\mathfrak{t}$ such that  
\begin{equation}\label{est 148}
\textbf{P} (\{\mathfrak{t} \geq T \}) > \kappa 
\end{equation} 
and the following is additionally satisfied. There exists a $(\mathcal{F}_{t})_{t\geq 0}$-adapted process $v$ that is a weak solution of \eqref{est 30} starting from a deterministic initial data $v^{\text{in}}$, satisfies 
\begin{equation}\label{est 129}
\esssup_{\omega \in \Omega} \lVert v(\omega) \rVert_{C_{\mathfrak{t}} C_{x}^{\frac{1}{2} + \iota}} + \esssup_{\omega \in \Omega} \lVert v(\omega) \rVert_{C_{\mathfrak{t}}^{\eta}C_{x}}  < \infty \hspace{3mm} \forall \hspace{1mm} \eta \in \left(0, \frac{1}{3}\right], 
\end{equation} 
and on a set $\{ \mathfrak{t} \geq T \}$, 
\begin{equation}\label{est 146}
\lVert v(T) \rVert_{\dot{H}_{x}^{\frac{1}{2}}} > K \left[ \lVert v^{\text{in}} \rVert_{\dot{H}_{x}^{\frac{1}{2}}} + \sqrt{ T \Tr ((-\Delta)^{\frac{1}{2}} GG^{\ast} ) } \right]. 
\end{equation} 
\end{theorem} 
\noindent The hypothesis \eqref{est 161} is needed to gain sufficiently high spatial regularity of $z$ from \eqref{est 152} in Proposition \ref{Proposition 4.4} and  handle some terms in the Reynolds stress (e.g., $R_{\text{Com2}}$ in \eqref{est 191e}). 

\begin{theorem}\label{Theorem 2.2} 
Suppose that  $\gamma_{1} \in (0, \frac{3}{2})$ and $G(v) dB$ in \eqref{est 30} is $dB$, a $GG^{\ast}$-Wiener process such that \eqref{est 161} holds for some $\sigma > 0$. Then non-uniqueness in law for \eqref{est 30} holds on $[0,\infty)$. Moreover, for all $T > 0$ fixed, non-uniqueness in law holds for \eqref{est 30} on $[0,T]$. 
\end{theorem}  
\noindent To the best of the author's knowledge, this is the first manuscript to study the QG momentum equations forced by random noise.

\subsection{Difficulties and new ideas to overcome them}\label{Subsection 2.2}
Let us elaborate on the proof of non-uniqueness of the deterministic QG equations in \cite{BSV19}. The authors fixed an arbitrary smooth function $\mathcal{H}: [0,T] \mapsto \mathbb{R}_{+}$ with compact support, defined 
\begin{equation}\label{est 3}
\lambda_{q} \triangleq \lambda_{0}^{q} \text{ for } \lambda_{0} \in 5 \mathbb{N} \text{ sufficiently large and } \delta_{q} \triangleq \lambda_{0}^{2} \lambda_{q}^{-2\beta}, 
\end{equation} 
constructed a sequence of solutions $(v_{0}, p_{0}, \mathring{R}_{0}) \equiv (0,0,0)$ and $(v_{q}, p_{q}, \mathring{R}_{q})_{q\in\mathbb{N}}$ to 
\begin{subequations}\label{est 6}
\begin{align}
& \partial_{t} v_{q} + (u_{q} \cdot \nabla) v_{q} - (\nabla v_{q})^{T} \cdot u_{q} + \nabla p_{q} + \Lambda^{\gamma_{1}} v_{q} = \nabla\cdot \mathring{R}_{q}, \hspace{3mm} \nabla\cdot v_{q} = 0, \\
&u_{q} = \Lambda v_{q},  
\end{align}
\end{subequations}
where $\mathring{R}_{q}$ is a symmetric trace-free matrix in $\mathbb{R}^{2\times 2}$ such that for universal constants $C_{0} > 0$ and $\epsilon_{R} > 0$, they satisfy the following inductive hypothesis. 
\begin{hypothesis}\label{Hypothesis 2.1} \rm{(\cite[equations (3.3)-(3.9)]{BSV19})} For all $t \in [0,T]$, 
\noindent
\begin{enumerate}[(a)]
\item\label{BSV19 Hypo a} $\supp \hat{v}_{q} \subset B(0, 2 \lambda_{q})$, a ball of radius $2 \lambda_{q}$ centered at the origin,  
\begin{equation}\label{est 4}
\lVert v_{q} \rVert_{C_{t} C_{x}^{1}} + \lVert u_{q} \rVert_{C_{t,x}} \leq C_{0} \lambda_{q} \delta_{q}^{\frac{1}{2}}, 
\end{equation} 
\item\label{BSV19 Hypo b} $\supp \hat{\mathring{R}}_{q} \subset B(0, 4\lambda_{q})$, 
\begin{equation}\label{est 5}
\lVert \mathring{R}_{q} \rVert_{C_{t,x}} \leq \epsilon_{R} \lambda_{q+1} \delta_{q+1}, 
\end{equation} 
\item\label{BSV19 Hypo c} 
\begin{equation}\label{est 20}
\lVert (\partial_{t} + u_{q} \cdot \nabla) v_{q} \rVert_{C_{t,x}} \leq C_{0} \lambda_{q}^{2} \delta_{q}, 
\end{equation} 
\item\label{BSV19 Hypo d} 
\begin{equation}\label{est 21}
\lVert (\partial_{t} + u_{q} \cdot \nabla) u_{q} \rVert_{C_{t,x}} \leq C_{0} \lambda_{q}^{3} \delta_{q}, 
\end{equation} 
\item\label{BSV19 Hypo e} 
\begin{equation}\label{est 22}
\lVert (\partial_{t} + u_{q} \cdot \nabla) \mathring{R}_{q} \rVert_{C_{t,x}} \leq \lambda_{q}^{2} \delta_{q}^{\frac{1}{2}} \lambda_{q+1} \delta_{q+1}, 
\end{equation} 
\item\label{BSV19 Hypo f} 
\begin{equation}\label{est 23}
0 \leq \mathcal{H}(t) - \lVert v_{q}(t) \rVert_{\dot{H}_{x}^{\frac{1}{2}}}^{2} \leq \lambda_{q+1} \delta_{q+1}, 
\end{equation} 
\item\label{BSV19 Hypo g} 
\begin{equation}\label{est 24}
\mathcal{H}(t) - \lVert v_{q}(t) \rVert_{\dot{H}_{x}^{\frac{1}{2}}}^{2} \leq \frac{\lambda_{q+1} \delta_{q+1}}{8} \Rightarrow \mathring{R}_{q}(t) \equiv 0. 
\end{equation} 
\end{enumerate} 
\end{hypothesis}
\noindent The hypothesis \ref{BSV19 Hypo a} gives the spatial regularity of the solution $v$ to \eqref{est 2} that is derived by taking limit $q\nearrow + \infty$ and observing the hypothesis \ref{BSV19 Hypo b}. The hypothesis \ref{BSV19 Hypo c} gives the temporal regularity of $v$ while the hypothesis \ref{BSV19 Hypo f} guarantees that the $\dot{H}^{\frac{1}{2}}(\mathbb{T}^{2})$-norm of $v(t)$ is $\mathcal{H}(t)$. To explain the role of other hypothesis, let us explain the construction of the solution $v_{q+1}$ to \eqref{est 6} that satisfies the Hypothesis \ref{Hypothesis 2.1}. The authors in \cite{BSV19} let $\chi_{j} \in [0,1]$ be certain smooth cutoff function such that 
\begin{equation}\label{est 7}
\sum_{j\in\mathbb{Z}} \chi^{2}(t-j) = 1 \hspace{3mm} \forall \hspace{1mm} t \in \mathbb{R}
\end{equation} 
and defined 
\begin{equation}\label{est 8}
\chi_{j}(t) \triangleq \chi(\tau_{q+1}^{-1} t - j).  
\end{equation}  
They selected $\tau_{q+1}$ in \cite[equation (4.7)]{BSV19} carefully (as we will too in \eqref{est 210}), and defined $\Phi_{j}$ and $\mathring{R}_{q,j}$ to be respectively the solutions to 
\begin{subequations}\label{est 10}
\begin{align}
(\partial_{t}+  u_{q} \cdot\nabla) \Phi_{j} =& 0, \\
 \Phi_{j} (\tau_{q+1} j, x) =& x, 
\end{align}
\end{subequations}
and 
\begin{subequations}\label{est 11} 
\begin{align}
(\partial_{t} + u_{q} \cdot \nabla) \mathring{R}_{q,j} =& 0, \label{est 25}\\
\mathring{R}_{q,j} (\tau_{q+1}j, x) =& \mathring{R}_{q}(\tau_{q+1}j, x). \label{est 26}  
\end{align}
\end{subequations} 
Postponing some details of notations, we mention that the authors of \cite{BSV19} defined 
\begin{equation}\label{est 16}
v_{q+1} \triangleq v_{q}  + w_{q+1}
\end{equation} 
where the perturbation $w_{q+1}$ is defined via 
\begin{equation}\label{est 17}
w_{q+1}(t,x) \triangleq \sum_{j \in \mathbb{Z}} \sum_{k\in\Gamma_{j}}  \chi_{j}(t) \mathbb{P}_{q+1, k} (a_{k,j} (t,x) b_{k}(\lambda_{q+1} \Phi_{j}(t,x)) 
\end{equation} 
with 
\begin{equation}\label{est 18} 
a_{k,j}(t,x) \triangleq 
\begin{cases}
\rho_{j}^{\frac{1}{2}} \gamma_{k} \left( \Id - \frac{\mathring{R}_{q,j}}{\lambda_{q+1}\rho_{j}} \right) &\text{ if } \rho_{j} \neq 0, \\
0 & \text{ if } \rho_{j} = 0, 
\end{cases}
\end{equation} 
\begin{equation}\label{est 19} 
\rho(t) \triangleq \frac{1}{(2\pi)^{2} \lambda_{q+1}} \max\left\{ \mathcal{H}(t) - \lVert v_{q}(t) \rVert_{\dot{H}^{\frac{1}{2}}}^{2} - \frac{\lambda_{q+2} \delta_{q+2}}{2}, 0 \right\}, \hspace{2mm} \text{ and } \hspace{2mm} \rho_{j} \triangleq \rho(\tau_{q+1} j) 
\end{equation} 
(see Lemma \ref{Geometric Lemma} for $\gamma_{k}$ and $\Gamma_{j}$, \eqref{est 192} for $\mathbb{P}_{q+1, k}$, and \eqref{est 178} for $b_{k}$). It is shown in \cite[Lemma 6.1]{BSV19} that the hypothesis \ref{BSV19 Hypo g} guarantees that if $\rho_{j} = 0$, then $\mathring{R}_{q}(\cdot, t) \equiv 0$ for all $t \in \supp \chi_{j}$ which, along with \eqref{est 7}-\eqref{est 8}, justified 
\begin{align}
\divergence \mathring{R}_{q}(t) = \sum_{j\in\mathbb{Z}: \rho_{j} \neq 0} \chi_{j}^{2}(t) \divergence (\mathring{R}_{q} - \mathring{R}_{q,j}) (t) + \sum_{j\in\mathbb{Z}: \rho_{j} \neq 0} \chi_{j}^{2}(t) \divergence \mathring{R}_{q,j}(t). \label{est 9}
\end{align}
The decomposition \eqref{est 9} is absolutely crucial in the estimate of the difficult Reynolds oscillation term because \eqref{est 11} indicates that 
\begin{subequations}\label{est 14} 
\begin{align}
(\partial_{t} + u_{q} \cdot \nabla) (\mathring{R}_{q} - \mathring{R}_{q,j}) =& (\partial_{t} + u_{q} \cdot \nabla) \mathring{R}_{q}, \label{est 12}\\
(\mathring{R}_{q} - \mathring{R}_{q,j}) (\tau_{q+1} j, x) =& 0, \label{est 13}
\end{align}
\end{subequations} 
so that we can estimate 
\begin{equation}\label{est 27}  
\lVert (\mathring{R}_{q} - \mathring{R}_{q,j})(t) \rVert_{C_{x}} \overset{\eqref{est 197a}}{\leq} \int_{\tau_{q+1} j}^{t} \lVert (\partial_{s} + u_{q} \cdot \nabla) \mathring{R}_{q} (s) \rVert_{C_{x}} ds \overset{\eqref{est 22}}{\leq} (t - \tau_{q+1} j) \lambda_{q}^{2} \delta_{q}^{\frac{1}{2}} \lambda_{q+1} \delta_{q+1}, 
\end{equation} 
and thereby estimate $\sum_{j\in\mathbb{Z}: \rho_{j} \neq 0} \chi_{j}^{2}(t) (\mathring{R}_{q} - \mathring{R}_{q,j}) (t)$ in \eqref{est 9}. Finally, having used the hypothesis \ref{BSV19 Hypo e} at level $q$ to get the necessary bound on the Reynolds oscillation term and thereby prove the hypothesis \ref{BSV19 Hypo b} at level $q+1$, one now needs to prove the hypothesis \ref{BSV19 Hypo e} at level $q+1$ and the hypothesis \ref{BSV19 Hypo d} is needed for that purpose. 

Let us now list some immediate difficulties in adapting this proof to the stochastic case.  
\begin{enumerate}[label=(\Alph*)]
\item The primary difficulty in adapting this approach of \cite{BSV19} comes from the hypothesis \ref{BSV19 Hypo e}. We need to transform the problem of \eqref{est 30a} to a random PDE, add a Reynolds stress as a force, and consider its solution inductively (see \eqref{est 134}). E.g., in the case of the 3-d NS equations \eqref{est 15}  forced by an additive noise, authors in \cite{HZZ19} let $z$ be a solution to a Stokes problem forced by the same noise (cf. \eqref{est 152}), defined $y \triangleq u-z$ to obtain a random PDE solved by $y$ that has the nonlinear term of 
\begin{equation}\label{est 410} 
(u\cdot\nabla) u = \divergence ((y+z) \mathring{\otimes} (y+z)) + \nabla \frac{\lvert y+z \rvert^{2}}{3} 
\end{equation} 
so that $\frac{\lvert y +z\rvert^{2}}{3}$ can be part of its pressure term. We observe that $(y+z) \mathring{\otimes} (y+z)$ in \eqref{est 410} is trace-free and symmetric and hence can become a part of the Reynolds stress. However, this implies that the Reynolds stress $\mathring{R}_{q}$ will be in $C_{t}^{\gamma}$ only for $\gamma < \frac{1}{2}$ due to the presence of $z$ (see \eqref{est 152} and Proposition \ref{Proposition 4.4}). Therefore, the corresponding Reynolds stress cannot possibly satisfy the hypothesis \ref{BSV19 Hypo e}. 
\item With a direct application of Hypothesis \ref{Hypothesis 2.1} to the stochastic case out of the picture, it is tempting to try to adapt to \eqref{est 2} the convex integration scheme of \cite[Section 7]{BV19b} on the 3-d NS equations that was proven to be successful in the stochastic case (e.g., \cite{HZZ19}). However, this ends up futile, quite clearly because the nonlinear term within \eqref{est 2} is one derivative more singular than that of the NS equations \eqref{est 15}; cf. $(u\cdot\nabla)u$ in \eqref{est 15} and $(\Lambda v \cdot \nabla) v - (\nabla v)^{T} \cdot \Lambda v$ in \eqref{est 2}. More specifically, the convex integration in \cite[Section 7]{BV19b} is based on intermittent jets (see \cite[equation (7.17)]{BV19b}) that allow delicate $L_{x}^{p}$-estimates whereas the convex integration in \cite{BSV19} such as $w_{q+1}$ and $a_{k,j}$ in \eqref{est 17}-\eqref{est 18} is only fit for $L_{x}^{\infty}$-estimate. Similarly, adapting the convex integration on the 2-d NS equations in \cite{LQ20, Y20c} to \eqref{est 2} seems difficult. 

In fact, upon a closer look, one realizes additional disadvantages in the case of the QG equations. The term $(\nabla v)^{T} \cdot \Lambda v$ cannot be written in a divergence form and an estimate on $\mathcal{B} ((\nabla v)^{T} \cdot \Lambda v)$ where $\mathcal{B}$ is an inverse of a divergence from Lemma \ref{Inverse-divergence lemma} is essentially as difficult as $(\nabla v)^{T} \cdot \Lambda v$ because the frequency support of $(\nabla v)^{T} \cdot \Lambda v$ is not compactly supported away from the origin in general and hence $\mathcal{B} ((\nabla v)^{T} \cdot \Lambda v)$ is actually two derivatives worse than the counterpart $u \mathring{\otimes} u$ in the case of the NS equations. Even for $(\Lambda v \cdot \nabla) v$, an analogous attempt to \eqref{est 410} leads us to 
\begin{equation}
(\Lambda v\cdot\nabla) v = \divergence ( \Lambda (y+z)  \mathring{\otimes } (y+z)) + \nabla \frac{ \lvert \Lambda(y+z) \cdot (y+z) \rvert^{2}}{2} 
\end{equation} 
and we realize that while $\Lambda (y+z)  \mathring{\otimes } (y+z)$ is  trace-free, it is not symmetric. 
\item Wishful thinking would be to essentially just drop the hypothesis \ref{BSV19 Hypo e} and see if the proof can be modified to go through; this is not impossible if one's goal is merely to prove non-uniqueness so that we should not have to perform a full-fledged proof of \cite{BSV19} that verifies high regularity of the solution with a prescribed $\dot{H}^{\frac{1}{2}}(\mathbb{T}^{2})$-norm. Unfortunately, this hope is destroyed the moment we realize that the hypothesis \ref{BSV19 Hypo e} at level $q$ is used to prove the hypothesis \ref{BSV19 Hypo b} at level $q+1$, as we described in \eqref{est 9}-\eqref{est 27}, and the hypothesis \ref{BSV19 Hypo b} is indispensable as that is needed to prove that $\mathring{R}_{q}$ vanishes in the limit $q\nearrow + \infty$ and conclude that the limiting solution $v$ solves \eqref{est 2}.  
\item\label{Difficulty D} Besides these issues, there is another fundamental problem in adapting deterministic convex integration to the stochastic case. The authors in \cite{BSV19} fixed an arbitrary smooth function $\mathcal{H}: [0,T]\mapsto \mathbb{R}_{+}$ with compact support such as $[a,b] \subsetneq [0,T]$, constructed a solution $v$ to \eqref{est 2} such that $\lVert v(t) \rVert_{\dot{H}_{x}^{\frac{1}{2}}}^{2} = \mathcal{H}(t)$ which implies that $v$ has compact support in time, and thus conclude non-uniqueness arguing that $v \equiv 0$ is not the only weak solution that vanishes on the complement of $[a,b]$. In case the QG equations are forced by an additive noise, this argument breaks down because even if its solution vanishes at time $t  = a$, a zero function would not be its solution for $t > a$ anyway. Considering the inequality \eqref{est 31} that is satisfied by the solution $v$ constructed via a Galerkin approximation, one may be tempted to fix $\mathcal{H}(t) \triangleq 2 \left( \lVert v^{\text{in}} \rVert_{\dot{H}_{x}^{\frac{1}{2}}}^{2} + t \Tr ((-\Delta)^{\frac{1}{2}} GG^{\ast})\right)$ to prove non-uniqueness; unfortunately, while $\mathcal{H}(t)$ must be fixed \emph{a priori}, the initial data $v^{\text{in}}$ cannot be prescribed in the convex integration scheme of \cite{BSV19}. 
\end{enumerate} 

Our new ideas to overcome these difficulties are as follows.
\begin{enumerate}[label=(\Alph*)]
\item The first idea is to replace $\rho_{j}$ in \eqref{est 19}, and hence in the definition of $a_{k,j}$ in \eqref{est 18} and ultimately  $w_{q+1}$ in \eqref{est 17} by $\delta_{q+1}$ (compare \eqref{est 18} and \eqref{est 196}). In particular, this implies that $\rho_{j}$ in \eqref{est 19} is always strictly positive, removing the necessity to consider distinct sets $\{j: \rho_{j} \neq 0\}$ and $\{j: \rho_{j} = 0\}$ in contrast to \cite{BSV19} and thus the hypothesis \ref{BSV19 Hypo g}. We also do not need to include the hypothesis \ref{BSV19 Hypo f} because we are willing to prove non-uniqueness by constructing solutions with opposite $\dot{H}^{\frac{1}{2}}(\mathbb{T}^{2})$-inequality. 
\item Second, we will mollify $\mathring{R}_{q}$ in space and time; merely doing this will not make any difference because $\mathring{R}_{q}$ remains non-differentiable in time. Specifically, we define 
\begin{equation}\label{est 32}
l \triangleq \lambda_{q+1}^{-\alpha} 
\end{equation} 
(see \eqref{est 154c} for the range of $\alpha$), let $\{ \phi_{l} \}_{l> 0}$ and $\{\varphi_{l}\}_{l> 0}$ be families of standard mollifiers with mass one and compact support respectively on $\mathbb{R}^{2}$ and in $(\tau_{q+1}, 2\tau_{q+1}] \subset \mathbb{R}_{+}$ (see \eqref{est 210} for a precise definition of $\tau_{q+1}$). Then we extend $\mathring{R}_{q}$ for $q \in \mathbb{N}_{0}$ to $t < 0$ by its value at $t = 0$, mollify it in space and time to obtain 
\begin{equation}\label{est 33} 
\mathring{R}_{l} \triangleq \mathring{R}_{q} \ast_{x} \phi_{l} \ast_{t} \varphi_{l}, 
\end{equation} 
and replace $\mathring{R}_{q}(\tau_{q+1}j, x)$ in \eqref{est 26} by $\mathring{R}_{l}(\tau_{q+1}j, x)$ (see \eqref{est 195b}). This will allow us to obtain \eqref{est 9} and \eqref{est 14} with $\mathring{R}_{q}$ therein replaced by $\mathring{R}_{l}$ (see \eqref{est 259}) leading us to \eqref{est 27} in which we do not even need \eqref{est 22} because we can estimate within \eqref{est 27} 
\begin{equation}\label{est 34}  
\lVert (\partial_{s} + u_{q} \cdot \nabla) \mathring{R}_{l} (s) \rVert_{C_{x}} \leq \lVert \partial_{s} \mathring{R}_{l} \rVert_{C_{s,x}} + \lVert u_{q} \rVert_{C_{s,x}} \lVert \nabla \mathring{R}_{l} \rVert_{C_{s,x}} \lesssim l^{-1} \lambda_{q+1} \delta_{q+1} + \lambda_{q} \delta_{q}^{\frac{1}{2}} \lambda_{q} \lambda_{q+1} \delta_{q+1}
\end{equation}
by mollifier estimates, Hypothesis \ref{Hypothesis 2.1} \ref{BSV19 Hypo a} and \ref{BSV19 Hypo b} (see \eqref{est 266}), completely removing the necessity for the hypothesis \ref{BSV19 Hypo e} and consequently also the hypothesis \ref{BSV19 Hypo d}. 
\item The two major changes we suggested above have many consequences. One important ingredient of any convex integration scheme is to cancel out the most difficult term in the Reynolds stress oscillation term. The changes we suggested, along with many others, must be made carefully to preserve the crucial cancellation (see \eqref{est 382}). Moreover, replacing $\divergence \mathring{R}_{q}$ in \eqref{est 9} by $\divergence \mathring{R}_{l}$ leads to the necessity of considering the mollified equation; i.e., we mollify not just $\mathring{R}_{q}$ but all other functions and work on the equation satisfied by the mollified functions (see \eqref{est 184}). This leads to an additional Reynolds stress error that was absent in \cite{BSV19}, informally
\begin{equation*}
R_{\text{Com1}} \triangleq \text{Nonlinear term mollified} - \text{Nonlinear term with each term therein mollified}
\end{equation*} 
(see \eqref{est 181}) for which we can apply mollifier estimates and end up with a bound that is smaller if we take large $\alpha > 0$. On the other hand, \eqref{est 34} already reveals that larger $\alpha > 0$ implies a worse bound, and so do the bounds for other Reynolds stress errors. Thus, there is the game of whether we can find a non-empty interval for $\alpha$ such that the addition of the new Reynolds stress error $R_{\text{Com1}}$ still leads to the closure of the necessary estimates; in fact, we will also see another error $R_{\text{Com2}}$ (see \eqref{est 191e}). In this regard, we will consider $z_{q}$ (see \eqref{est 150}) rather than $z$ at every step which is the only way to guarantee the frequency support of $\mathring{R}_{q}$ to be contained inside $B(0, 4 \lambda_{q})$ and control the errors; otherwise, we get $l^{-1} = \lambda_{q+1}^{\alpha}$ instead of $\lambda_{q}$ for any spatial derivative on $\mathring{R}_{l}$ and we will not be able to close our estimates as $\alpha > 1$ from \eqref{est 154c}. As we are going to prove non-uniqueness by constructing a solution with opposite $\dot{H}^{\frac{1}{2}}(\mathbb{T}^{2})$-inequality, in contrast to the deterministic case in \cite{BSV19} (recall the difficulty \ref{Difficulty D}), we also need to come up with an appropriate solution to \eqref{est 6} at step $q = 0$ rather than just $(v_{0}, \mathring{R}_{0}) \equiv (0,0)$ in \cite{BSV19}. We also change the definition of $\lambda_{q}$ and $\delta_{q}$ from \eqref{est 3} to 
\begin{equation}\label{est 35}
\lambda_{q} \triangleq a^{b^{q}} \text{ for } a \in 5 \mathbb{N}, b \in \mathbb{N} \text{ sufficiently large and } \delta_{q} \triangleq \lambda_{q}^{-2\beta}
\end{equation}  
because the definition in \eqref{est 3} does not distinguish between $\lambda_{0}$ and $\lambda_{1}$ which makes the necessary computations at the step $q= 0$ very difficult. Finally, we will replace the convections by the velocity field $u_{q}$ within the transport equations \eqref{est 10}-\eqref{est 11} to the convections by a sum of $u$ and a mollified Ornstein-Uhlenbeck process (see \eqref{est 193a} and \eqref{est 195a}); this allows us to better handle the Reynolds stress transport error, as the author discovered in \cite{Y21c} (see also \cite{RS21}). 
\end{enumerate} 

In Section \ref{Preliminaries}, we give a minimum amount notations and previous results on convex integration, leaving the rest to the Appendix. In Section \ref{Section 4} we prove Theorems \ref{Theorem 2.1}-\ref{Theorem 2.2}. In fact, we prove the main steps of convex integration, specifically Propositions \ref{Proposition 4.7}-\ref{Proposition 4.10}, for the generalized QG equations with the velocity field $u = \Lambda^{1-\gamma_{2}} \mathcal{R}^{\bot} \theta$ according to \eqref{est 1} or equivalently the momentum equation \eqref{est 30a} with 
\begin{equation}\label{est 133}
u = \Lambda^{2-\gamma_{2}} v. 
\end{equation} 
This can be derived following the same reasoning in \cite[Sections 1.4 and A.1]{BSV19} (see also \cite[Section 2.2]{R95}); i.e., the self-adjoint positive-definite operator that is used to define the metric on the Lie algebra associated to the group of volume-preserving diffeomorphisms would be $A \triangleq \Lambda^{\gamma_{2} -2}$ so that the cases $\gamma_{2} = 1$ and $\gamma_{2} = 2$ give the QG and the Euler equations, respectively. The corresponding solutions emanating from initial data $v^{\text{in}} \in \dot{H}^{1- \frac{\gamma_{2}}{2}} (\mathbb{T}^{2})$ constructed via Galerkin approximation would satisfy 
\begin{equation}\label{est 39}
\mathbb{E}^{\textbf{P}} [\lVert v(t) \rVert_{\dot{H}_{x}^{1- \frac{\gamma_{2}}{2}}}^{2}] \leq \lVert v^{\text{in}} \rVert_{\dot{H}_{x}^{1- \frac{\gamma_{2}}{2}}}^{2} + t \Tr ((-\Delta)^{1 - \frac{\gamma_{2}}{2}} GG^{\ast}), 
\end{equation} 
(cf. \eqref{est 31}). The proof of Proposition \ref{Proposition 4.1} concerning the existence of martingale solutions via a Galerkin approximation relies on Proposition \ref{Proposition 6.3} of Calder$\acute{\mathrm{o}}$n commutator which we can also generalize (see Proposition \ref{Proposition 6.3} (2)). Unfortunately, our proof in the case of the generalized QG equations ultimately break down at the last step of the proof of Theorem \ref{Theorem 2.1} upon relying on the Calder$\acute{\mathrm{o}}$n commutator estimate. Especially due to \eqref{est 151b}, we need to reduce $\beta > 1 - \frac{\gamma_{2}}{2}$ sufficiently close to $1- \frac{\gamma_{2}}{2}$ and this makes $\iota \in (0, \frac{1}{2} - \frac{\gamma_{1}}{3})$ quite small. If the size of $\iota > 0$ was nontrivial, then we can deduce the existence of the solution to the generalized QG equations constructed via convex integration by relying on Proposition \ref{Proposition 6.3} (2). Regardless, we present in such a general way in hope that these computations may be of use in future works on the generalized QG equations.

\begin{remark}
Significant progress on this manuscript was made in early May 2022. Then, on 05/26/2022 Hofmanov$\acute{\mathrm{a}}$, Zhu, and Zhu \cite{HZZ22} posted a paper on ArXiV in which, among several results, they elegantly proved the non-uniqueness of the QG equations forced by random noise. Our work was completed entirely independently of \cite{HZZ22}, as can be seen clearly from the difference in results and proofs. While \cite{HZZ22} works directly on the equation of the temperature $\theta$, based on the convex integration technique of \cite{CKL20} on the steady-state weak solution to the deterministic QG equations with the probabilistic approach from \cite{HZZ21a}, we work on the equation of the momentum $v$, based on the convex integration technique of \cite{BSV19} with the probabilistic strategy from \cite{HZZ19} and some key ideas from \cite{Y21c}. Furthermore, the random noise tackled is white in space in \cite{HZZ22} while white in time in this manuscript. 
\end{remark}

\section{Preliminaries}\label{Preliminaries}
\subsection{Notations and assumptions}\label{Subsection 3.1}
We write $C_{0}^{\infty} \triangleq \{f \in C^{\infty} (\mathbb{T}^{2}): \fint_{\mathbb{T}^{2}} f dx = 0 \}$,  $C_{0,\sigma}^{\infty} \triangleq \{f \in C_{0}^{\infty} (\mathbb{T}^{2}): \nabla\cdot f = 0 \}$, and $\dot{H}_{\sigma}^{s} \triangleq \{f \in \dot{H}^{s}(\mathbb{T}^{2}): \fint_{\mathbb{T}^{2}} f dx = 0, \nabla\cdot f = 0 \}$ where $\fint_{\mathbb{T}^{2}} f dx \triangleq \lvert \mathbb{T}^{2} \rvert^{-1} \int_{\mathbb{T}^{2}} f(x) dx$, and $H_{\sigma}^{s}$ similarly for any $s \in \mathbb{R}$. We define $\mathbb{P} \triangleq \Id + \mathcal{R} \otimes \mathcal{R}$ to be the Leray projection onto the space of divergence-free vector fields. We define $\lVert \cdot \rVert_{C_{t,x}^{N}} \triangleq \sum_{0 \leq k + \lvert \alpha \rvert \leq N} \lVert \partial_{t}^{k} D^{\alpha} \cdot \rVert_{L_{t,x}^{\infty}}$ where $k \in \mathbb{N}_{0}$ and $\alpha$ is a multi-index. For any Polish space $H$, we write $\mathcal{B}(H)$ to denote the $\sigma$-algebra of Borel sets in $H$. We denote by $L^{2}(\mathbb{T}^{2})$-inner products by $\langle \cdot, \cdot \rangle$, a duality pairing of $\dot{H}^{-\frac{1}{2}}(\mathbb{T}^{2})-\dot{H}^{\frac{1}{2}}(\mathbb{T}^{2})$ by $\langle \cdot, \cdot \rangle_{\dot{H}_{x}^{-\frac{1}{2}}-\dot{H}_{x}^{\frac{1}{2}}}$, a quadratic variation of $A$ and $B$ by $\langle\langle A, B \rangle \rangle$, as well as $\langle \langle A \rangle \rangle \triangleq \langle \langle A, A \rangle \rangle$. We let  
\begin{equation}\label{est 87}
\Omega_{t} \triangleq C( [t,\infty); (H_{\sigma}^{4})^{\ast} ) \cap L_{\text{loc}}^{\infty} ([t,\infty); \dot{H}_{\sigma}^{1 - \frac{\gamma_{2}}{2}}), \hspace{3mm} t \geq 0 
\end{equation} 
where $(H_{\sigma}^{4})^{\ast}$ denotes the dual of $H_{\sigma}^{4}$. We also denote by $\mathcal{P} (\Omega_{0})$ the set of all probability measures on $(\Omega, \mathcal{B})$ where $\mathcal{B}$ is the Borel $\sigma$-field of $\Omega_{0}$ from the topology of locally uniform convergence on $\Omega_{0}$. We define the canonical process $\xi: \Omega_{0} \mapsto (H_{\sigma}^{4})^{\ast}$ by $\xi_{t}(\omega) \triangleq \omega(t)$. For general $t \geq 0$ we equip $\Omega_{t}$ with Borel $\sigma$-field $\mathcal{B}^{t} \triangleq \sigma \{\xi(s): s \geq t \}$, and additionally define 
\begin{equation}\label{est 413}
\mathcal{B}_{t}^{0} \triangleq \sigma \{\xi(s): s \leq t \} \text{ and } \mathcal{B}_{t} \triangleq \cap_{s > t} \mathcal{B}_{s}^{0}, \hspace{3mm} t \geq 0.
\end{equation} 
For any Hilbert space $U$, we denote by $L_{2}(U, \dot{H}_{\sigma}^{s})$ with $s \in \mathbb{R}_{+}$ the space of all Hilbert-Schmidt operators from $U$ to $\dot{H}_{\sigma}^{s}$ with norm $\lVert \cdot \rVert_{L_{2}(U, \dot{H}_{\sigma}^{s})}$. We impose on $G: \dot{H}_{\sigma}^{1- \frac{\gamma_{2}}{2}} \mapsto L_{2} (U, \dot{H}_{\sigma}^{1 - \frac{\gamma_{2}}{2}})$ to be $\mathcal{B} (\dot{H}_{\sigma}^{1- \frac{\gamma_{2}}{2}})/\mathcal{B} (L_{2} (U, \dot{H}_{\sigma}^{1- \frac{\gamma_{2}}{2}}))$-measurable and satisfy 
\begin{equation}\label{est 60} 
\lVert G(\phi ) \rVert_{L_{2}(U, \dot{H}_{\sigma}^{1- \frac{\gamma_{2}}{2}})} \leq C(1+ \lVert \phi \rVert_{\dot{H}_{x}^{1- \frac{\gamma_{2}}{2}}}), \hspace{3mm} \lim_{j\to\infty} \lVert G(\psi^{j})^{\ast} \phi - G(\psi)^{\ast} \phi \rVert_{U} = 0 
\end{equation} 
for all $\phi, \psi^{j}, \psi \in C^{\infty} (\mathbb{T}^{2}) \cap \dot{H}_{\sigma}^{1- \frac{\gamma_{2}}{2}}$ such that $\lim_{j\to\infty} \lVert \psi^{j} - \psi \rVert_{\dot{H}_{x}^{1- \frac{\gamma_{2}}{2}}} = 0$. 
 
\subsection{Convex integration}\label{Subsection 3.2}
Most of the following preliminaries come from \cite{BSV19}.
\begin{lemma}\label{Inverse-divergence lemma} 
\rm{(Inverse divergence)} Let $f$ be divergence-free and mean-zero. Then we define the $(i,j)$-th entry of $\mathcal{B}f$ for $i, j \in \{1,2\}$ by 
\begin{equation}
(\mathcal{B}f)^{ij} \triangleq - \partial_{j} \Lambda^{-2} f_{i} - \partial_{i}\Lambda^{-2} f_{j}.
\end{equation} 
For $f$ that is not divergence-free, we define $\mathcal{B} f \triangleq \mathcal{B} \mathbb{P} f$; for $f$ that is not mean-zero, we define $\mathcal{B} f \triangleq \mathcal{B} \left( f - \fint_{\mathbb{T}^{2}} f dx \right)$. It follows that $\divergence (\mathcal{B} f) = \mathbb{P} f$, and $\mathcal{B} f$ is a symmetric and trace-free matrix. One property that we will frequently rely on in estimates is $\lVert \mathcal{B} \rVert_{C_{x} \mapsto C_{x}} \lesssim 1$ (e.g., \cite[p. 127]{BV19a} and \cite[Lemma 7.3]{LQ20}). 
\end{lemma}
Next, for $k \in \mathbb{S}^{1}$ we define 
\begin{equation}\label{est 178} 
b_{k}(\xi) \triangleq i k^{\bot} e^{ik\cdot \xi} \hspace{1mm} \text{ and } \hspace{1mm} c_{k}(\xi) \triangleq e^{ik\cdot \xi}
\end{equation} 
which satisfies 
\begin{equation}\label{est 179} 
b_{k} = \nabla_{\xi}^{\bot} c_{k}, \hspace{3mm} c_{k} = - \nabla_{\xi}^{\bot} \cdot b_{k}, \hspace{3mm} \Lambda_{\xi}^{\alpha} b_{k}(\xi) = b_{k}(\xi) \hspace{1mm} \forall \hspace{1mm} \alpha \geq 0, \hspace{3mm} b_{k}(\xi)^{\bot} = - \nabla_{\xi} c_{k}(\xi).
\end{equation}  
For any finite family of vectors $\Gamma \subset \mathbb{S}^{1}$ and $\{ a_{k}  \} \subset \mathbb{C}$ such that $a_{-k} = \bar{a}_{k}$, we set $W(\xi) \triangleq \sum_{k \in \Gamma} a_{k} b_{k} (\xi)$ and $W_{k}(\xi) \triangleq a_{k} b_{k}(\xi)$ so that  
\begin{equation}
\divergence (W \otimes W) = \frac{1}{2} \nabla_{\xi} \lvert W \rvert^{2} + ( \nabla_{\xi}^{\bot} \cdot W) W^{\bot} \hspace{2mm}\text{ and } \hspace{2mm}  \sum_{k \in \Gamma} W_{k} \otimes W_{-k} = \sum_{k\in \Gamma} \lvert a_{k} \rvert^{2} k^{\bot} \otimes k^{\bot}. 
\end{equation} 

\begin{lemma}\label{Geometric Lemma} 
\rm{(Geometric lemma from \cite[Lemma 4.2]{BSV19})} Let $B(\Id, \epsilon)$ denote the ball of symmetric $2\times 2$ matrices, centered at $\Id$ of radius $\epsilon > 0$. Then there exists $\epsilon_{\gamma} > 0$ with which there exist disjoint finite subsets $\Gamma_{j} \subset \mathbb{S}^{1}$ for $j \in \{1,2\}$ and smooth positive functions 
\begin{equation}
\gamma_{k} \in C^{\infty} (B(\Id, \epsilon_{\gamma})), \hspace{3mm} j \in \{1,2\}, k \in \Gamma_{j},
\end{equation} 
such that 
\begin{enumerate}
\item $5 \Gamma_{j} \subset \mathbb{Z}^{2}$ for each $j$, 
\item if $k \in \Gamma_{j}$, then $- k \in \Gamma_{j}$ and $\gamma_{k} =  \gamma_{-k}$, 
\item 
\begin{equation}\label{est 265}
R = \frac{1}{2} \sum_{k \in \Gamma_{j}} (\gamma_{k}(R))^{2} (k^{\bot} \otimes k^{\bot})  \hspace{3mm} \forall \hspace{1mm} R \in B(\Id, \epsilon_{\gamma}),  
\end{equation} 
\item $\lvert k + k' \rvert \geq \frac{1}{2}$ for all $k, k' \in \Gamma_{j}$ such that $k + k' \neq 0$. 
\end{enumerate} 
\end{lemma}
Next, we recall that $\mathbb{P}$ denotes the Leray projection and define for $k \in \mathbb{S}^{1}$, 
\begin{equation}\label{est 192} 
\mathbb{P}_{q+1, k} \triangleq \mathbb{P} P_{\approx k \lambda_{q+1}}
\end{equation} 
where $P_{\approx k\lambda_{q+1}}$ is a Fourier operator with a Fourier symbol $\hat{K}_{\approx k \lambda_{q+1}}(\xi) = \hat{K}_{\approx 1} \left( \frac{ \xi}{\lambda_{q+1}} - k \right)$; i.e., 
\begin{equation*}
\widehat{ P_{\approx \lambda_{q+1} k} f} (\xi) = \hat{K}_{\approx k\lambda_{q+1}}(\xi) \hat{f}(\xi) = \hat{K}_{\approx 1} \left( \frac{\xi}{\lambda_{q+1}} - k \right) \hat{f}(\xi) 
\end{equation*}
where $\hat{K}_{\approx 1}$ is a smooth bump function such that $\supp \hat{K}_{\approx 1} \subset \{\xi:  \lvert \xi \rvert \leq \frac{1}{8} \}$ and $\hat{K}_{\approx 1} \rvert_{\{\xi:  \lvert \xi \rvert \leq \frac{1}{16} \}} \equiv 1$. We see that within the support of $\hat{K}_{\approx 1} \left( \frac{\cdot}{\lambda_{q+1}} - k \right)$, $\frac{7}{8} \lambda_{q+1} \leq \lvert \xi \rvert \leq \frac{9}{8} \lambda_{q+1}$. This implies that $\supp \widehat{ \mathbb{P}_{q+1, k} f} \subset \{\xi: \frac{7}{8} \lambda_{q+1} \leq \lvert \xi \rvert \leq \frac{9}{8} \lambda_{q+} \}$ and for any $0 \leq a, b$, for some suitable constant $C_{a,b}$ that is independent of $\lambda_{q+1}$, 
\begin{align*}
\sup_{\xi \in\mathbb{R}^{2}} \lvert \xi \rvert^{a} \left\lvert   \nabla_{\xi} ^{b} \hat{K}_{\approx k \lambda_{q+1} } \right\rvert \leq C_{a,b} \lambda_{q+1}^{a-b},
\end{align*}
and similarly 
\begin{align*}
\lVert \lvert x \rvert^{b} \nabla_{x}^{a} K_{\approx k \lambda_{q+1}} \rVert_{L^{1}(\mathbb{R}^{2})} \leq C_{a,b} \lambda_{q+1}^{a-b} \hspace{3mm} \forall \hspace{1mm} a, b \in [0,2]. 
\end{align*}
It follows that for all $f$ that is $\mathbb{T}^{2}$-periodic, we can write $\mathbb{P}_{q+1, k} f(x) = \int_{\mathbb{R}^{2}} K_{q+1, k} (y) f(x-y) dy$ where the kernel $K_{q+1, k}$ satisfies for all $a, b \geq 0$, 
\begin{equation}\label{est 416}  
\lVert \lvert x \rvert^{b} \nabla_{x}^{a} K_{q+1, k} (x) \rVert_{L^{1}(\mathbb{R}^{2})}  \leq C_{a,b} \lambda_{q+1}^{a-b},
\end{equation}
which allows one to deduce 
\begin{equation}\label{est 202}
\lVert \mathbb{P}_{q+1, k} \rVert_{C_{x} \mapsto C_{x}} \leq C_{1}. 
\end{equation}  
Finally, we define $\tilde{P}_{\approx \lambda_{q+1}}$ to be the Fourier operator with a symbol supported in $\{\xi: \frac{\lambda_{q+1}}{4} \leq \lvert \xi \rvert \leq 4 \lambda_{q+1} \}$ and is identically one on $\{\xi: \frac{3 \lambda_{q+1}}{8} \leq \lvert \xi \rvert \leq 3 \lambda_{q+1}\}$. As the last item in this Preliminaries section, following \cite[Section 5]{BV19b} (also \cite{RS21, Y21c}), we split $[0, T_{L}]$ into finitely many subintervals of size $\tau_{q+1}$ and let $0 \leq \chi \leq 1$ be a smooth cut-off function with support in $(-1, 1)$ such that $\chi \equiv 1$ on $(-\frac{1}{4}, \frac{1}{4})$, and for $l = \lambda_{q+1}^{-\alpha}$ from \eqref{est 32} and  
\begin{equation}\label{est 210}
\tau_{q+1}^{-1} \triangleq l^{-\frac{1}{2}} \lambda_{q+1}^{\frac{ 3- \gamma_{2}}{2}} \delta_{q+1}^{\frac{1}{4}}
\end{equation} 
which is different from \cite[equation (4.7)]{BSV19}, 
\begin{equation}\label{est 209} 
\chi_{j}(t) \triangleq \chi(\tau_{q+1}^{-1} t - j) \hspace{3mm} \text{ for } j \in \{0, 1, \hdots, \lceil \tau_{q+1}^{-1}T_{L} \rceil\}, 
\end{equation}
where we suppressed the dependence of $\chi_{j}$ on $q$ following \cite[p. 1828]{BSV19}, satisfies 
\begin{equation}\label{est 211}
\sum_{j=0}^{\lceil \tau_{q+1}^{-1} T_{L} \rceil} \chi_{j}^{2}(t)= 1 \hspace{3mm} \forall \hspace{1mm} t \in [0, T_{L}]. 
\end{equation} 
By comparing the definitions of $\chi_{j}$ in \eqref{est 209} and previous works on probabilistic convex integration such as ``$\chi_{j}(t) = \chi(l^{-1} t - j)$'' in \cite[equation (70)]{Y21c} (see also \cite[equation (5.20)]{BV19b}), we see that the convex integration on the QG equations requires one to carefully select $\tau_{q+1}$ rather than just take $\tau_{q+1} = l$, basically optimize over the errors of $R_{T}$ in \eqref{est 301} and $R_{O, \text{approx}}$ in \eqref{est 266}. On many occasions when no confusion arises, we will write for brevity 
\begin{equation}\label{est 401}
\sum_{j} \triangleq \sum_{j=0}^{\lceil \tau_{q+1}^{-1} T_{L} \rceil }, \hspace{1mm} \sum_{k} \triangleq \sum_{k\in\Gamma_{j}} \text{ for fixed } j \in \{0, \hdots, \lceil \tau_{q+1}^{-1} T_{L} \rceil \}, \hspace{1mm} \text{ and } \sum_{j,k} \triangleq \sum_{j=0}^{\lceil \tau_{q+1}^{-1} T_{L} \rceil} \sum_{k\in\Gamma_{j}}. 
\end{equation}
 
\section{Proofs of Theorems \ref{Theorem 2.1}-\ref{Theorem 2.2}}\label{Section 4}
We give the general definitions of a solution to the stochastic momentum equation \eqref{est 30}. From here to Proposition \ref{Proposition 4.6}, we only consider the case $\gamma_{2} = 1$ because Proposition \ref{Proposition 4.1} can be proven only for this special case.  
\begin{define}\label{Definition 4.1}
Fix $\iota \in (0,1)$. Let $s \geq 0$ and $\xi^{\text{in}} \in \dot{H}_{\sigma}^{\frac{1}{2}}$. Then $P \in \mathcal{P} (\Omega_{0})$ is a martingale solution to \eqref{est 30} with initial data $\xi^{\text{in}}$ at initial time $s$ if 
\begin{enumerate}
\item [](M1) $P (\{ \xi(t) = \xi^{\text{in}} \hspace{1mm} \forall \hspace{1mm} t \in [0,s] \}) = 1$ and for all $l \in \mathbb{N}$, 
\begin{equation}
P \left( \left\{ \xi\in\Omega_{0}: \int_{0}^{l} \lVert G(\xi(r)) \rVert_{L_{2} (U, \dot{H}_{\sigma}^{\frac{1}{2}})}^{2} dr < \infty \right\} \right) = 1, 
\end{equation} 
\item [](M2) for all $\psi^{k} \in C^{\infty} (\mathbb{T}^{2}) \cap \dot{H}_{\sigma}^{\frac{1}{2}}$ and $t\geq s$, the process 
\begin{align}
M_{t,s}^{k} =& \langle \xi(t) - \xi^{\text{in}}, \psi^{k} \rangle \label{est 54}\\
-& \int_{s}^{t} \sum_{i,j=1}^{2} \langle \Lambda \xi_{i}, \partial_{i} \psi_{j}^{k} \xi_{j}\rangle_{\dot{H}_{x}^{-\frac{1}{2}} - \dot{H}_{x}^{\frac{1}{2}}} - \frac{1}{2} \langle \partial_{i} \xi_{j}, [ \Lambda, \psi_{i}^{k} ] \xi_{j} \rangle_{\dot{H}_{x}^{-\frac{1}{2}} -\dot{H}_{x}^{\frac{1}{2}}} - \langle \xi, \Lambda^{\gamma_{1}} \psi^{k} \rangle dr \nonumber 
\end{align} 
is a continuous, square-integrable $(\mathcal{B}_{t})_{t\geq s}$-martingale under $P$ such that $\langle \langle M_{t,s}^{k} \rangle \rangle = \int_{s}^{t} \lVert G(\xi(r))^{\ast} \psi^{k} \rVert_{U}^{2} dr$,
\item [](M3) for any $q \in \mathbb{N}$, there exists a function $t \mapsto C_{t,q} \in\mathbb{R}_{+}$ for all $t \geq s$ such that  
\begin{equation}\label{est 53}
\mathbb{E}^{P} \left[ \sup_{r \in [0,t]} \lVert \xi(r) \rVert_{\dot{H}_{x}^{\frac{1}{2}}}^{2q} + \int_{s}^{t} \lVert \xi(r) \rVert_{\dot{H}_{x}^{\frac{1}{2} + \iota}}^{2} dr\right] \leq C_{t,q} (1+ \lVert \xi^{\text{in}} \rVert_{\dot{H}_{x}^{\frac{1}{2}}}^{2q}). 
\end{equation} 
The set of all such martingale solutions with the same constant $C_{t,q}$ in \eqref{est 53} for every $q \in \mathbb{N}$ and $t\geq s$ will be denoted by $\mathcal{C} ( s, \xi^{\text{in}}, \{C_{t,q} \}_{q\in\mathbb{N}, t \geq s})$. 
\end{enumerate} 
\end{define} 
In \eqref{est 54} we implicitly relied on the the identity of 
\begin{equation}\label{est 59} 
-\int_{\mathbb{T}^{2}} (\nabla \xi)^{T} \cdot \Lambda \xi \cdot \psi^{k} dx = \sum_{i,j=1}^{2}\frac{1}{2} \int_{\mathbb{T}^{2}} \partial_{i} \xi_{j} [\Lambda, \psi_{i}^{k} ] \xi_{j} dx
\end{equation}
and Calder$\acute{\mathrm{o}}$n's commutator estimate from \eqref{est 50} of Proposition \ref{Proposition 6.3}. If $\{\psi^{k}\}_{k=1}^{\infty}$ is a complete orthonormal basis of $\dot{H}_{\sigma}^{\frac{1}{2}}$ that consists of eigenvectors of $GG^{\ast}$, then $M_{t,s} \triangleq \sum_{k=1}^{\infty} M_{t,s}^{k} \psi^{k}$ becomes a $GG^{\ast}$-Wiener process starting from initial time $s$ w.r.t. $(\mathcal{B}_{t})_{t\geq s}$ under $P$. 

\begin{define}\label{Definition 4.2}
Fix $\iota \in (0,1)$. Let $s \geq 0$, $\xi^{\text{in}} \in \dot{H}_{\sigma}^{\frac{1}{2}}$, and $\tau: \Omega_{0} \mapsto [s, \infty]$ be a stopping time of $(\mathcal{B}_{t})_{t\geq s}$. Define the space of trajectories stopped at $\tau$ by 
\begin{equation}\label{est 92} 
\Omega_{0,\tau} \triangleq \{\omega( \cdot \wedge \tau(\omega)): \omega \in \Omega_{0}  \} = \{ \omega \in \Omega_{0}: \xi(t,\omega) = \xi(t\wedge \tau(\omega), \omega) \hspace{1mm} \forall \hspace{1mm} t \geq 0 \}. 
\end{equation} 
Then $P \in \mathcal{P} (\Omega_{0,\tau})$ is a martingale solution to \eqref{est 30} on $[s,\tau]$ with initial data $\xi^{\text{in}}$ at initial time $s$ if 
\begin{enumerate}
\item [](M1) $P (\{ \xi(t) = \xi^{\text{in}} \hspace{1mm} \forall \hspace{1mm} t \in [0,s] \}) = 1$ and for all $l \in \mathbb{N}$, 
\begin{equation}
P \left( \left\{ \xi\in\Omega_{0}: \int_{0}^{l\wedge \tau} \lVert G(\xi(r)) \rVert_{L_{2} (U, \dot{H}_{\sigma}^{\frac{1}{2}})}^{2} dr < \infty \right\} \right) = 1, 
\end{equation} 
\item [](M2) for all $\psi^{k} \in C^{\infty} (\mathbb{T}^{2}) \cap \dot{H}_{\sigma}^{\frac{1}{2}}$ and $t\geq s$, the process 
\begin{align*}
M_{t\wedge \tau,s}^{k} =& \langle \xi(t\wedge \tau) - \xi^{\text{in}}, \psi^{k} \rangle\\
-& \int_{s}^{t\wedge \tau} \sum_{i,j=1}^{2} \langle \Lambda \xi_{i}, \partial_{i} \psi_{j}^{k} \xi_{j}\rangle_{\dot{H}_{x}^{-\frac{1}{2}} - \dot{H}_{x}^{\frac{1}{2}}} - \frac{1}{2} \langle \partial_{i} \xi_{j}, [ \Lambda, \psi_{i}^{k} ] \xi_{j} \rangle_{\dot{H}_{x}^{-\frac{1}{2}} -\dot{H}_{x}^{\frac{1}{2}}} - \langle \xi, \Lambda^{\gamma_{1}} \psi^{k} \rangle dr 
\end{align*} 
is a continuous, square-integrable $(\mathcal{B}_{t})_{t\geq s}$-martingale under $P$ such that $\langle \langle M_{t \wedge \tau ,s}^{k} \rangle \rangle = \int_{s}^{t\wedge \tau} \lVert G(\xi(r))^{\ast} \psi^{k} \rVert_{U}^{2} dr$, 
\item [](M3) for any $q \in \mathbb{N}$, there exists a function $t \mapsto C_{t,q} \in\mathbb{R}_{+}$ for all $t \geq s$ such that 
\begin{equation}
\mathbb{E}^{P} \left[ \sup_{r \in [0,t\wedge \tau]} \lVert \xi(r) \rVert_{\dot{H}_{x}^{\frac{1}{2}}}^{2q} + \int_{s}^{t\wedge \tau} \lVert \xi(r) \rVert_{\dot{H}_{x}^{\frac{1}{2} + \iota}}^{2} dr \right] \leq C_{t,q} (1+ \lVert \xi^{\text{in}} \rVert_{\dot{H}_{x}^{\frac{1}{2}}}^{2q}). 
\end{equation} 
\end{enumerate} 
\end{define} 

The following result is concerned with the existence and stability of martingale solutions to \eqref{est 30} in the case of an additive noise. Because the QG momentum equations forced by random noise has never been studied before, we give details of its proof in the Appendix.  
\begin{proposition}\label{Proposition 4.1}
\begin{enumerate} 
\item For every $(s, \xi^{\text{in}}) \in [0,\infty) \times \dot{H}_{\sigma}^{\frac{1}{2}}$, there exists a martingale solution $P \in \mathcal{P} (\Omega_{0})$  to \eqref{est 30} with initial data $\xi^{\text{in}}$ at initial time $s$ according to Definition \ref{Definition 4.1}. 
\item Moreover, if there exists a family $\{(s_{l}, \xi_{l}) \}_{l\in\mathbb{N}} \subset [0,\infty) \times \dot{H}_{\sigma}^{\frac{1}{2}}$ such that $\lim_{l\to\infty} \lVert (s_{l},\xi_{l}) - (s, \xi^{\text{in}}) \rVert_{\mathbb{R} \times \dot{H}_{x}^{\frac{1}{2}}} = 0$ and $P_{l} \in \mathcal{C} ( s_{l}, \xi_{l}, \{ C_{t,q} \}_{q\in\mathbb{N}, t \geq s_{l}})$ is the martingale solution corresponding to $(s_{l}, \xi_{l})$, then there exists a subsequence $\{P_{l_{k}} \}_{k\in\mathbb{N}}$ 
and $P \in \mathcal{C} ( s, \xi^{\text{in}}$, $\{C_{t,q} \}_{q\in\mathbb{N}, t \geq s})$ such that $P_{l_{k}}$ converges weakly to $P$.
\end{enumerate}
\end{proposition}
\noindent Proposition \ref{Proposition 4.1} leads to the following two results, which are only slight modifications of \cite[Propositions 3.2 and 3.4]{HZZ19} to which we refer interested readers for details. 

\begin{lemma}\label{Lemma 4.2}
\rm{(cf. \cite[Proposition 3.2]{HZZ19})} Let $\tau$ be a bounded stopping time of $(\mathcal{B}_{t})_{t\geq 0}$. Then, for every $\omega \in \Omega_{0}$, there exists $Q_{\omega} \triangleq \delta_{\omega} \otimes_{\tau(\omega)} R_{\tau(\omega), \xi(\tau(\omega), \omega)} \in \mathcal{P} (\Omega_{0})$ where $\delta_{\omega}$ is a point-mass at $\omega$ such that for $\omega \in \{\xi(\tau) \in \dot{H}_{\sigma}^{\frac{1}{2}} \}$ 
\begin{subequations}
\begin{align} 
& Q_{\omega} ( \{ \omega' \in \Omega_{0}:\hspace{1mm}  \xi(t, \omega') = \omega(t) \hspace{1mm} \forall \hspace{1mm} t \in [0, \tau(\omega) ] \}) = 1,  \label{est 88}\\
& Q_{\omega}(A) = R_{\tau (\omega), \xi(\tau(\omega), \omega)} (A) \hspace{1mm} \forall \hspace{1mm} A \in \mathcal{B}^{\tau(\omega)}, \label{est 89}  
\end{align}
\end{subequations} 
where $R_{\tau(\omega), \xi(\tau(\omega), \omega)} \in \mathcal{P}(\Omega_{0})$ is a martingale solution to \eqref{est 30} with initial data $\xi(\tau(\omega), \omega)$ at initial time $\tau(\omega)$, and the mapping $\omega \mapsto Q_{\omega}(B)$ is $\mathcal{B}_{\tau}$-measurable for every $B \in \mathcal{B}$.
\end{lemma} 

\begin{lemma}\label{Lemma 4.3}
\rm{(cf. \cite[Proposition 3.4]{HZZ19})} Let $\tau$ be a bounded stopping time of $(\mathcal{B}_{t})_{t\geq 0}$, $\xi^{\text{in}} \in \dot{H}_{\sigma}^{\frac{1}{2}}$, and $P$ be a martingale solution to \eqref{est 30} on $[0,\tau]$ with initial data $\xi^{\text{in}}$ at initial time 0 that satisfies Definition \ref{Definition 4.2}. Suppose that there exists a Borel set $\mathcal{N} \subset \Omega_{0,\tau}$ such that $P(\mathcal{N}) = 0$ and $Q_{\omega}$ from Lemma \ref{Lemma 4.2} satisfies for every $\omega \in \Omega_{0} \setminus \mathcal{N}$ 
\begin{equation}\label{est 94}
Q_{\omega} (\{\omega' \in \Omega_{0}:\hspace{1mm}  \tau(\omega') = \tau(\omega) \}) = 1. 
\end{equation} 
Then the probability measure $P \otimes_{\tau}R \in \mathcal{P}(\Omega_{0})$ defined by 
\begin{equation}\label{est 93}
P\otimes_{\tau} R (\cdot) \triangleq \int_{\Omega_{0}} Q_{\omega} (\cdot) P(d\omega) 
\end{equation} 
satisfies $P \otimes_{\tau}R  = P$ on the $\sigma$-algebra $\sigma \{\xi(t\wedge \tau), t \geq 0 \}$ and it is a martingale solution to \eqref{est 30} on $[0,\infty)$ with initial data $\xi^{\text{in}}$ at initial time 0. 
\end{lemma} 

Now we fix a $GG^{\ast}$-Wiener process $B$ on $(\Omega, \mathcal{F}, \textbf{P})$ with $(\mathcal{F}_{t})_{t\geq 0}$ as the canonical filtration of $B$ augmented by all the $\textbf{P}$-negligible sets. We let $\mathcal{B}_{\tau}$ represent the $\sigma$-algebra associated to the stopping time $\tau$ and consider 
\begin{subequations}\label{est 152}
\begin{align}
& dz + [\nabla p_{1} +\Lambda^{\gamma_{1}} z] dt = dB,  \hspace{3mm} \nabla\cdot z = 0, \\
& z(0,x) \equiv 0, \label{est 152b}
\end{align}
\end{subequations}
so that $y \triangleq v - z$, together with $p_{2} \triangleq p - p_{1}$ satisfies
\begin{equation}\label{est 113}
\partial_{t} y + (\Lambda (y+z) \cdot \nabla) (y+z) - (\nabla (y+z))^{T} \cdot \Lambda (y+z) + \nabla p_{2} + \Lambda^{\gamma_{1}} y = 0
\end{equation} 
due to \eqref{est 30}. We see that 
\begin{equation}\label{est 114} 
z(t) = \int_{0}^{t} e^{-(t-r) (-\Delta)^{\frac{\gamma_{1}}{2}}} \mathbb{P} d B(r) 
\end{equation} 
where $e^{- r(-\Delta)^{\frac{\gamma_{1}}{2}}}$ is a semigroup generated by $- (-\Delta)^{\frac{\gamma_{1}}{2}}$. The following proposition informs us the regularity of $z$ from the hypothesis of Theorem \ref{Theorem 2.1}. 

\begin{proposition}\label{Proposition 4.4}
Suppose that $\gamma_{1} \in (0, \frac{3}{2})$ and $B$ is a $GG^{\ast}$-Wiener process that satisfies \eqref{est 161} for some $\sigma > 0$. Then, for all $T > 0$, $\delta \in (0, \frac{1}{2})$, and $l \in \mathbb{N}$,  
\begin{equation}\label{est 115}
\mathbb{E}^{\textbf{P}} \left[ \lVert z \rVert_{C_{T} \dot{H}_{x}^{4 + \frac{\sigma}{2}}}^{2l} + \lVert z \rVert_{C_{T}^{\frac{1}{2} - \delta} \dot{H}_{x}^{4 - \frac{\gamma_{1}}{2} + \frac{\sigma}{2} }}^{2l} \right] < \infty. 
\end{equation} 
\end{proposition}
\noindent Proposition \ref{Proposition 4.4} is a straight-forward extension of previous works such as \cite[Proposition 3.6]{HZZ19} (also \cite[Proposition 4.4]{Y20a}); thus, we sketch its proof in the Appendix for completeness. Next, we recall $\Omega_{0}$ from \eqref{est 87} and define for every $\omega \in \Omega_{0}$ 
\begin{subequations}\label{est 124} 
\begin{align}
&M_{t,0}^{\omega} \triangleq \omega(t) - \omega(0) + \int_{0}^{t} \mathbb{P} \divergence (\Lambda \omega (r) \otimes \omega (r)) - \mathbb{P} ((\nabla \omega)^{T} (r) \cdot \Lambda \omega (r)) + \Lambda^{\gamma_{1}}\omega(r) dr,\label{est 122} \\
& Z^{\omega}(t) \triangleq M_{t,0}^{\omega} - \int_{0}^{t} \mathbb{P} \Lambda^{\gamma_{1}} e^{- (t-r) \Lambda^{\gamma_{1}}} M_{r,0}^{\omega} dr.\label{est 123} 
\end{align}
\end{subequations} 
If $P$ is a martingale solution to \eqref{est 30}, then the mapping $\omega \mapsto M_{t,0}^{\omega}$ is a $GG^{\ast}$-Wiener process under $P$ and we can show using \eqref{est 124} and \eqref{est 87} that 
\begin{equation}\label{est 130}
Z(t) = \int_{0}^{t} e^{- (t-r)\Lambda^{\gamma_{1}}} \mathbb{P} dM_{r,0}
\end{equation} 
which implies due to Proposition \ref{Proposition 4.4} that for any $\delta \in (0, \frac{1}{2})$ and $T > 0$, 
\begin{equation}\label{est 125}
z \in C_{T} \dot{H}_{x}^{4+ \frac{\sigma}{2}} \cap C_{T}^{\frac{1}{2} - \delta} \dot{H}_{x}^{4 - \frac{\gamma_{1}}{2} + \frac{\sigma}{2} } \hspace{3mm} \textbf{P}\text{-a.s.} 
\end{equation} 
Next, we define for $\delta \in (0, \frac{1}{4})$ and the Sobolev constant $C_{S_{1}} \geq 0$ such that $\lVert f \rVert_{L_{x}^{\infty}} \leq C_{S_{1}} \lVert f \rVert_{\dot{H}_{x}^{1+ \frac{\sigma}{2}}}$ for all $f \in \dot{H}^{1+ \frac{\sigma}{2}} (\mathbb{T}^{2})$ that is mean-zero, 
\begin{subequations}
\begin{align}
\tau_{L}^{\lambda}(\omega) \triangleq& \inf \left\{t \geq 0: C_{S_{1}} \lVert Z^{\omega} (t) \rVert_{\dot{H}_{x}^{4+ \frac{\sigma}{2}}} > \left(L - \frac{1}{\lambda} \right)^{\frac{1}{4}} \right\} \nonumber \\
& \wedge \inf \left\{ t \geq 0: C_{S_{1}} \lVert Z^{\omega} \rVert_{C_{t}^{\frac{1}{2} - 2 \delta} \dot{H}_{x}^{4 - \frac{\gamma_{1}}{2} + \frac{\sigma}{2}  }} > \left(L - \frac{1}{\lambda} \right)^{\frac{1}{2}} \right\} \wedge L, \label{est 126a}  \\
\tau_{L}(\omega) \triangleq& \lim_{\lambda \to \infty} \tau_{L}^{\lambda}(\omega) \label{est 126b} 
\end{align}
\end{subequations} 
so that $\tau_{L}$ is a stopping time due to \cite[Lemma 3.5]{HZZ19}. Additionally, for $L > 1$ and $\delta \in (0, \frac{1}{4})$ we define 
\begin{equation}\label{est 128} 
T_{L} \triangleq \inf\{ t \geq 0: C_{S_{1}} \lVert z(t) \rVert_{\dot{H}_{x}^{4+ \frac{\sigma}{2}}} \geq L^{\frac{1}{4}} \} 
 \wedge \inf\{t \geq 0: C_{S_{1}} \lVert z \rVert_{C_{t}^{\frac{1}{2} - 2 \delta} \dot{H}_{x}^{4- \frac{\gamma_{1}}{2} + \frac{\sigma}{2}}} \geq L^{\frac{1}{2}} \} \wedge L; 
\end{equation}
due to Proposition \ref{Proposition 4.4} we know that $T_{L} > 0$ and $\lim_{L\to\infty} T_{L} = \infty$ $\textbf{P}$-a.s. Next, we assume Theorem \ref{Theorem 2.1} on a probability space $(\Omega, \mathcal{F}, (\mathcal{F}_{t})_{t\geq 0}, \textbf{P})$, denote by $P = \mathcal{L}(v)$ the law of $v$ constructed from Theorem \ref{Theorem 2.1}, and obtain the following two results. 

\begin{proposition}\label{Proposition 4.5}
Let $\tau_{L}$ be defined by \eqref{est 126b}. Then $P = \mathcal{L}(v)$ is a martingale solution on $[0,\tau_{L}]$ according to Definition \ref{Definition 4.2}. 
\end{proposition}
\begin{proposition}\label{Proposition 4.6}
Let $\tau_{L}$ be defined by \eqref{est 126b} and $P = \mathcal{L}(v)$. Then $P \otimes_{\tau_{L}} R$ defined by  \eqref{est 93} is a martingale solution on $[0,\infty)$ according to Definition \ref{Definition 4.1}. 
\end{proposition}
\noindent For completeness, we sketch the proof of Propositions \ref{Proposition 4.5}-\ref{Proposition 4.6} in the Appendix referring to \cite[Propositions 3.7-3.8]{HZZ19} respectively for details. 

\begin{proof}[Proof of Theorem \ref{Theorem 2.2} assuming Theorem \ref{Theorem 2.1}]
We fix $T > 0$ arbitrarily, $K > 1$, and $\kappa \in (0,1)$ such that $\kappa K^{2} \geq 1$, rely on Theorem \ref{Theorem 2.1} and Proposition \ref{Proposition 4.6} to deduce the existence of $L > 1$ and a martingale solution $P \otimes_{\tau_{L}} R$ to \eqref{est 30} on $[0, \infty)$ such that $P \otimes_{\tau_{L}} R = P$ on $[0, \tau_{L}]$ where $P =  \mathcal{L} (v)$. Hence, $P \otimes_{\tau_{L}} R$ has a deterministic initial data $v^{\text{in}}$ from Theorem \ref{Theorem 2.1} and satisfies 
\begin{equation}\label{est 147} 
P \otimes_{\tau_{L} } R ( \{ \tau_{L} \geq T \}) \overset{\eqref{est 93} \eqref{est 143} }{=} P(\{\tau_{L} \geq T \})  \overset{\eqref{est 132} \eqref{est 92} }{=} \textbf{P} ( \{ T_{L} \geq T \}) \overset{\eqref{est 148}}{>} \kappa. 
\end{equation} 
This implies 
\begin{equation}\label{est 149} 
\mathbb{E}^{P \otimes_{\tau_{L}} R} [ \lVert \xi(T) \rVert_{\dot{H}_{x}^{\frac{1}{2}}}^{2}] \geq \mathbb{E}^{P \otimes_{\tau_{L}} R}[ 1_{\{ \tau_{L} \geq T \}} \lVert \xi(T) \rVert_{\dot{H}_{x}^{\frac{1}{2}}}^{2} ]  >\kappa K^{2}\left[\lVert v^{\text{in}} \rVert_{\dot{H}_{x}^{\frac{1}{2}}}^{2} + T \Tr ((-\Delta)^{\frac{1}{2}} GG^{\ast}) \right]
\end{equation} 
where the second inequality is due to \eqref{est 146} and \eqref{est 147}. On the other hand, we can construct another martingale solution $\textbf{P}$ starting from same initial data $v^{\text{in}}$ via a Galerkin approximation that satisfies \eqref{est 31}.
\end{proof} 
Next, in addition to $\lambda_{q}$ and $\delta_{q}$ in \eqref{est 35} we define 
\begin{equation}\label{est 135}
M_{0}(t) \triangleq L^{4} e^{4Lt}. 
\end{equation} 
We see from \eqref{est 128} and Sobolev embeddings in $\mathbb{T}^{2}$ that for any $\delta \in (0, \frac{1}{4})$ and $t \in [0, T_{L}]$, 
\begin{equation}\label{est 136}
\lVert z(t) \rVert_{\dot{W}_{x}^{k,\infty}} \leq L^{\frac{1}{4}} \hspace{1mm} \forall \hspace{1mm} k \in \{0, 1,2,3\}, \hspace{1mm} \text{ and }\hspace{1mm}  \lVert z \rVert_{C_{t}^{\frac{1}{2} - 2\delta} \dot{W}_{x}^{3 - \frac{\gamma_{1}}{2}, \infty}} \leq L^{\frac{1}{2}}.
\end{equation}  
Because it is not true in general that $\supp \hat{z} \subset B(0, 2 \lambda_{q})$, via Littlewood-Paley theory we convolute $z$ in space by a smooth function $\tilde{\phi}_{q}$ with frequency support in $B(0, \frac{\lambda_{q}}{4})$, $q \in \mathbb{N}_{0}$, such that $z_{q} \triangleq z \ast_{x} \tilde{\phi}_{q} \to z$ as $q\to \infty$ in distributions so that 
\begin{equation}\label{est 150} 
\supp \hat{z}_{q}  \subset B\left(0,  \frac{\lambda_{q}}{4}\right). 
\end{equation} 
Here we chose $\frac{\lambda_{q}}{4}$ rather than $2\lambda_{q}$ for the convenience of estimating $R_{\text{Com2}}$ in Section \ref{Section 4.1.6}. We will construct solutions $(v_{q}, \mathring{R}_{q})$ for $q \in \mathbb{N}_{0}$ to \eqref{est 113} with an error iteratively. In fact, as we discussed in Subsection \ref{Subsection 2.2}, we shall consider the generalized QG equations with $u = \Lambda^{2-\gamma_{2}}v$ from \eqref{est 133} in the following Propositions \ref{Proposition 4.7}-\ref{Proposition 4.10}, which brings us to
\begin{align} 
&\partial_{t} y_{q} + (\Lambda^{2-\gamma_{2}} (y_{q}+z_{q}) \cdot \nabla) (y_{q}+z_{q}) - (\nabla (y_{q}+z_{q}))^{T} \cdot \Lambda^{2-\gamma_{2}} (y_{q}+z_{q}) + \nabla p_{q} + \Lambda^{\gamma_{1}} y_{q} = \divergence \mathring{R}_{q}, \nonumber \\
&\nabla\cdot y_{q} = 0,\label{est 134} 
\end{align}
where $\mathring{R}_{q}$ is a symmetric trace-free matrix. The range of parameters we shall consider in the following Propositions \ref{Proposition 4.7}-\ref{Proposition 4.10} for the generalized QG equations described by \eqref{est 133} are  
\begin{subequations}\label{est 154} 
\begin{align}
& \gamma_{2} \in [1, 2), \label{est 154a} \\
&\begin{cases}
\frac{1}{2} < \beta < \frac{8}{15} & \text{ if } \gamma_{2} = 1,  \\
1 - \frac{\gamma_{2}}{2} < \beta < 
\min\{ \frac{15 - 7 \gamma_{2}}{15}, \frac{8- 4\gamma_{2}}{7}, \frac{3- \gamma_{2}}{4} \} & \text{ if } \gamma_{2} > 1,
\end{cases}  
\label{est 154b} \\
&
\begin{cases}
-2 + 6 \beta< \alpha < \frac{ 8 - 5 \beta}{4} & \text{ if } \gamma_{2} = 1, \\
1< \alpha <  7- 3 \gamma_{2} - 5 \beta & \text{ if } \gamma_{2} > 1,  
\end{cases} 
\label{est 154c} \\
&\gamma_{1} + 2 \gamma_{2} < 5 - 3\beta.  \label{est 154d} 
\end{align}
\end{subequations}
We emphasize that for all $\gamma_{2} \in [1, 2)$, we have $\alpha > 1$ from \eqref{est 154c}. We also point out that $\beta < \frac{3- \gamma_{2}}{4}$ for $\gamma_{2} > 1$ is needed in \eqref{est 308} and it may seem to lead to a contradiction because $\frac{3-\gamma_{2}}{4} = \frac{1}{2}$ when $\gamma_{2} = 1$ and we have $\frac{1}{2} < \beta$ in \eqref{est 154b} in case $\gamma_{2} = 1$; however, these conditions are not continuous in $\gamma_{2}$ due to our convenient choice of $\alpha > 1$ in \eqref{est 154c} (see e.g., \eqref{est 223} concerning its convenience). Next, we denote specific universal non-negative constants $C_{G_{1}}, C_{G_{2}}$, and $C_{S_{2}}$ due to Gagliardo-Nirenberg inequality and Sobolev's inequality: for any $f$ that is sufficiently smooth in $x \in \mathbb{T}^{2}$ and mean-zero,  
\begin{subequations}\label{est 155}
\begin{align}
&\lVert f \rVert_{C_{x}^{2-\gamma_{2}}} + \lVert \Lambda^{2-\gamma_{2}} f \rVert_{L_{x}^{\infty}} \leq C_{G_{1}} \lVert f\rVert_{L_{x}^{2}}^{\frac{\gamma_{2}}{3}} \lVert \Lambda^{3} f \rVert_{L_{x}^{2}}^{1- \frac{\gamma_{2}}{3}}, \hspace{3mm} \lVert f \rVert_{L_{x}^{\infty}} \leq C_{G_{2}} \lVert f \rVert_{L_{x}^{2}}^{\frac{1}{2}} \lVert \Delta f \rVert_{L_{x}^{2}}^{\frac{1}{2}}, \label{est 155a} \\
&\lVert f \rVert_{L_{x}^{\infty}} \leq C_{S_{2}} \lVert f \rVert_{\dot{H}_{x}^{3-\gamma_{1}}}. \label{est 155b}
\end{align}
\end{subequations} 
For a large universal constant $C_{0} > 0$ from the proof of Proposition \ref{Proposition 4.8}, we fix another universal constant $\bar{C} > 0$ such that 
\begin{equation}\label{est 419}
\frac{2^{\frac{5}{2} - \frac{\gamma_{2}}{2}} C_{0} \sqrt{1+ C_{G_{1}} C_{G_{2}} 8}}{\pi} < \sqrt{\bar{C}}. 
\end{equation} 
Similarly to \eqref{est 33}, we extend $y_{q}$, $\mathring{R}_{q}$, and $z_{q}$  to $t < 0$ by their respective values at $t = 0$ and mollify them to obtain 
\begin{equation}\label{est 180}
y_{l} \triangleq y_{q} \ast_{x} \phi_{l} \ast_{t} \varphi_{l}, \hspace{2mm} \mathring{R}_{l} \triangleq \mathring{R}_{q} \ast_{x} \phi_{l} \ast_{t} \varphi_{l}, \hspace{2mm} \text{ and } \hspace{2mm} z_{l} \triangleq z_{q} \ast_{x} \phi_{l} \ast_{t} \varphi_{l}. 
\end{equation} 
We informally convoluted $z$ twice in space; this is for the convenience upon estimates that $y_{q}, \mathring{R}_{q},$ and $z_{q}$ all have the same structure. For a subsequent purpose in \eqref{est 174} we assume 
\begin{equation}\label{est 428} 
\frac{2^{\frac{5}{2} - \frac{\gamma_{2}}{2}}C_{0} L^{\frac{1}{2}}}{\pi} < a^{-b(1- \frac{\gamma_{2}}{2} - \beta)}.
\end{equation} 
Now we set a convention that $\sum_{1 \leq j \leq 0} \triangleq 0$ and consider the following inductive hypothesis for the universal constants $C_{0} \geq 1$. 
\begin{hypothesis}\label{Hypothesis 4.1}
\noindent For all $t \in [0, T_{L}]$, 
\begin{enumerate}[(a)] 
\item\label{Hypothesis 4.1 (a)} $\supp \hat{y}_{q} \subset B(0, 2 \lambda_{q})$,  
\begin{equation}\label{est 165}
\lVert y_{q} \rVert_{C_{t,x}} \leq C_{0} \left(1+ \sum_{1 \leq j \leq q} \delta_{j}^{\frac{1}{2}} \right) L^{\frac{1}{2}}M_{0}(t)^{\frac{1}{2}}, 
\end{equation} 
\begin{equation}\label{est 166} 
\lVert y_{q} \rVert_{C_{t} C_{x}^{2-\gamma_{2}}} + \lVert \Lambda^{2-\gamma_{2}} y_{q} \rVert_{C_{t,x}} \leq C_{0} L^{\frac{1}{2}}M_{0}(t)^{\frac{1}{2}} \lambda_{q}^{2-\gamma_{2}} \delta_{q}^{\frac{1}{2}}, 
\end{equation} 
\item\label{Hypothesis 4.1 (b)}  $\supp \hat{\mathring{R}}_{q} \subset B(0, 4 \lambda_{q})$, 
\begin{equation}\label{est 167}
\lVert \mathring{R}_{q} \rVert_{C_{t,x}} \leq \bar{C} LM_{0}(t) \lambda_{q+1}^{2-\gamma_{2}} \delta_{q+1}, 
\end{equation} 
\item\label{Hypothesis 4.1 (c)}  
\begin{equation}\label{est 168}
\lVert (\partial_{t} + (\Lambda^{2-\gamma_{2}} y_{l} + \Lambda^{2-\gamma_{2}} z_{l}) \cdot \nabla) y_{q} \rVert_{C_{t,x}} \leq C_{0} LM_{0}(t) \lambda_{q}^{3-\gamma_{2}} \delta_{q}.
\end{equation} 
\end{enumerate}  
\end{hypothesis}

\begin{remark}\label{Remark 4.1}
We need to consider $z_{q}$ in \eqref{est 150} because, as we discussed in Subsection \ref{Subsection 2.2}, in order to handle a Reynolds stress transport error, we will replace the convection by velocity $u_{q}$ in \eqref{est 10}-\eqref{est 11} by $\Lambda^{2-\gamma_{2}} v_{l} + \Lambda^{2- \gamma_{2}}z_{l}$ (see  \eqref{est 193a} and \eqref{est 195a}), and it will be crucial that this convection term as a whole has the frequency support in $B(0, 2 \lambda_{q})$ (see \eqref{est 219}); although $v_{l}$ has such a frequency support due to Hypothesis \ref{Hypothesis 4.1} \ref{Hypothesis 4.1 (a)}, $z\ast _{x} \phi_{l} \ast_{t} \varphi_{l} $ does not have such a frequency support in general due to our choice of $\alpha$ (recall \eqref{est 154c} which implies $\alpha > 1$). We also note that if we consider $z$ in \eqref{est 134} rather than $z_{q}$, then we cannot include the inductive hypothesis of $\supp \hat{\mathring{R}}_{q} \subset B(0, 4 \lambda_{q})$ in Hypothesis \ref{Hypothesis 2.1} \ref{BSV19 Hypo b} which will make our estimates too difficult because $\lambda_{q+1}^{\alpha} \overset{\eqref{est 32}}{=} l^{-1} \gg \lambda_{q}$ again due to $\alpha > 1$ in \eqref{est 154c}.   
\end{remark}

\begin{proposition}\label{Proposition 4.7}
Suppose that $\gamma_{2}, \beta, \alpha$, and $\gamma_{1}$ satisfy \eqref{est 154}. Let $L > 1$ and define 
\begin{equation}\label{est 172} 
y_{0} (t,x) \triangleq \frac{L^{2} e^{2Lt}}{2\pi}  
\begin{pmatrix}
\sin(x_{2}) \\
0
\end{pmatrix}.
\end{equation} 
Then, together with 
\begin{align}
\mathring{R}_{0}&(t,x) \triangleq \frac{L^{3} e^{2Lt}}{\pi} 
\begin{pmatrix}
0 & - \cos(x_{2}) \\
-\cos(x_{2}) & 0 
\end{pmatrix}  + \mathcal{B} (\Lambda^{\gamma_{1}} y_{0})(t,x) \label{est 156} \\
& + \mathcal{B}\left(- (\nabla y_{0})^{T} \cdot\Lambda^{2-\gamma_{2}} y_{0}  +  \Lambda^{2-\gamma_{2}} y_{0}^{\bot} \nabla^{\bot} \cdot z_{0} + \Lambda^{2-\gamma_{2}} z_{0}^{\bot} \nabla^{\bot} \cdot y_{0} + \Lambda^{2-\gamma_{2}} z_{0}^{\bot} \nabla^{\bot} \cdot z_{0} \right)(t,x),  \nonumber
\end{align} 
the pair $(y_{0}, \mathring{R}_{0})$ solves \eqref{est 134} and satisfies the Hypothesis \ref{Hypothesis 4.1} at level $ q = 0$ on $[0,T_{L}]$ provided that 
\begin{subequations}\label{est 151} 
\begin{align}
&\frac{2L M_{0}(t)^{\frac{1}{2}}}{\pi} + M_{0}(t)^{\frac{1}{2}}C_{S_{2}} 2^{-\frac{1}{2}} + C_{G_{2}}16 \sqrt{2}  L^{\frac{1}{4}} M_{0}(t)^{\frac{1}{2}}  + C_{G_{2}} 32 L^{\frac{1}{2}} \leq M_{0}(t), \label{est 151a}\\
& \frac{2^{\frac{5}{2} - \frac{\gamma_{2}}{2}} C_{0} L^{\frac{1}{2}}}{\pi} < a^{-b(1- \frac{\gamma_{2}}{2} - \beta)} \leq \frac{ \sqrt{\bar{C}} L^{\frac{1}{2}}}{ \sqrt{1+ C_{G_{1}} C_{G_{2}} 8}}, \label{est 151b} 
\end{align}
\end{subequations}  
where the first inequality of \eqref{est 151b} guarantees \eqref{est 428}. Finally, $y_{0}(0,x)$ and $\mathring{R}_{0}(0,x)$ are both deterministic. 
\end{proposition}  

\begin{proof}[Proof of Proposition \ref{Proposition 4.7}]
First, we observe that $y_{0}$ is divergence-free, mean-zero while $\mathring{R}_{0}$ is symmetric and trace-free due to Lemma \ref{Inverse-divergence lemma}. Because $(\Lambda^{2-\gamma_{2}} y_{0} \cdot \nabla) y_{0} = 0$, with $p_{0} =0$ we see due to \eqref{est 38} that $(y_{0}, \mathring{R}_{0})$ is a solution to \eqref{est 134}. Next, direct computations give 
\begin{align}
& \int_{\mathbb{T}^{2}} \lvert \sin(x_{2}) \rvert^{2} dx = \int_{\mathbb{T}^{2}} \lvert \cos(x_{2}) \rvert^{2}dx = 2\pi^{2}, \nonumber \\
& \supp \widehat{\sin(x_{2})} \subset \{(0,1), (0, -1)\}, \hspace{2mm} \widehat{\sin(x_{2}))}((0,1)) = -i2\pi^{2}, \hspace{2mm} \widehat{ \sin(x_{2})}((0, -1)) = i2\pi^{2}, \nonumber\\
& \supp \widehat{\cos(x_{2})} \subset \{(0,1), (0,-1) \}, \hspace{2mm}\widehat{\cos(x_{2})}((0,1)) = \widehat{\cos(x_{2})}((0,-1)) = 2 \pi^{2}, \label{est 171}
\end{align}
and hence $\supp \hat{y}_{0} \subset B(0, 2 \lambda_{0})$ if we take $\lambda_{0} \overset{\eqref{est 35}}{=} a > \frac{1}{2}$. Due to the frequency support of $y_{0}$ and $z_{0}$, it follows from \eqref{est 156} that $\supp \hat{\mathring{R}}_{0} \subset B(0, 4\lambda_{0})$. Now as long as we choose $C_{0} \geq \frac{1}{2\pi}$, we clearly have $\lVert y_{0} \rVert_{C_{t,x}} \leq \frac{M_{0}(t)^{\frac{1}{2}}}{2\pi} \leq C_{0} L^{\frac{1}{2}}M_{0}(t)^{\frac{1}{2}}$ as desired. Next, using the Gagliardo-Nirenberg inequality in \eqref{est 155a}, and observing that $\beta$ in \eqref{est 154} guarantees that $2 - \gamma_{2} - \beta > 0$, we can take $a \in 5 \mathbb{N}$ sufficiently large to deduce 
\begin{equation}\label{est 157}
\lVert y_{0} \rVert_{C_{t}C_{x}^{2-\gamma_{2}}} + \lVert \Lambda^{2-\gamma_{2}} y_{0} \rVert_{C_{t,x}}  \overset{\eqref{est 172} \eqref{est 155a}}{\leq} C_{G_{1}} 2^{-\frac{1}{2}}M_{0}(t)^{\frac{1}{2}} \leq C_{0}L^{\frac{1}{2}} M_{0}(t)^{\frac{1}{2}} \lambda_{0}^{2-\gamma_{2}} \delta_{0}^{\frac{1}{2}}.
\end{equation} 
Thus, the hypothesis \ref{Hypothesis 4.1 (a)} holds. To verify the hypothesis \ref{Hypothesis 4.1 (b)} we compute 
\begin{equation}\label{est 159}
\lVert \mathring{R}_{0}(t) \rVert_{C_{x}} \overset{\eqref{est 156}}{\leq} \sum_{k=1}^{3} \RomanII_{k}(t) 
\end{equation}
where  
\begin{subequations}\label{est 158} 
\begin{align} 
\RomanII_{1}(t) \triangleq& \frac{2L M_{0}(t)^{\frac{1}{2}}}{\pi}, \hspace{3mm} \RomanII_{2}(t) \triangleq \lVert \mathcal{B}\Lambda^{\gamma_{1}} y_{0}(t) \rVert_{C_{x}}, \label{est 158b} \\
\RomanII_{3}(t) \triangleq& \lVert \mathcal{B} ( ( \nabla y_{0})^{T} \cdot \Lambda^{2-\gamma_{2}} y_{0} )(t) \rVert_{C_{x}} + \lVert \mathcal{B} ( \Lambda^{2-\gamma_{2}} y_{0}^{\bot} \nabla^{\bot} \cdot z_{0} ) (t) \rVert_{C_{x}}  \nonumber\\
&+ \lVert \mathcal{B} (\Lambda^{2-\gamma_{2}} z_{0}^{\bot} \nabla^{\bot} \cdot y_{0})(t) \rVert_{C_{x}}  + \lVert \mathcal{B}(\Lambda^{2-\gamma_{2}} z_{0}^{\bot} \nabla^{\bot} \cdot z_{0} )(t) \rVert_{C_{x}}.   \label{est 158c} 
\end{align}
\end{subequations} 
We can estimate 
\begin{subequations}\label{est 160}
\begin{align}
&  \RomanII_{2}(t) \overset{\eqref{est 155b}}{\leq} \frac{M_{0}(t)^{\frac{1}{2}}}{2\pi} C_{S_{2}} \lVert \mathcal{B} \Lambda^{\gamma_{1}} \sin(x_{2}) \rVert_{\dot{H}_{x}^{3- \gamma_{1}}}   = M_{0}(t)^{\frac{1}{2}}C_{S_{2}} 2^{-\frac{1}{2}},  \label{est 160a} \\
& \lVert \mathcal{B} (( \nabla y_{0})^{T} \cdot \Lambda^{2-\gamma_{2}} y_{0} )(t) \rVert_{C_{x}} \overset{\eqref{est 155a}}{\leq} C_{G_{2}}C_{G_{1}} 8M_{0}(t), \label{est 160c} \\
& \lVert \mathcal{B} (\Lambda^{2-\gamma_{2}} y_{0}^{\bot} \nabla^{\bot} \cdot z_{0})(t) \rVert_{C_{x}} \overset{\eqref{est 155a} \eqref{est 136}}{\leq} C_{G_{2}} 8\sqrt{2} L^{\frac{1}{4}}M_{0}(t)^{\frac{1}{2}}, \label{est 160d} \\
& \lVert \mathcal{B} ( \Lambda^{2-\gamma_{2}} z_{0}^{\bot} \nabla^{\bot} \cdot y_{0} )(t) \rVert_{C_{x}} \overset{\eqref{est 155a}\eqref{est 136}}{\leq}C_{G_{2}} 8\sqrt{2} L^{\frac{1}{4}} M_{0}(t)^{\frac{1}{2}}, \label{est 160e} \\
& \lVert \mathcal{B} (\Lambda^{2-\gamma_{2}} z_{0}^{\bot} \nabla^{\bot} \cdot z_{0}  )(t) \rVert_{C_{x}} \overset{\eqref{est 155a}\eqref{est 136}}{\leq} C_{G_{2}}32 L^{\frac{1}{2}}, \label{est 160f} 
\end{align}
\end{subequations} 
where the four terms \eqref{est 160c}-\eqref{est 160f} within $\RomanII_{3}(t)$ were computed using the definition of $\mathcal{B}$ from Lemma \ref{Inverse-divergence lemma}.  Applying \eqref{est 160} to \eqref{est 158} and \eqref{est 159} allows us to compute 
\begin{align*}
\lVert \mathring{R}_{0}(t) \rVert_{C_{x}} \overset{\eqref{est 159} \eqref{est 158} \eqref{est 160}\eqref{est 151a}}{\leq}& [1+ C_{G_{1}} C_{G_{2}} 8] M_{0}(t)  \\
\overset{\eqref{est 151b}}{\leq}& \bar{C} L M_{0}(t) a^{b(2- \gamma_{2} - 2 \beta)} \overset{\eqref{est 35}}{=} \bar{C} L M_{0}(t) \lambda_{1}^{2-\gamma_{2}} \delta_{1};
\end{align*} 
thus, the Hypothesis \ref{Hypothesis 4.1} \ref{Hypothesis 4.1 (b)} was verified. The verification of Hypothesis \ref{Hypothesis 4.1} \ref{Hypothesis 4.1 (c)} is immediate because $\beta$ in \eqref{est 154b} guarantees that $3 - \gamma_{2} - 2 \beta > 0$ allowing us to estimate  
\begin{align*}
&\lVert (\partial_{t} + (\Lambda^{2-\gamma_{2}} y_{\lambda_{1}^{-\alpha}} + \Lambda^{2-\gamma_{2}} z_{\lambda_{1}^{-\alpha}} ) \cdot \nabla) y_{0} \rVert_{C_{t,x}} \\
\overset{\eqref{est 155a} \eqref{est 136}}{\leq}& \frac{ L}{\pi} M_{0}(t)^{\frac{1}{2}} + C_{G_{1}} \frac{M_{0}(t)}{2\sqrt{2} \pi} + L^{\frac{1}{4}} \frac{M_{0}(t)^{\frac{1}{2}}}{2\pi} \leq C_{0} LM_{0}(t) a^{3 - \gamma_{2} - 2 \beta} \overset{\eqref{est 35}}{=} C_{0} LM_{0}(t) \lambda_{0}^{3- \gamma_{2}} \delta_{0} 
\end{align*} 
for $a \in 5 \mathbb{N}$ sufficiently large where we denoted 
\begin{equation*}
y_{\lambda_{1}^{-\alpha}}\triangleq  y_{0} \ast_{x}\phi_{\lambda_{1}^{-\alpha}} \ast_{t} \varphi_{\lambda_{1}^{-\alpha}} \hspace{1mm} \text{  and } \hspace{1mm}z_{\lambda_{1}^{-\alpha}} \triangleq z_{0} \ast_{x} \phi_{\lambda_{1}^{-\alpha}} \ast_{t} \varphi_{\lambda_{1}^{-\alpha}}. 
\end{equation*}
Finally, $y_{0}(0,x)$ is deterministic and $\mathring{R}_{0}(0,x)$ is deterministic due to $z(0,x) \equiv 0$ from \eqref{est 152}. 
\end{proof} 

In contrast to \eqref{est 16} we shall define 
\begin{equation}\label{est 164} 
y_{q+1} \triangleq y_{l} + w_{q+1}. 
\end{equation} 

\begin{proposition}\label{Proposition 4.8}
Fix $L > 1$. Suppose that $\gamma_{2}, \beta,  \alpha$, and $\gamma_{1}$ satisfy \eqref{est 154} and $(y_{q}, \mathring{R}_{q})$  is a $(\mathcal{F}_{t})_{t\geq 0}$-adapted solution to \eqref{est 134} that satisfies Hypothesis \ref{Hypothesis 4.1}. Then there exist $a \in 5 \mathbb{N}$ sufficiently large and $\beta > 1 - \frac{\gamma_{2}}{2}$ in \eqref{est 154b} sufficiently close to $1- \frac{\gamma_{2}}{2}$ that satisfies \eqref{est 151b} and $(\mathcal{F}_{t})_{t\geq 0}$-adapted processes $(y_{q+1}, \mathring{R}_{q+1})$ that solves \eqref{est 134}, satisfies the Hypothesis \ref{Hypothesis 4.1} at level $q+1$, and for all $t \in [0, T_{L}]$, 
\begin{equation}\label{est 162} 
\lVert y_{q+1}(t) - y_{q}(t) \rVert_{C_{x}} \leq C_{0}L^{\frac{1}{2}} M_{0}(t)^{\frac{1}{2}}  \delta_{q+1}^{\frac{1}{2}}.
\end{equation} 
Moreover, $w_{q+1} = y_{q+1} - y_{l}$ satisfies 
\begin{equation}\label{est 163} 
\supp\hat{w}_{q+1} \subset \left\{ \xi: \frac{\lambda_{q+1}}{2} \leq \lvert \xi \rvert \leq 2 \lambda_{q+1} \right\}.
\end{equation} 
Finally, if $y_{q}(0,x)$ and $\mathring{R}_{q}(0,x)$ are deterministic, then so are $y_{q+1}(0,x)$ and $\mathring{R}_{q+1}(0,x)$. 
\end{proposition}

\begin{proof}[Proof of Theorem \ref{Theorem 2.1} assuming Proposition \ref{Proposition 4.8}]
Fix $T > 0, K > 1$, and $\kappa \in (0,1)$. We take $L > 1$ sufficiently large that satisfies \eqref{est 151a}, $2C_{0} L^{\frac{1}{2}} \geq \pi$, and \begin{subequations}
\begin{align}
&\sqrt{2} \pi^{2} L^{2} e^{2L T} > (\sqrt{2} \pi^{2} L^{2} + \sqrt{8} \pi^{2} L^{2} + L) e^{LT}, \label{est 175} \\
&e^{LT} > K, \hspace{2mm} \text{ and } \hspace{2mm} L e^{LT} \geq L^{\frac{1}{4}} + K \sqrt{ T \Tr ((-\Delta)^{\frac{1}{2}} GG^{\ast} )}. \label{est 176} 
\end{align}
\end{subequations} 
Because $\lim_{L\to\infty} T_{L} = \infty$ $\textbf{P}$-a.s. due to \eqref{est 128}, for the fixed $T > 0$  and $\kappa > 0$, we increase $L$ larger if necessary to attain \eqref{est 148}. We fix such $L > 1$ and rely on Proposition \ref{Proposition 4.8} to find $a \in 5 \mathbb{N}$ and $\beta$ in \eqref{est 154b}, allowing us to start from $(y_{0}, \mathring{R}_{0})$ in Proposition \ref{Proposition 4.7} and continue with $(y_{q}, \mathring{R}_{q})$ for all $q \in \mathbb{N}$
that solves \eqref{est 134} and satisfies the Hypothesis \ref{Hypothesis 4.1}, as well as \eqref{est 162}. Now, because $z\in C_{T} \dot{H}_{x}^{4+ \frac{\sigma}{2}}$ due to \eqref{est 115}, it is immediate that $z_{q} \to z$ in $L_{T_{L}}^{2} \dot{H}_{x}^{1}$ as $q\to \infty$ which is all we need for our subsequent purpose. We compute for all $\beta ' \in (\frac{1}{2}, \beta)$, and $t \in [0, T_{L}]$ by using the interpolation inequality (e.g.,  \cite[p. 88]{CDS12b}), as well as the fact that $b^{q+1} \geq b(q+1)$ for all $q \geq 0$ and $b \geq 2$, 
\begin{align} 
&\sum_{q\geq 0} \lVert y_{q+1}(t) - y_{q}(t) \rVert_{C_{x}^{\beta'}} 
\overset{\eqref{est 162} }{\lesssim} \sum_{q\geq 0} ( L^{\frac{1}{2}} M_{0}(t)^{\frac{1}{2}} \delta_{q+1}^{\frac{1}{2}})^{1- \beta'} ( \lVert y_{q+1} \rVert_{C_{t}C_{x}^{1}} + \lVert y_{q} \rVert_{C_{t}C_{x}^{1}})^{\beta'} \label{est 328} \\
\overset{\eqref{est 166}}{\lesssim}& \sum_{q\geq 0} (L^{\frac{1}{2}}M_{0}(t)^{\frac{1}{2}} \delta_{q+1}^{\frac{1}{2}})^{1- \beta'} ( L^{\frac{1}{2}}M_{0}(t)^{\frac{1}{2}} \delta_{q+1}^{\frac{1}{2}} \lambda_{q+1})^{\beta'} \lesssim L^{\frac{1}{2}}M_{0}(t)^{\frac{1}{2}}\sum_{q\geq 0} a^{b(q+1) (-\beta + \beta')} < \infty.  \nonumber 
\end{align} 
Therefore, $\{y_{q}\}_{q \in \mathbb{N}_{0}}$ is Cauchy in $C_{t}C_{x}^{\beta'}$ for $\beta' \in (\frac{1}{2}, \beta)$. This is why $\frac{1}{2} + \iota$ in \eqref{est 129} requires $\frac{1}{2} + \iota < \beta \overset{\eqref{est 154d} }{<} \frac{3-\gamma_{1}}{3}$ so that $\iota < \frac{1}{2} - \frac{\gamma_{1}}{3}$. We can also prove the following temporal regularity estimate; because it will be useful in the proof of Proposition \ref{Proposition 4.8}, we prove this estimate for general $\gamma_{2} \in [1,2)$ although $\gamma_{2} = 1$ in the current proof of Theorem \ref{Theorem 2.1}. Because \eqref{est 154} guarantees that $\beta < 2- \gamma_{2}$, by \eqref{est 168}, \eqref{est 166}, and \eqref{est 136}, we estimate 
\begin{align}
\lVert \partial_{t} y_{q}(t) \rVert_{C_{x}} \leq& \lVert ( \partial_{t} + (\Lambda^{2- \gamma_{2}} y_{l} + \Lambda^{2- \gamma_{2}} z_{l}) \cdot \nabla) y_{q} \rVert_{C_{t,x}} + \lVert (\Lambda^{2- \gamma_{2}}  y_{l} + \Lambda^{2- \gamma_{2}} z_{l}) \cdot \nabla y_{q} \rVert_{C_{t,x}} \nonumber \\
&\lesssim LM_{0}(t) \lambda_{q}^{3-\gamma_{2}} \delta_{q} +  ( L^{\frac{1}{2}} M_{0}(t)^{\frac{1}{2}} \delta_{q}^{\frac{1}{2}} \lambda_{q}^{2-\gamma_{2}} + L^{\frac{1}{4}}) \lambda_{q}^{\gamma_{2} -1} (L^{\frac{1}{2}}M_{0}(t)^{\frac{1}{2}} \delta_{q}^{\frac{1}{2}} \lambda_{q}^{2-\gamma_{2}}) \nonumber\\
&\lesssim LM_{0}(t) \lambda_{q}^{3-\gamma_{2}} \delta_{q}. \label{est 169} 
\end{align}
Considering \eqref{est 169} in case $\gamma_{2} = 1$, by interpolation inequality again, \eqref{est 162} and \eqref{est 169}, this leads to, for all $q \in \mathbb{N}$ and $\eta \in [0, \frac{\beta}{2-\beta})$ where $\beta < 2$ due to \eqref{est 154b}, 
\begin{equation}\label{est 329}
\lVert y_{q} - y_{q-1} \rVert_{C_{t}^{\eta} C_{x}} \lesssim \lVert y_{q} - y_{q-1} \rVert_{C_{t,x}}^{1 - \eta} \lVert y_{q} - y_{q-1} \rVert_{C_{t}^{1}C_{x}}^{\eta} \lesssim L^{\frac{1+ \eta}{2}}M_{0}(t)^{\frac{1+\eta}{2}} \lambda_{q}^{2\eta - \beta (1+ \eta)} \to 0 
\end{equation} 
as $q\to \infty$. Hence, there exists a limiting solution $\lim_{q\to\infty} y_{q} \triangleq y \in C([0,T_{L}]; C^{\beta'} (\mathbb{T}^{2})) \cap C^{\eta} ( [0, T_{L}]; C(\mathbb{T}^{2}))$ for which  
there exists deterministic constant $C_{L} > 0$ such that 
\begin{equation}\label{est 137}
\lVert y \rVert_{C_{T_{L}}C_{x}^{\beta'}} + \lVert y \rVert_{C_{T_{L}}^{\eta} C_{x}} \leq C_{L} \hspace{3mm}\forall \hspace{1mm} \beta' \in \left(\frac{1}{2}, \beta \right), \eta \in \left[0, \frac{\beta}{2-\beta} \right), 
\end{equation}  
where $\frac{1}{3}< \frac{\beta}{2-\beta}$; we denote its initial data as $y^{\text{in}} \triangleq y\rvert_{t=0}$. Because each $y_{q}$ is $(\mathcal{F}_{t})_{t\geq 0}$-adapted due to Propositions \ref{Proposition 4.7}-\ref{Proposition 4.8}, we deduce that $y$ is $(\mathcal{F}_{t})_{t\geq 0}$-adapted. Moreover, because the range of $\beta$ in \eqref{est 154b} guarantees that $1 - 2 \beta < 0$, we deduce that  
\begin{equation}\label{est 331}
\lVert \mathring{R}_{q}\rVert_{C_{T_{L}} C_{x}} \overset{\eqref{est 167}}{\leq} \bar{C} LM_{0}(L) \lambda_{q+1} \delta_{q+1} \overset{\eqref{est 35}}{=} \bar{C}  L^{5} e^{4L^{2}} \lambda_{q+1}^{1 - 2\beta} \to 0  
\end{equation} 
as $q\to\infty$. Using the convergence of $y_{q}$ to $y$ in $C_{T_{L}}C_{x}^{\beta'}$ for $\beta' \in (\frac{1}{2}, \beta)$, as well as \eqref{est 79} and \eqref{est 50}, it follows that $(y,z)$ solves \eqref{est 113} weakly; i.e. for all $\psi \in C^{\infty}(\mathbb{T}^{2}) \cap \dot{H}_{\sigma}^{\frac{1}{2}}$ and $t \in [0,T_{L}]$ 
\begin{align}
& \langle y(t) - y(0)), \psi \rangle - \int_{0}^{t} \sum_{i,j=1}^{2} \langle \Lambda (y+z)_{i}(s), \partial_{i} \psi_{j} (y+z)_{j}(s)\rangle_{\dot{H}_{x}^{-\frac{1}{2}} - \dot{H}_{x}^{\frac{1}{2}}}  - \langle y(s), \Lambda^{\gamma_{1}} \psi \rangle \nonumber\\
& \hspace{30mm} - \frac{1}{2} \sum_{i,j=1}^{2} \langle  \partial_{i} (y+z)_{j}(s), [\Lambda, \psi_{i}] (y+z)_{j}(s) \rangle_{\dot{H}_{x}^{-\frac{1}{2}} - \dot{H}_{x}^{\frac{1}{2}}} ds = 0. \label{est 170}
\end{align}  
It follows from \eqref{est 152} and definitions of $y = v-z$ and $p_{2} = p - p_{1}$ that $v$ solves \eqref{est 30} weakly. As $z(0,x) = 0$ from \eqref{est 152}, we denote $v^{\text{in}} \triangleq y^{\text{in}}$. Next, we compute using Parseval theorem 
\begin{equation}\label{est 173}
\lVert \Lambda^{\frac{1}{2}} y_{0}(t) \rVert_{L_{x}^{2}}^{2} \overset{\eqref{est 172} \eqref{est 171}}{=} M_{0}(t) 8 \pi^{4}. 
\end{equation} 
Because $z(t)$ from \eqref{est 114} and $y$ are $(\mathcal{F}_{t})_{t\geq 0}$ -adapted, we see that $v$ is $(\mathcal{F}_{t})_{t\geq 0}$-adapted. Moreover, the bounds from \eqref{est 137} on $y_{q}$ and \eqref{est 136} on $z$ give the regularity of $v$ in \eqref{est 129} as follows: for almost every (a.e.) $\omega \in \Omega$, $\beta' \in (\frac{1}{2}, \beta)$, and $\eta \in (0, \frac{1}{3} ]$, 
\begin{align*}
\lVert v(\omega) \rVert_{C_{T_{L}} C_{x}^{\beta'}} + \lVert v(\omega) \rVert_{C_{T_{L}}^{\eta} C_{x}} \leq \lVert y(\omega) \rVert_{C_{T_{L}} C_{x}^{\beta'}} + \lVert z (\omega) \rVert_{C_{T_{L}} C_{x}^{\beta'}} + \lVert y(\omega) \rVert_{C_{T_{L}}^{\eta} C_{x}} + \lVert z(\omega) \rVert_{C_{T_{L}}^{\eta} C_{x}} < \infty. 
\end{align*} 
Next, we compute for all $t \in [0, T_{L}]$ using the fact that $\supp (y_{q+1} - y_{q}) \hspace{1mm} \hat{} \subset B(0, 2\lambda_{q+1})$ due to Hypothesis \ref{Hypothesis 4.1} \ref{Hypothesis 4.1 (a)}, the fact that $b^{q+1} \geq b(q+1)$ for all $q \geq 0$ and $b \geq 2$, and $\frac{1}{2} < \beta$ from \eqref{est 154b}, 
\begin{align}
& \lVert y(t) - y_{0}(t) \rVert_{\dot{H}_{x}^{\frac{1}{2}}} \overset{ \eqref{est 163} }{\leq} \sum_{q \geq 0} ( 2\lambda_{q+1})^{\frac{1}{2}} \lVert (y_{q+1} - y_{q})(t) \rVert_{L_{x}^{2}} \label{est 174}\\
\overset{\eqref{est 162}}{\leq}& 2^{\frac{3}{2}}\pi C_{0} L^{\frac{1}{2}} M_{0}(t)^{\frac{1}{2}} \sum_{q \geq 0}  a^{b(q+1) (\frac{1}{2} - \beta)}  = 2^{\frac{3}{2}}\pi C_{0} L^{\frac{1}{2}} M_{0}(t)^{\frac{1}{2}} \frac{ a^{b (\frac{1}{2}- \beta)}}{1- a^{b(\frac{1}{2} - \beta)}} \overset{\eqref{est 151b}}{<} \sqrt{2} \pi^{2} M_{0}(t)^{\frac{1}{2}}. \nonumber 
\end{align}
Then we can compute 
\begin{align}
\lVert y(T) \rVert_{\dot{H}_{x}^{\frac{1}{2}}} 
\overset{\eqref{est 173}}{\geq}& M_{0}(T)^{\frac{1}{2}} \sqrt{8} \pi^{2} - \lVert y(T) - y_{0}(T) \rVert_{\dot{H}_{x}^{\frac{1}{2}}}  \label{est 177}\\
\overset{\eqref{est 174}}{\geq}&  \sqrt{2} \pi^{2} M_{0}(T)^{\frac{1}{2}} 
\overset{\eqref{est 175}}{>} ( \sqrt{2} \pi^{2} L^{2} + \sqrt{8} \pi^{2} L^{2} + L) e^{LT} \nonumber \\
\overset{\eqref{est 173} \eqref{est 174}}{\geq}& ( \lVert y(0) - y_{0}(0) \rVert_{\dot{H}_{x}^{\frac{1}{2}}} + \lVert y_{0} (0) \rVert_{\dot{H}_{x}^{\frac{1}{2}}} + L) e^{LT} \geq ( \lVert y(0) \rVert_{\dot{H}_{x}^{\frac{1}{2}}} + L) e^{LT}. \nonumber 
\end{align} 
Therefore, we can deduce \eqref{est 146} as follows: on $\{T_{L} \geq T \}$
\begin{align*}
\lVert v(T) \rVert_{\dot{H}_{x}^{\frac{1}{2}}} \geq& \lVert y(T) \rVert_{\dot{H}_{x}^{\frac{1}{2}}} - \lVert z(T) \rVert_{\dot{H}_{x}^{\frac{1}{2}}} \\
\overset{\eqref{est 177} \eqref{est 136}}{>}& (\lVert y(0) \rVert_{\dot{H}_{x}^{\frac{1}{2}}} + L)e^{LT} - L^{\frac{1}{4}} \overset{\eqref{est 176}}{\geq}  K \left[ \lVert v^{\text{in}} \rVert_{\dot{H}_{x}^{\frac{1}{2}}} + \sqrt{ T \Tr ((-\Delta)^{\frac{1}{2}} GG^{\ast} ) } \right].  
\end{align*}
Finally, because $y_{0} (0,x)$ is deterministic due to Proposition \ref{Proposition 4.7}, Proposition \ref{Proposition 4.8} implies that $y(0,x)$ is deterministic. As $z (0,x) \equiv 0$ by \eqref{est 152b}, this implies that $v^{\text{in}}$ is deterministic. 
\end{proof} 

\subsection{Proof of Proposition \ref{Proposition 4.8}}
Similarly to previous works on probabilistic convex integration, there will be various functions of $L$ in estimates and we will take $b = b(L, \gamma_{2}) \in \mathbb{N}$ sufficiently large (e.g., \eqref{138 star}, \eqref{est 201}, \eqref{est 223}, \eqref{est 266}, and \eqref{est 284}) to bound them. Here, we emphasize that this $b$ can be chosen independently from $a$ and $\beta$. Considering such $b \in \mathbb{N}$ fixed not depending on $a$ or $\beta$, we notice that due to \eqref{est 419} we can take $a \in 5 \mathbb{N}$ sufficiently large such that $a \geq e^{4}$ and $\beta > 1 -\frac{\gamma_{2}}{2}$ sufficiently close to $1 -\frac{\gamma_{2}}{2}$ and achieve \eqref{est 151b}. It follows from \eqref{est 134}, \eqref{est 180}, \eqref{est 180}, and \eqref{est 38}  that
\begin{align}
&\partial_{t} y_{l} + ( \Lambda^{2- \gamma_{2}} (y_{l} + z_{l}) \cdot \nabla) (y_{l} + z_{l}) - (\nabla (y_{l} + z_{l}))^{T} \cdot \Lambda^{2- \gamma_{2}} (y_{l} + z_{l}) \nonumber\\
&\hspace{45mm} + \nabla p_{l} + \Lambda^{\gamma_{1}} y_{l} = \divergence (\mathring{R}_{l} + R_{\text{Com1}}), \label{est 184}
\end{align}
where  
\begin{align}
& p_{l} \triangleq  p_{q} \ast_{x} \phi_{l} \ast_{t} \varphi_{l}, \label{est 181}\\
& R_{\text{Com1}} \triangleq \mathcal{B} [ \Lambda^{2- \gamma_{2}} (y_{l} + z_{l})^{\bot} \nabla^{\bot} \cdot (y_{l} + z_{l})  - [\Lambda^{2- \gamma_{2}} (y_{q} + z_{q})^{\bot} \nabla^{\bot} \cdot (y_{q} + z_{q})] \ast_{x} \phi_{l} \ast_{t} \varphi_{l} ].  \nonumber 
\end{align}
Because \eqref{est 154b} guarantees that $\beta < 2- \gamma_{2}$, for $a \in 5 \mathbb{N}$ sufficiently large we see that 
\begin{equation}\label{est 203} 
\lVert y_{l} - y_{q} \rVert_{C_{t,x}} \lesssim l \lVert y_{q} \ast_{x} \phi_{l} \rVert_{C_{t}^{1}C_{x}} + l \lVert y_{q} \rVert_{C_{t}C_{x}^{1}} 
\overset{\eqref{est 169} \eqref{est 166}}{\lesssim} l L M_{0}(t) \lambda_{q}^{3-\gamma_{2}} \delta_{q}.
\end{equation} 
Now we strategically decompose the Reynolds stress at level $q+1$; we note that especially due to $z_{q+1}$ and $z_{q}$, there will be an extra layer of complication in comparison to previous works (e.g., \cite{HZZ19, Y20a, Y20c}). First, due to  \eqref{est 134}, \eqref{est 164}, and \eqref{est 184}, 
\begin{align}
& \divergence \mathring{R}_{q+1} - \nabla p_{q+1} =  - (\Lambda^{2- \gamma_{2}} y_{l} \cdot \nabla) z_{l} - (\Lambda^{2- \gamma_{2}} z_{l} \cdot \nabla) y_{l} - (\Lambda^{2- \gamma_{2}} z_{l} \cdot \nabla) z_{l} \nonumber\\
& + (\nabla y_{l})^{T} \cdot \Lambda^{2- \gamma_{2}} z_{l} + (\nabla z_{l})^{T} \cdot \Lambda^{2- \gamma_{2}} y_{l} + (\nabla z_{l})^{T} \cdot \Lambda^{2- \gamma_{2}} z_{l} - \nabla p_{l} + \divergence \mathring{R}_{l} + \divergence R_{\text{Com1}} \nonumber\\
&+  \partial_{t} w_{q+1}  + (\Lambda^{2- \gamma_{2}} y_{l} \cdot \nabla) w_{q+1} + (\Lambda^{2- \gamma_{2}} y_{l} \cdot \nabla) z_{q+1} \nonumber\\
&+ (\Lambda^{2- \gamma_{2}} w_{q+1} \cdot \nabla) y_{l} + (\Lambda^{2- \gamma_{2}} w_{q+1} \cdot \nabla) w_{q+1} + (\Lambda^{2- \gamma_{2}} w_{q+1} \cdot \nabla) z_{q+1} \nonumber\\
&+ (\Lambda^{2- \gamma_{2}} z_{q+1} \cdot \nabla) y_{l} + (\Lambda^{2- \gamma_{2}} z_{q+1} \cdot \nabla) w_{q+1} + (\Lambda^{2- \gamma_{2}} z_{q+1} \cdot \nabla) z_{q+1} \nonumber\\
& - (\nabla y_{l})^{T} \cdot \Lambda^{2- \gamma_{2}} w_{q+1} - (\nabla y_{l})^{T} \cdot \Lambda^{2- \gamma_{2}} z_{q+1} \nonumber\\
& - (\nabla w_{q+1})^{T} \cdot \Lambda^{2- \gamma_{2}} y_{l} - (\nabla w_{q+1})^{T} \cdot \Lambda^{2- \gamma_{2}} w_{q+1} - (\nabla w_{q+1})^{T} \cdot \Lambda^{2- \gamma_{2}} z_{q+1} \nonumber \\
& - (\nabla z_{q+1})^{T} \cdot \Lambda^{2- \gamma_{2}} y_{l} - (\nabla z_{q+1})^{T} \cdot \Lambda^{2- \gamma_{2}} w_{q+1} - (\nabla z_{q+1})^{T} \cdot \Lambda^{2- \gamma_{2}} z_{q+1} + \Lambda^{\gamma_{1}} w_{q+1}.  \label{est 189}
\end{align}
We rewrite some of these terms as follows: 
\begin{align}
& - (\Lambda^{2- \gamma_{2}} y_{l} \cdot \nabla) z_{l} - (\Lambda^{2- \gamma_{2}} z_{l} \cdot \nabla) y_{l} - ( \Lambda^{2- \gamma_{2}} z_{l} \cdot \nabla) z_{l} \label{est 185}\\
&+ (\Lambda^{2- \gamma_{2}} y_{l} \cdot \nabla) z_{q+1} + (\Lambda^{2- \gamma_{2}} w_{q+1} \cdot \nabla) z_{q+1} + (\Lambda^{2- \gamma_{2}} z_{q+1} \cdot \nabla) y_{l} \nonumber\\
&+ (\Lambda^{2- \gamma_{2}} z_{q+1} \cdot \nabla) w_{q+1} + (\Lambda^{2- \gamma_{2}} z_{q+1} \cdot \nabla) z_{q+1} \nonumber\\
\overset{\eqref{est 164} }{=}&  \Lambda^{2- \gamma_{2}} y_{q+1} \cdot\nabla (z_{q+1} - z_{q}) + \Lambda^{2- \gamma_{2}} y_{q+1} \cdot\nabla (z_{q} - z_{l}) + \Lambda^{2- \gamma_{2}} w_{q+1} \cdot\nabla z_{l}  \nonumber\\
&+ \Lambda^{2- \gamma_{2}} (z_{q+1} - z_{q}) \cdot\nabla y_{q+1} + \Lambda^{2- \gamma_{2}} (z_{q} - z_{l}) \cdot\nabla y_{q+1} +\Lambda^{2- \gamma_{2}} z_{l} \cdot\nabla w_{q+1} \nonumber \\
&+ \Lambda^{2- \gamma_{2}} (z_{q+1} - z_{q}) \cdot\nabla z_{q+1} + \Lambda^{2- \gamma_{2}} (z_{q} - z_{l}) \cdot\nabla z_{q+1} \nonumber\\
&+ \Lambda^{2-\gamma_{2}} z_{l} \cdot\nabla (z_{q+1} - z_{q}) + \Lambda^{2- \gamma_{2}} z_{l} \cdot\nabla (z_{q} - z_{l} ), \nonumber 
\end{align}
\begin{align}
& (\nabla y_{l})^{T} \cdot \Lambda^{2- \gamma_{2}} z_{l} + (\nabla z_{l})^{T} \cdot \Lambda^{2- \gamma_{2}} y_{l} + (\nabla z_{l})^{T} \cdot \Lambda^{2- \gamma_{2}} z_{l} \label{est 186}  \\
& - (\nabla y_{l})^{T} \cdot \Lambda^{2- \gamma_{2}} z_{q+1} - (\nabla z_{q+1})^{T} \cdot \Lambda^{2- \gamma_{2}} y_{l} - (\nabla z_{q+1})^{T} \cdot \Lambda^{2- \gamma_{2}} z_{q+1}\nonumber \\
=& (\nabla y_{l})^{T} \cdot \Lambda^{2- \gamma_{2}} (z_{l} - z_{q}) + (\nabla y_{l})^{T} \cdot \Lambda^{2- \gamma_{2}} (z_{q}  - z_{q+1}) \nonumber \\
&+ ( \nabla (z_{l} - z_{q} ))^{T} \cdot \Lambda^{2- \gamma_{2}} y_{l}+ ( \nabla (z_{q} - z_{q+1}))^{T} \cdot \Lambda^{2- \gamma_{2}} y_{l} \nonumber \\
&+ ( \nabla (z_{l} - z_{q}))^{T} \cdot \Lambda^{2- \gamma_{2}} z_{l} + ( \nabla (z_{q} - z_{q+1}))^{T} \cdot \Lambda^{2- \gamma_{2}} z_{l}  \nonumber \\
&+ (\nabla z_{q+1})^{T} \cdot \Lambda^{2- \gamma_{2}} (z_{l} - z_{q}) + (\nabla z_{q+1})^{T} \cdot\Lambda^{2- \gamma_{2}} (z_{q} - z_{q+1}), \nonumber 
\end{align}
\begin{align}
&- (\nabla w_{q+1})^{T} \cdot \Lambda^{2- \gamma_{2}} y_{l} - (\nabla w_{q+1})^{T} \cdot \Lambda^{2- \gamma_{2}} z_{q+1} \label{est 187}\\
=& ( \nabla  \Lambda^{2- \gamma_{2}} (y_{l} + z_{q} ))^{T} \cdot w_{q+1} + ( \nabla \Lambda^{2- \gamma_{2}} (z_{q+1} - z_{q}) )^{T} \cdot w_{q+1} - \nabla (w_{q+1} \cdot \Lambda^{2-\gamma_{2}} (y_{l} + z_{q+1})), \nonumber 
\end{align} 
and because $\Lambda^{2- \gamma_{2}} w_{q+1} \cdot \nabla z_{l}$ in \eqref{est 185} will still create a difficulty to handle, we couple it with $- (\nabla z_{q+1})^{T} \cdot \Lambda^{2- \gamma_{2}} w_{q+1}$ from \eqref{est 189} to write 
\begin{align}
& \Lambda^{2- \gamma_{2}} w_{q+1} \cdot\nabla z_{l} - (\nabla z_{q+1})^{T} \cdot \Lambda^{2- \gamma_{2}} w_{q+1}  \label{est 188}\\
\overset{\eqref{est 38}  }{=}& \Lambda^{2- \gamma_{2}} w_{q+1}^{\bot} (\nabla^{\bot} \cdot z_{l}) + ( \nabla ( z_{l} - z_{q}))^{T} \cdot \Lambda^{2- \gamma_{2}} w_{q+1} + ( \nabla (z_{q} - z_{q+1}))^{T} \cdot \Lambda^{2- \gamma_{2}} w_{q+1}.\nonumber 
\end{align} 
Finally, the products of $y_{q+1}$ with $z_{q}- z_{l}$ in \eqref{est 185} are difficult to treat while a product of $w_{q+1}$ with $z_{q} - z_{l}$ grants us a factor of $\lambda_{q+1}^{-1}$ from $\mathcal{B}$ considering its frequency support (see how we treat $R_{T}$ in Subsection \ref{Subsection 4.1.1} and $R_{N}$ in Subsection \ref{Subsection 4.1.2}); therefore, we pair up some terms from \eqref{est 185} and \eqref{est 186} selectively to rewrite by using \eqref{est 164} and \eqref{est 38},  
\begin{subequations}\label{est 411} 
\begin{align}
& \Lambda^{2- \gamma_{2}} y_{q+1} \cdot \nabla (z_{q} - z_{l}) + (\nabla (z_{l} - z_{q}))^{T} \cdot \Lambda^{2-\gamma_{2}} y_{l}  \nonumber\\
=& \Lambda^{2-\gamma_{2}} w_{q+1} \cdot \nabla (z_{q} - z_{l}) + \Lambda^{2-\gamma_{2}} y_{l}^{\bot} \nabla^{\bot} \cdot (z_{q} - z_{l}), \label{est 411a}  \\
&\Lambda^{2-\gamma_{2}} (z_{q} - z_{l}) \cdot \nabla y_{q+1} + (\nabla y_{l})^{T} \cdot \Lambda^{2-\gamma_{2}} (z_{l} - z_{q})  \nonumber \\
=& \Lambda^{2-\gamma_{2}} (z_{q} - z_{l}) \cdot \nabla w_{q+1} + \Lambda^{2-\gamma_{2}} (z_{q} - z_{l})^{\bot} \nabla^{\bot} \cdot y_{l}, \label{est 411b}  
\end{align}
\end{subequations}
and furthermore for convenience by \eqref{est 38}, 
\begin{align}
& ( \nabla (z_{l}- z_{q}))^{T} \cdot \Lambda^{2-\gamma_{2}} w_{q+1} + \Lambda^{2-\gamma_{2}} y_{q+1} \cdot \nabla (z_{q} - z_{l}) + ( \nabla (z_{l} - z_{q}))^{T}\cdot \Lambda^{2-\gamma_{2}} y_{l} \nonumber\\
& \hspace{10mm}  = - \Lambda^{2-\gamma_{2}} w_{q+1}^{\bot} \nabla^{\bot} \cdot (z_{l} - z_{q}) + \Lambda^{2-\gamma_{2}} y_{l}^{\bot} \nabla^{\bot} \cdot (z_{q} - z_{l}), \nonumber \\ 
& \Lambda^{2-\gamma_{2}} (z_{q+1} - z_{q}) \cdot \nabla z_{q+1} + (\nabla z_{q+1})^{T} \cdot \Lambda^{2-\gamma_{2}} (z_{q} - z_{q+1}) = \Lambda^{2-\gamma_{2}} (z_{q+1} - z_{q})^{\bot} \nabla^{\bot} \cdot z_{q+1},  \nonumber \\
& \Lambda^{2-\gamma_{2}} (z_{q} - z_{l}) \cdot \nabla z_{q+1} + (\nabla z_{q+1})^{T} \cdot \Lambda^{2-\gamma_{2}} (z_{l} - z_{q}) =  \Lambda^{2-\gamma_{2}} (z_{q} - z_{l})^{\bot} \nabla^{\bot} \cdot z_{q+1},  \nonumber \\
& \Lambda^{2-\gamma_{2}} z_{l} \cdot \nabla (z_{q+1} - z_{q}) + (\nabla (z_{q} - z_{q+1}))^{T}\cdot \Lambda^{2-\gamma_{2}} z_{l} = \Lambda^{2-\gamma_{2}} z_{l}^{\bot} \nabla^{\bot} \cdot (z_{q+1} - z_{q}),  \nonumber \\
&\Lambda^{2-\gamma_{2}} z_{l} \cdot \nabla (z_{q} - z_{l}) + (\nabla (z_{l} - z_{q}))^{T} \cdot \Lambda^{2-\gamma_{2}} z_{l} = \Lambda^{2-\gamma_{2}} z_{l}^{\bot} \nabla^{\bot} \cdot (z_{q} - z_{l}). \label{est 412}
\end{align} 
Applying \eqref{est 185}-\eqref{est 412} to \eqref{est 189} allows us to deduce 
\begin{equation}\label{est 190} 
\divergence \mathring{R}_{q+1} = \divergence (R_{T} + R_{N} + R_{L} + R_{O} + R_{\text{Com1}} + R_{\text{Com2}})
\end{equation} 
where in addition to $R_{\text{Com1}}$ in \eqref{est 181} we defined  
\begin{subequations}\label{est 191} 
\begin{align}
 \divergence R_{T} \triangleq& \partial_{t} w_{q+1} + (\Lambda^{2- \gamma_{2}} y_{l} + \Lambda^{2-\gamma_{2}} z_{l}) \cdot \nabla w_{q+1},  \label{est 191a} \\
\divergence R_{N} \triangleq& ( \nabla ( \Lambda^{2- \gamma_{2}} (y_{l} + z_{q} ))^{T} \cdot w_{q+1} + (\Lambda^{2- \gamma_{2}} w_{q+1} \cdot \nabla) y_{l} - (\nabla y_{l})^{T} \cdot \Lambda^{2- \gamma_{2}} w_{q+1}, \label{est 191c} \\
\divergence R_{L} \triangleq& \Lambda^{\gamma_{1}} w_{q+1} + \Lambda^{2- \gamma_{2}} w_{q+1}^{\bot} (\nabla^{\bot} \cdot z_{l}), \label{est 191d} \\
 \divergence R_{O} \triangleq& \divergence \mathring{R}_{l} + (\Lambda^{2- \gamma_{2}} w_{q+1} \cdot \nabla) w_{q+1} - (\nabla w_{q+1})^{T} \cdot \Lambda^{2- \gamma_{2}} w_{q+1}, \label{est 191b} \\
\divergence R_{\text{Com2}} \triangleq&  \Lambda^{2-\gamma_{2}} y_{q+1} \cdot \nabla (z_{q+1} - z_{q}) + \Lambda^{2-\gamma_{2}} (z_{q+1} - z_{q}) \cdot \nabla y_{q+1}+ (\nabla y_{l})^{T} \cdot \Lambda^{2-\gamma_{2}} (z_{q} - z_{q+1}) \nonumber \\
& \hspace{4mm} + (\nabla (z_{q} - z_{q+1}))^{T} \cdot \Lambda^{2-\gamma_{2}} y_{l} + (\nabla \Lambda^{2-\gamma_{2}} (z_{q+1} - z_{q}))^{T} \cdot w_{q+1} \nonumber \\
&\hspace{4mm} -\Lambda^{2-\gamma_{2}} w_{q+1}^{\bot} \nabla^{\bot} \cdot (z_{l} - z_{q}) + (\nabla (z_{q} - z_{q+1}))^{T} \cdot \Lambda^{2-\gamma_{2}} w_{q+1}   \nonumber \\
&\hspace{4mm} + \Lambda^{2-\gamma_{2}} y_{l}^{\bot} \nabla^{\bot} \cdot (z_{q} - z_{l}) + \Lambda^{2-\gamma_{2}} (z_{q} - z_{l}) \cdot \nabla w_{q+1} + \Lambda^{2-\gamma_{2}} (z_{q} - z_{l})^{\bot} \nabla^{\bot} \cdot y_{l} \nonumber \\
&\hspace{4mm} + \Lambda^{2-\gamma_{2}} (z_{q+1} - z_{q})^{\bot} \nabla^{\bot} \cdot z_{q+1} + \Lambda^{2-\gamma_{2}} (z_{q} - z_{l})^{\bot}\nabla^{\bot} \cdot z_{q+1} \nonumber \\
&\hspace{4mm} + \Lambda^{2-\gamma_{2}} z_{l}^{\bot} \nabla^{\bot} \cdot (z_{q+1} - z_{q}) + \Lambda^{2-\gamma_{2}} z_{l}^{\bot} \nabla^{\bot} \cdot (z_{q} - z_{l}), \label{est 191e} 
\end{align}
\end{subequations}
representing the Reynolds stress errors respectively of transport, Nash, linear, oscillation, and second commutator types, if we also define 
\begin{equation}\label{est 191f} 
p_{q+1} \triangleq  p_{l} + w_{q+1} \cdot \Lambda^{2- \gamma_{2}} (y_{l} + z_{q+1}). 
\end{equation}  
Next, we will consider two special cases of the following transport equation 
\begin{subequations}\label{est 402} 
\begin{align}
& \partial_{t} f + (u\cdot\nabla) f = g, \\
& f(t_{0}, x) = f_{0}(x), 
\end{align}
\end{subequations}
where $u(t,x)$ is a given smooth vector field, for which if we let $\Phi$ be the inverse of the flux $X$ of $u$ starting at time $t_{0}$ as the identity
\begin{subequations}\label{flux}
\begin{align}
& \partial_{t} X = u(X, t), \\
&X(t_{0}, x) = x, 
\end{align}
\end{subequations}
then one of the consequences is that $f(t,x) = f_{0} (\Phi(t,x))$ in case $g \equiv 0$. We defer more consequences to Lemma \ref{Lemma 6.4}. Now for all $j \in \{0, \hdots, \lceil \tau_{q+1}^{-1} T_{L} \rceil\}$ we define $\Phi_{j}(t,x)$ as a special case of \eqref{est 402} that solves 
\begin{subequations}\label{est 193} 
\begin{align}
&( \partial_{t} + (\Lambda^{2- \gamma_{2}} y_{l} + \Lambda^{2- \gamma_{2}} z_{l}) \cdot \nabla) \Phi_{j} =0, \label{est 193a}\\
&\Phi_{j} (\tau_{q+1}j, x) = x, \label{est 193b}
\end{align}
\end{subequations}  
similarly to \cite[equation (68)]{Y21c} (also \cite[equation (43)]{RS21}). With $\Gamma_{1}$ and $\Gamma_{2}$ from Lemma \ref{Geometric Lemma}, we define $\Gamma_{j}$ to be $\Gamma_{1}$ and $\Gamma_{2}$ when $j$ is odd and even, respectively. With $\mathbb{P}_{q+1, k}$ from \eqref{est 192}  and $a_{k,j}$ to be defined subsequently in \eqref{est 196}, we define, recalling the notations from \eqref{est 401}, 
\begin{equation}\label{est 194} 
w_{q+1}(t,x) \triangleq \sum_{j,k} \chi_{j}(t) \mathbb{P}_{q+1,k}(a_{k,j} (t,x) b_{k} (\lambda_{q+1} \Phi_{j}(t,x))), 
\end{equation} 
where thanks to $\mathbb{P}_{q+1, k}$, we see that $w_{q+1}$ has the frequency support of \eqref{est 163}, which also implies directly from \eqref{est 189} that $\supp \hat{ \mathring{R}}_{q+1} \subset  B(0, 4 \lambda_{q+1})$ as desired. Next, as the second special case of \eqref{est 402}, we define $\mathring{R}_{q,j}$ to be the solution to 
\begin{subequations}\label{est 195} 
\begin{align}
& ( \partial_{t} + ( \Lambda^{2-\gamma_{2}} y_{l} + \Lambda^{2- \gamma_{2}} z_{l}) \cdot \nabla) \mathring{R}_{q,j} = 0, \label{est 195a}\\
& \mathring{R}_{q,j} (\tau_{q+1}j, x) = \mathring{R}_{l} (\tau_{q+1}j, x). \label{est 195b}
\end{align}
\end{subequations} 
Moreover, for $t \geq \tau_{q+1} j$ and $\gamma_{k}$ for $k \in \Gamma_{j} \subset \mathbb{S}^{1}$ from Lemma \ref{Geometric Lemma} we define 
\begin{equation}\label{est 196} 
a_{k,j}(t,x) \triangleq \frac{\sqrt{\bar{C}} L^{\frac{1}{2}} M_{0} (\tau_{q+1} j)^{\frac{1}{2}}}{\sqrt{\epsilon_{\gamma}} \sqrt{\gamma_{2}}}\delta_{q+1}^{\frac{1}{2}} \gamma_{k} \left( \Id - \frac{ \epsilon_{\gamma}\mathring{R}_{q,j}(t,x)}{\bar{C} \lambda_{q+1}^{2- \gamma_{2}} \delta_{q+1} LM_{0} (\tau_{q+1} j)} \right) 
\end{equation} 
where $\gamma_{k} \left( \Id - \frac{\epsilon_{\gamma} \mathring{R}_{q,j}(t,x)}{\bar{C} \lambda_{q+1}^{2- \gamma_{2}} \delta_{q+1} LM_{0} (\tau_{q+1} j)} \right)$ is well-defined because 
\begin{equation}\label{est 384}
\left\lvert \frac{ \epsilon_{\gamma} \mathring{R}_{q,j}(t,x)}{\bar{C} \lambda_{q+1}^{2- \gamma_{2}} \delta_{q+1} LM_{0}(\tau_{q+1} j)} \right\rvert \overset{\eqref{est 197a} \eqref{est 195b}}{\leq} \frac{ \epsilon_{\gamma} \lVert \mathring{R}_{l}(\tau_{q+1} j) \rVert_{C_{x}}}{\bar{C} \lambda_{q+1}^{2- \gamma_{2}} \delta_{q+1} LM_{0}(\tau_{q+1} j)} \overset{\eqref{est 167}}{\leq}\epsilon_{\gamma}
\end{equation} 
so that $\Id - \frac{ \epsilon_{\gamma}\mathring{R}_{q,j}(t,x)}{\bar{C} \lambda_{q+1}^{2- \gamma_{2}} \delta_{q+1} M_{0} (\tau_{q+1} j)}  \in B(\Id, \epsilon_{\gamma})$; here we used the facts that $\varphi_{l}$ is supported in $(\tau_{q+1}, 2\tau_{q+1}] \subset \mathbb{R}_{+}$ and $M_{0}(t)$ from \eqref{est 135}  is non-decreasing. We recall $\chi_{j}(t)$ from \eqref{est 209},  $b_{k}$ and $c_{k}$ from \eqref{est 179} and define 
\begin{subequations}\label{est 200} 
\begin{align}
& \tilde{w}_{q+1, j, k} (t,x) \triangleq \chi_{j}(t) a_{k,j}(t,x) b_{k} (\lambda_{q+1} \Phi_{j}(t,x)), \label{est 200a}\\
& \psi_{q+1, j, k} (t,x) \triangleq \frac{c_{k} (\lambda_{q+1} \Phi_{j} (t,x))}{c_{k} (\lambda_{q+1} x)} \overset{\eqref{est 178} }{=} e^{i \lambda_{q+1} (\Phi_{j} (t,x) - x) \cdot k}, \label{est 200b} 
\end{align}
\end{subequations} 
so that due to \eqref{est 194} and \eqref{est 200a} 
\begin{equation}\label{est 208}
w_{q+1} = \sum_{j, k} \mathbb{P}_{q+1, k} \tilde{w}_{q+1, j, k} \hspace{2mm} \text{ and } \hspace{2mm} b_{k} (\lambda_{q+1} \Phi_{j} (x)) = b_{k} (\lambda_{q+1} x) \psi_{q+1, j, k} (x).
\end{equation} 
In Propositions \ref{Proposition 4.9}-\ref{Proposition 4.10}, we collect some estimates, similarly to \cite[Lemmas 4.3-4.4]{BSV19}.  
\begin{proposition}\label{Proposition 4.9}
Define 
\begin{equation}\label{est 206}
D_{t,q} \triangleq \partial_{t} + (\Lambda^{2-\gamma_{2}} y_{l} + \Lambda^{2- \gamma_{2}} z_{l}) \cdot \nabla. 
\end{equation}
Then $w_{q+1}$ defined in \eqref{est 194} satisfies the following inequalities: for $C_{1}$ from \eqref{est 202},
\begin{subequations}\label{est 204} 
\begin{align}
&\lVert w_{q+1}(t) \rVert_{C_{x}} \leq 2\sup_{k\in \Gamma_{1} \cup \Gamma_{2}} \lVert \gamma_{k} \rVert_{C(B(\Id, \epsilon_{\gamma} ))} \frac{ \sqrt{\bar{C}}}{\sqrt{\epsilon_{\gamma}}} \frac{C_{1} L^{\frac{1}{2}} M_{0}(t)^{\frac{1}{2}}}{\sqrt{\gamma_{2}}} \delta_{q+1}^{\frac{1}{2}},\label{est 204a} \\
&\lVert D_{t,q} w_{q+1} (t) \rVert_{C_{x}}  \lesssim L M_{0}(t) \delta_{q+1}^{\frac{1}{2}}\tau_{q+1}^{-1}. \label{est 204b}
\end{align}
\end{subequations} 
Consequently, \eqref{est 162}, as well as all of \eqref{est 165}, \eqref{est 166}, and \eqref{est 168} at level $q+1$ and hence Hypothesis \ref{Hypothesis 4.1} \ref{Hypothesis 4.1 (a)} and \ref{Hypothesis 4.1 (c)} at level $q+1$ hold.  
\end{proposition}

\begin{proof}[Proof of Proposition \ref{Proposition 4.9}]
First, we can prove \eqref{est 204a} from \eqref{est 194} as 
\begin{equation}\label{est 205} 
\lVert w_{q+1}(t) \rVert_{C_{x}} \overset{\eqref{est 202} \eqref{est 196}}{\leq} \sup_{k \in \Gamma_{1} \cup \Gamma_{2}} \lVert \gamma_{k} \rVert_{C(B(\Id, \epsilon_{\gamma}))} \frac{ \sqrt{\bar{C}}}{\sqrt{\epsilon_{\gamma}}} \sum_{j} \frac{1_{\supp \chi_{j}}(t) C_{1} L^{\frac{1}{2}}M_{0}(\tau_{q+1}j)^{\frac{1}{2}}}{\sqrt{\gamma_{2}}} \delta_{q+1}^{\frac{1}{2}} 
\end{equation} 
which implies \eqref{est 204a} considering the fact that for all $t$ there exist at most two non-trivial cutoffs. Its immediate consequence is \eqref{est 162} because  
\begin{align}
\lVert y_{q+1}(t) - y_{q}(t) \rVert_{C_{x}} \overset{\eqref{est 203} \eqref{est 205}}{\lesssim} L^{\frac{1}{2}} M_{0}(t)^{\frac{1}{2}} \delta_{q+1}^{\frac{1}{2}} + l L M_{0}(t) \lambda_{q}^{3- \gamma_{2}} \delta_{q}  \lesssim L^{\frac{1}{2}} M_{0}(t)^{\frac{1}{2}} \delta_{q+1}^{\frac{1}{2}} \label{138 star}
\end{align} 
where we took $b > 2 (L^{2} + 2)$, $a \in 5 \mathbb{N}$ sufficiently large and used that \eqref{est 154c} guarantees $\alpha > \frac{3- \gamma_{2}}{2}$. Next, by Young's inequality for convolution we can prove \eqref{est 165} at level $q+1$ as follows: 
\begin{equation}
\lVert y_{q+1} \rVert_{C_{t,x}} \overset{\eqref{est 164} }{\leq} \lVert y_{l} \rVert_{C_{t,x}} + \lVert w_{q+1} \rVert_{C_{t,x}}  \overset{\eqref{est 165}\eqref{est 204a} }{\leq} C_{0} \left(1+ \sum_{1 \leq j \leq q+1} \delta_{j}^{\frac{1}{2}} \right) L^{\frac{1}{2}} M_{0}(t)^{\frac{1}{2}}.
\end{equation} 
Similarly, because \eqref{est 154b} guarantees that $2- \gamma_{2} > \beta$, we can prove \eqref{est 166} at level $q+1$ as follows: for $C_{0} \geq 2^{4} \sup_{k \in \Gamma_{1} \cup \Gamma_{2}} \lVert \gamma_{k} \rVert_{C(B(\Id, \epsilon_{\gamma} ))} C_{1} \frac{ \sqrt{\bar{C}}}{\sqrt{\epsilon_{\gamma}}}$ and $a \in 5 \mathbb{N}$ sufficiently large   
\begin{align*}
\lVert y_{q+1} \rVert_{C_{t}C_{x}^{2- \gamma_{2}}} + &\lVert \Lambda^{2- \gamma_{2}} y_{q+1} \rVert_{C_{t,x}} \overset{\eqref{est 164}\eqref{est 163} \eqref{est 166} \eqref{est 204a}}{\leq}C_{0} L^{\frac{1}{2}} M_{0}(t)^{\frac{1}{2}} [ \frac{1}{2} \lambda_{q+1}^{2- \gamma_{2}} \delta_{q+1}^{\frac{1}{2}} + \lambda_{q}^{2- \gamma_{2}}\delta_{q}^{\frac{1}{2}}] \nonumber\\
& \hspace{48mm} \leq C_{0} L^{\frac{1}{2}} M_{0}(t)^{\frac{1}{2}} \lambda_{q+1}^{2- \gamma_{2}} \delta_{q+1}^{\frac{1}{2}}.
\end{align*} 
Next, $y_{q+1} = y_{l} + w_{q+1}$ by \eqref{est 164} so that \eqref{est 163}  and Hypothesis \ref{Hypothesis 4.1} \ref{Hypothesis 4.1 (a)} at level $q$ imply that $\supp \hat{y}_{q+1}  \subset B(0, 2 \lambda_{q+1})$; therefore, we conclude that the Hypothesis \ref{Hypothesis 4.1} \ref{Hypothesis 4.1 (a)} at level $q+1$ holds.  Next, it follows immediately from \eqref{est 206}, \eqref{est 193a}, and \eqref{est 195a} that 
\begin{equation}\label{est 207}
D_{t,q} a_{k,j}(t,x) = 0, \hspace{2mm} D_{t,q} b_{k}(\lambda_{q+1} \Phi_{j}(t,x) ) = 0, \hspace{2mm} D_{t,q} c_{k} (\lambda_{q+1} \Phi_{j}(t,x))  = 0. 
\end{equation} 
By \eqref{est 208}, \eqref{est 200a}, and \eqref{est 207} we obtain 
\begin{equation}\label{est 212}
 D_{t,q} w_{q+1} =\sum_{j,k}  [D_{t,q}, \mathbb{P}_{q+1, k} ] \tilde{w}_{q+1, j, k} + \mathbb{P}_{q+1,k} [\partial_{t} \chi_{j} a_{k,j} b_{k} (\lambda_{q+1} \Phi_{j} )].
\end{equation} 
We estimate from \eqref{est 212} using \cite[Corollary A.8]{BSV19} with \eqref{est 416} and the facts that \eqref{est 154b} and \eqref{est 154c} respectively guarantee that $3- \gamma_{2} - \beta > 0$ and $\alpha > 0$ so that $\lambda_{q}^{3- \gamma_{2}} \delta_{q}^{\frac{1}{2}} \lesssim \tau_{q+1}^{-1}$, for $a \in 5 \mathbb{N}$ sufficiently large 
\begin{align}
& \lVert D_{t,q} w_{q+1} (t) \rVert_{C_{x}} \label{est 213}\\
&\overset{\eqref{est 212} \eqref{est 202}}{\lesssim}  \sum_{j,k}   \lVert \nabla (\Lambda^{2- \gamma_{2}} y_{l} + \Lambda^{2- \gamma_{2}} z_{l})(t) \rVert_{C_{x}} \lVert \tilde{w}_{q+1, j,k}(t) \rVert_{C_{x}} + \lVert \partial_{t} \chi_{j}(t)a_{k,j}(t) b_{k} (\lambda_{q+1} \Phi_{j} (t)) \rVert_{C_{x}} \nonumber \\
& \hspace{25mm} \overset{\eqref{est 178} \eqref{est 196}  \eqref{est 136} \eqref{est 166}}{\lesssim} L M_{0}(t) \delta_{q+1}^{\frac{1}{2}} (\lambda_{q}^{3- \gamma_{2}} \delta_{q}^{\frac{1}{2}} + \tau_{q+1}^{-1}) \overset{\eqref{est 210}}{\lesssim} L M_{0}(t) \delta_{q+1}^{\frac{1}{2}}\tau_{q+1}^{-1}. \nonumber 
\end{align}
Next, we recall the definitions of $y_{l}$ and $z_{l}$ from \eqref{est 180} and write $D_{t,q+1}y_{q+1}$ as
\begin{align*}
D_{t, q+1} &y_{q+1} 
= D_{t,q} w_{q+1} + (\Lambda^{2- \gamma_{2}} y_{l} \ast_{t} \phi_{\lambda_{q+2}^{-\alpha}} \ast_{t} \varphi_{\lambda_{q+2}^{-\alpha}} - \Lambda^{2- \gamma_{2}} y_{l}) \cdot \nabla w_{q+1} \\
&+ (\Lambda^{2- \gamma_{2}} z_{q+1} \ast_{x} \phi_{\lambda_{q+2}^{-\alpha}} \ast_{t} \varphi_{\lambda_{q+2}^{-\alpha}} - \Lambda^{2- \gamma_{2}} z_{l}) \cdot \nabla w_{q+1}\nonumber  \\
&+ \Lambda^{2- \gamma_{2}} w_{q+1} \ast_{x} \phi_{\lambda_{q+2}^{-\alpha}} \ast_{t} \varphi_{\lambda_{q+2}^{-\alpha}} \cdot \nabla w_{q+1} + \Lambda^{2- \gamma_{2}} w_{q+1} \ast_{x} \phi_{\lambda_{q+2}^{-\alpha}} \ast_{t} \varphi_{\lambda_{q+2}^{-\alpha}} \cdot \nabla y_{l}\nonumber  \\
&+ [D_{t,q} y_{q}]\ast_{x} \phi_{\lambda_{q+1}^{-\alpha}} \ast_{t} \varphi_{\lambda_{q+1}^{-\alpha}} + \Lambda^{2-\gamma_{2}} y_{l} \ast_{x} \phi_{\lambda_{q+2}^{-\alpha}} \ast_{t} \varphi_{\lambda_{q+2}^{-\alpha}} \cdot \nabla y_{l}\nonumber \\
& - [\Lambda^{2 - \gamma_{2}} y_{q} \ast_{x} \phi_{\lambda_{q+1}^{-\alpha}} \ast_{t} \varphi_{\lambda_{q+1}^{-\alpha}} \cdot \nabla y_{q} ] \ast_{x} \phi_{\lambda_{q+1}^{-\alpha}} \ast_{t} \varphi_{\lambda_{q+1}^{-\alpha}}\nonumber  \\
&+ \Lambda^{2- \gamma_{2}} z_{q+1} \ast_{x} \phi_{\lambda_{q+2}^{-\alpha}} \ast_{t} \varphi_{\lambda_{q+2}^{-\alpha}} \cdot \nabla y_{l} - [\Lambda^{2- \gamma_{2}} z_{q} \ast_{x} \phi_{\lambda_{q+1}^{-\alpha}} \ast_{t} \varphi_{\lambda_{q+1}^{-\alpha}} \cdot \nabla y_{q} ] \ast_{x} \phi_{\lambda_{q+1}^{-\alpha}} \ast_{t} \varphi_{\lambda_{q+1}^{-\alpha}} \nonumber 
\end{align*}
due to \eqref{est 206} and \eqref{est 164} so that we can prove \eqref{est 168} at level $q+1$ as follows: for $a \in 5 \mathbb{N}$ sufficiently large 
\begin{align*}
&\lVert D_{t, q+1} y_{q+1} \rVert_{C_{t,x}} \\ 
&\lesssim \lVert D_{t,q} w_{q+1} \rVert_{C_{t,x}} + \lVert \Lambda^{2-\gamma_{2}} y_{l} \rVert_{C_{t,x}} \lVert \nabla w_{q+1} \rVert_{C_{t,x}} + \lVert \Lambda^{2- \gamma_{2}} z \rVert_{C_{t,x}} \lVert \nabla w_{q+1} \rVert_{C_{t,x}} \\
&+ \lVert\Lambda^{2- \gamma_{2}} w_{q+1} \rVert_{C_{t,x}} \lVert \nabla w_{q+1} \rVert_{C_{t,x}} + \lVert \Lambda^{2- \gamma_{2}} w_{q+1} \rVert_{C_{t,x}} \lVert \nabla y_{q} \rVert_{C_{t,x}} \\
&+ \lVert D_{t,q} y_{q} \rVert_{C_{t,x}} + \lVert \Lambda^{2- \gamma_{2}} y_{q} \rVert_{C_{t,x}} \lVert \nabla y_{q} \rVert_{C_{t,x}} + \lVert \Lambda^{2- \gamma_{2}} y_{q} \rVert_{C_{t,x}} \lVert \nabla y_{q} \rVert_{C_{t,x}} + \lVert \Lambda^{2- \gamma_{2}} z \rVert_{C_{t,x}} \lVert \nabla y_{q} \rVert_{C_{t,x}}  \\
&\overset{\eqref{est 136} \eqref{est 163}\eqref{est 204} \eqref{est 166} \eqref{est 168} \eqref{est 210} \eqref{est 32}}{\lesssim} LM_{0}(t) \lambda_{q+1}^{3- \gamma_{2}} \delta_{q+1} \\
& \hspace{15mm} \times [ a^{b^{q+1} [-\frac{3}{2} + \frac{\gamma_{2}}{2} + \frac{\beta}{2} + \frac{\alpha}{2} ]} + a^{b^{q} [(b-1)(-2+ \gamma_{2} + \beta)]} + 1 + a^{b^{q} [(b-1)(-1 +\beta)]}] \lesssim LM_{0}(t) \lambda_{q+1}^{3- \gamma_{2}} \delta_{q+1} 
\end{align*}
where we used that $\alpha < 3 - \gamma_{2} - \beta$ by \eqref{est 154c} while $\beta < \min\{\frac{3-\gamma_{2}}{2}, 2- \gamma_{2}, 1 \}$ by \eqref{est 154b}. 
\end{proof}

\begin{proposition}\label{Proposition 4.10}   
Let $a_{k,j}$ and $\psi_{q+1, j,k}$ be defined by \eqref{est 196} and \eqref{est 200b}, respectively. Then they satisfy the following inequalities for all $t \in \supp \chi_{j}$:   
\begin{subequations}\label{est 216} 
\begin{align}
& \lVert D^{N} a_{k,j} (t) \rVert_{C_{x}} \lesssim L^{\frac{1}{2}}M_{0}(t)^{\frac{1}{2}}  \lambda_{q}^{N} \delta_{q+1}^{\frac{1}{2}} \hspace{1mm} \forall \hspace{1mm} N \in \mathbb{N}_{0}, \label{est 216a}\\
&\lVert \psi_{q+1, j, k} \rVert_{C_{t,x}} \leq 1, \hspace{2mm} \lVert D^{N} \psi_{q+1, j, k} (t) \rVert_{C_{x}} \lesssim L^{\frac{1}{2}} M_{0}(t)^{\frac{1}{2}} \lambda_{q+1} \tau_{q+1} \lambda_{q}^{N+ 2 - \gamma_{2}}  \delta_{q}^{\frac{1}{2}} \hspace{1mm} \forall \hspace{1mm} N \in \mathbb{N}. \label{est 216b}
\end{align}
\end{subequations}
\end{proposition}  

\begin{proof}[Proof of Proposition \ref{Proposition 4.10}]
First, immediately by definition from \eqref{est 196} we can estimate 
\begin{equation}
\lVert a_{k,j}(t) \rVert_{C_{x}} \lesssim L^{\frac{1}{2}}  M_{0}(t)^{\frac{1}{2}} \delta_{q+1}^{\frac{1}{2}}.
\end{equation} 
Next, for $N \in \mathbb{N}$, by standard chain rule estimates (e.g., \cite[equation (130)]{BDIS15}),
\begin{align}
& \lVert D^{N} a_{k,j}(t) \rVert_{C_{x}} \lesssim \delta_{q+1}^{\frac{1}{2}} L^{\frac{1}{2}}  M_{0}(\tau_{q+1} j)^{\frac{1}{2}}  \label{est 217}\\
&\hspace{15mm} \times \left[ \frac{1}{\lambda_{q+1}^{2-\gamma_{2}} \delta_{q+1} L M_{0}(\tau_{q+1} j)} \lVert D^{N} \mathring{R}_{q,j} \rVert_{C_{t,x}} + \left( \frac{1}{\lambda_{q+1}^{2-\gamma_{2}} \delta_{q+1} L M_{0}(\tau_{q+1} j)} \right)^{N} \lVert D\mathring{R}_{q,j} \rVert_{C_{t,x}}^{N} \right]. \nonumber 
\end{align}
Because \eqref{est 154b}-\eqref{est 154c} guarantee that $3 - \gamma_{2} > \beta$, we can estimate for $b> \frac{8}{7} (L^{2} + \frac{3}{2})$ and $a \in 5 \mathbb{N}$ sufficiently large  
\begin{align}
\tau_{q+1} \lVert D (\Lambda^{2- \gamma_{2}} y_{l} + \Lambda^{2- \gamma_{2}} z_{l}) \rVert_{C_{t,x}} \overset{\eqref{est 136} \eqref{est 166}\eqref{est 210}\eqref{est 32}}{\lesssim} a^{L^{2}} a^{b^{q} [b( \frac{ - \alpha - 3 + \gamma_{2} + \beta}{2}) + 3 - \gamma_{2} - \beta]}  \lesssim 1, \label{est 201} 
\end{align}
and consequently in case $N = 1$ we can estimate by \eqref{est 197b} and \eqref{est 195}
\begin{align}
& \lVert D \mathring{R}_{q,j}(t) \rVert_{C_{x}} \leq \lVert D \mathring{R}_{l} (\tau_{q+1} j) \rVert_{C_{x}}  e^{(t- \tau_{q+1} j) \lVert D(\Lambda^{2- \gamma_{2}} y_{l} + \Lambda^{2- \gamma_{2}} z_{l}) \rVert_{C_{t,x}}} \nonumber\\
& \hspace{35mm} \overset{\eqref{est 201} \eqref{est 167}}{\lesssim} \lambda_{q} L M_{0}(\tau_{q+1} j) \lambda_{q+1}^{2 - \gamma_{2}} \delta_{q+1}, \label{est 214} 
\end{align} 
and similarly in case $N \in \mathbb{N}\setminus \{1\}$ by relying instead on \eqref{est 198}, for $b> 2L^{2}$ 
\begin{align}
&\lVert D^{N} \mathring{R}_{q,j}(t)  \rVert_{C_{x}} \lesssim ( \lVert D^{N} \mathring{R}_{l} (\tau_{q+1} j) \rVert_{C_{x}} + (t -\tau_{q+1} j) \lVert D^{N} (\Lambda^{2- \gamma_{2}} y_{l} + \Lambda^{2- \gamma_{2}} z_{l}) \rVert_{C_{x}} \lVert D \mathring{R}_{l} (\tau_{q+1} j) \rVert_{C_{x}}) \nonumber \\
& \hspace{12mm} \times e^{( t-  \tau_{q+1} j) \lVert D ( \Lambda^{2- \gamma_{2}} y_{l} + \Lambda^{2- \gamma_{2}} z_{l}) \rVert_{C_{t,x}}} \overset{\eqref{est 201} \eqref{est 166} \eqref{est 136} \eqref{est 167}}{\lesssim}  \lambda_{q}^{N} L M_{0}(\tau_{q+1} j) \lambda_{q+1}^{2- \gamma_{2}} \delta_{q+1}. \label{est 215}
\end{align}  
Applying \eqref{est 214}-\eqref{est 215} to \eqref{est 217} gives us \eqref{est 216a}. Next, concerning \eqref{est 216b}, its first inequality is clear from \eqref{est 200b}. Moreover, by \cite[equation (130)]{BDIS15} again we have 
\begin{align}
\lVert D \psi_{q+1, j, k} (t) \rVert_{C_{x}} 
\overset{\eqref{est 200b}}{\lesssim}& \lambda_{q+1} \lVert \nabla \Phi_{j}(t) -\Id \rVert_{C_{x}} \nonumber\\
& \overset{\eqref{est 199a}\eqref{est 136}\eqref{est 201}\eqref{est 166} }{\lesssim} \lambda_{q+1} \tau_{q+1} \lambda_{q}^{3- \gamma_{2}} \delta_{q}^{\frac{1}{2}}L^{\frac{1}{2}} M_{0}(t)^{\frac{1}{2}}. \label{est 222}
\end{align}
In case $N \in \mathbb{N}\setminus \{1\}$, we estimate again via \cite[equation (130)]{BDIS15} from \eqref{est 200b} for 
\begin{align*}
b > 
\begin{cases}
2 + 4L^{2} & \text{ if } \gamma_{2} = 1, \\
2+ \frac{2L^{2}}{1- \frac{\gamma_{2}}{2}} & \text{ if } \gamma_{2} > 1, 
\end{cases}
\end{align*}
and $a \in 5 \mathbb{N}$ sufficiently large 
\begin{align*}
& \lVert D^{N} \psi_{q+1, j, k} (t) \rVert_{C_{x}} \overset{\eqref{est 200b}}{\lesssim}   \lambda_{q+1} \lVert \Phi_{j}(t)\rVert_{C_{x}^{N}} + \lambda_{q+1}^{N} \lVert \nabla \Phi_{j}(t) -\Id \rVert_{C_{x}}^{N} 
 \\
&\hspace{10mm} \overset{\eqref{est 199} \eqref{est 201} \eqref{est 136}\eqref{est 166}}{\lesssim} \lambda_{q+1} \tau_{q+1} \lambda_{q}^{N+ 2 - \gamma_{2}} \delta_{q}^{\frac{1}{2}}L^{\frac{1}{2}} M_{0}(t)^{\frac{1}{2}} + \lambda_{q+1}^{N} ( \tau_{q+1} \lambda_{q}^{3- \gamma_{2}}  \delta_{q}^{\frac{1}{2}} L^{\frac{1}{2}} M_{0}(t)^{\frac{1}{2}})^{N} \\
& \hspace{23mm} \lesssim \lambda_{q+1} \tau_{q+1} \lambda_{q}^{N+ 2 - \gamma_{2}}   \delta_{q}^{\frac{1}{2}}L^{\frac{1}{2}} M_{0}(t)^{\frac{1}{2}}  
\end{align*}
where the last inequality used the facts that $e^{4} \leq a$ and $\alpha > 1$ from \eqref{est 154c} so that 
\begin{equation}\label{est 223}
 \lambda_{q+1} \tau_{q+1} \lambda_{q}^{2 - \gamma_{2}}   \delta_{q}^{\frac{1}{2}}  L^{\frac{1}{2}} M_{0}(t)^{\frac{1}{2}} \overset{\eqref{est 210} \eqref{est 32}}{\lesssim} a^{b^{q} [(b-2) ( \frac{- \alpha - 1 + \gamma_{2} + \beta}{2}) -\alpha + 1 ]} a^{L^{2}}\ll 1. 
\end{equation}  
\end{proof}

Before we proceed, let us observe an immediate useful corollary of \eqref{est 216} and \eqref{est 223}: for all $t \in \supp \chi_{j}$, 
\begin{align}  
& \lVert   \nabla (a_{k,j} \psi_{q+1, j, k})(t) \rVert_{C_{x}}  \lesssim \lVert  \nabla a_{k,j} \rVert_{C_{t,x}} \lVert \psi_{q+1, j,k} \rVert_{C_{t,x}} + \lVert a_{k,j} \rVert_{C_{t,x}} \lVert \nabla \psi_{q+1, j,k} \rVert_{C_{t,x}} \label{est 268} \\
& \hspace{17mm} \overset{\eqref{est 216}}{\lesssim}  \lambda_{q} \delta_{q+1}^{\frac{1}{2}} L^{\frac{1}{2}} M_{0}(t)^{\frac{1}{2}} [ 1 + L^{\frac{1}{2}} M_{0}(t)^{\frac{1}{2}} \lambda_{q+1} \tau_{q+1} \lambda_{q}^{2- \gamma_{2}}  \delta_{q}^{\frac{1}{2}} ] \overset{\eqref{est 223}}{\lesssim} \lambda_{q} \delta_{q+1}^{\frac{1}{2}} L^{\frac{1}{2}} M_{0}(t)^{\frac{1}{2}}. \nonumber 
\end{align}

\subsubsection{Bounds on $R_{T}$, the transport error}\label{Subsection 4.1.1}
Using \eqref{est 200a}, \eqref{est 208}-\eqref{est 206}, and \eqref{est 207}, we can rewrite $R_{T}$ from \eqref{est 191a} as
\begin{align}
R_{T} =& \sum_{j,k} 1_{\supp \chi_{j}}(t) \mathcal{B} ( [(\Lambda^{2- \gamma_{2}} y_{l} + \Lambda^{2- \gamma_{2}} z_{l}) \cdot \nabla, \mathbb{P}_{q+1,k} ] \tilde{w}_{q+1, j, k}  \nonumber \\
& \hspace{35mm}+ \mathbb{P}_{q+1, k} (\partial_{t} \chi_{j} a_{k,j} b_{k} (\lambda_{q+1} \Phi_{j} ))). \label{est 218}
\end{align}
Next, by definitions of $\mathbb{P}_{q+1, k}$ and $\tilde{P}_{\approx \lambda_{q+1}}$ from Subsection \ref{Subsection 3.2}, we see that for $a \in 5 \mathbb{N}$ such that $a \geq \frac{16}{5}$ we can rewrite furthermore as 
\begin{align}
R_{T} =& \sum_{j,k} 1_{\supp \chi_{j}}(t) \mathcal{B} \tilde{P}_{\approx \lambda_{q+1}} ( [(\Lambda^{2- \gamma_{2}} y_{l} + \Lambda^{2- \gamma_{2}} z_{l}) \cdot \nabla, \mathbb{P}_{q+1,k} ] \tilde{w}_{q+1, j, k}  \nonumber \\
& \hspace{35mm}+ \mathbb{P}_{q+1, k} (\partial_{t} \chi_{j} a_{k,j} b_{k} (\lambda_{q+1} \Phi_{j} ))). \label{est 219}
\end{align} 
Thus, we estimate by \cite[Lemma A.6]{BSV19} and using the fact that \eqref{est 154b} and \eqref{est 154c} guarantee that $\beta < \frac{7-3\gamma_{2}}{5}$ and $\alpha < 7 - 3 \gamma_{2} - 5 \beta$, for $a \in 5 \mathbb{N}$ sufficiently large 
\begin{align}
\lVert R_{T} &\rVert_{C_{t,x}} \overset{\eqref{est 202} \eqref{est 216a}}{\lesssim} \lambda_{q+1}^{-1} \sum_{j,k} \nonumber  \\
& \times [ ( \lVert \nabla \Lambda^{2- \gamma_{2}} y_{l} \rVert_{C_{t,x}} + \lVert \nabla \Lambda^{2- \gamma_{2}} z_{l} \rVert_{C_{t,x}}) \lVert 1_{\supp \chi_{j}} \tilde{w}_{q+1, j, k} \rVert_{C_{t,x}} + \tau_{q+1}^{-1} \delta_{q+1}^{\frac{1}{2}} L^{\frac{1}{2}} M_{0}(t)^{\frac{1}{2}} ] \nonumber \\
&\overset{\eqref{est 136} \eqref{est 166} \eqref{est 216a}}{\lesssim} \lambda_{q+1}^{-1} M_{0}(t)^{\frac{1}{2}} [ \lambda_{q}^{3-\gamma_{2}} L M_{0}(t)^{\frac{1}{2}} \delta_{q}^{\frac{1}{2}} \delta_{q+1}^{\frac{1}{2}} + L^{\frac{1}{2}} \tau_{q+1}^{-1} \delta_{q+1}^{\frac{1}{2}} ]  \nonumber\\
&\overset{ \eqref{est 210}\eqref{est 32}}{\lesssim}  L M_{0}(t) \lambda_{q+2}^{2-\gamma_{2}} \delta_{q+2} [ a^{b^{q} [(b-2) [(b+2) (-2 + \gamma_{2} + 2 \beta) - 1 - \beta] - 7 + 3 \gamma_{2} + 5 \beta]} \nonumber\\
& \hspace{33mm} + a^{b^{q+1} [(b-2) (-2+ \gamma_{2} + 2\beta) + \frac{ -7 + 3 \gamma_{2} + 5 \beta + \alpha}{2} ]}] \ll L M_{0}(t) \lambda_{q+2}^{2-\gamma_{2}} \delta_{q+2}. \label{est 301}
\end{align}

\subsubsection{Bounds on $R_{N}$, the Nash error}\label{Subsection 4.1.2}
Let us write from \eqref{est 191c} using \eqref{est 38}   
\begin{equation}\label{est 304}
R_{N} = N_{1} + N_{2} 
\end{equation} 
where 
\begin{equation}\label{est 221} 
N_{1} \triangleq \mathcal{B} ( ( \nabla (\Lambda^{2-\gamma_{2}} (y_{l} + z_{q} ))^{T} \cdot w_{q+1}) \hspace{2mm} \text{ and } \hspace{2mm} N_{2} \triangleq \mathcal{B} ( \Lambda^{2-\gamma_{2}} w_{q+1}^{\bot} (\nabla^{\bot} \cdot y_{l})).
\end{equation} 
Considering that $\supp  \left((\nabla (\Lambda^{2- \gamma_{2}} (y_{l} + z_{q} )))^{T}\right)\hat{} \subset B(0, 2 \lambda_{q})$ and definition of $\mathbb{P}_{q+1, k}$ within $w_{q+1}$, we can rewrite similarly to \eqref{est 219}
\begin{equation}\label{est 220}
N_{1} =  \sum_{j,k}  \mathcal{B} \tilde{P}_{\approx \lambda_{q+1}} [ ( \nabla (\Lambda^{2-\gamma_{2}} y_{l} + z_{q}))^{T} \cdot \mathbb{P}_{q+1, k} (\chi_{j} a_{k,j}  b_{k} (\lambda_{q+1} \Phi_{j}))]. 
\end{equation} 
This allows us to immediately estimate for $a \in 5 \mathbb{N}$ sufficiently large 
\begin{align}
 &\lVert N_{1} \rVert_{C_{t,x}} 
\lesssim \lambda_{q+1}^{-1} \sum_{j,k}  ( \lVert \nabla \Lambda^{2- \gamma_{2}} y_{l} \rVert_{C_{t,x}} + \lVert \nabla \Lambda^{2-\gamma_{2}} z_{q} \rVert_{C_{t,x}}) \lVert 1_{\supp \chi_{j}} a_{k,j} \rVert_{C_{t,x}} \label{est 302}\\
& \hspace{2mm} \overset{\eqref{est 166} \eqref{est 136} \eqref{est 216a}}{\lesssim} L M_{0}(t) \lambda_{q+2}^{2 - \gamma_{2}} \delta_{q+2} a^{b^{q} [(b-2)[(b+2) (-2+ \gamma_{2} + 2 \beta) - 1 - \beta] - 7 + 3 \gamma_{2} + 5 \beta]} \ll L M_{0}(t) \lambda_{q+2}^{2-\gamma_{2}} \delta_{q+2} \nonumber 
\end{align}
where we used that $\beta < \frac{7-3\gamma_{2}}{5}$ due to \eqref{est 154b}. Next, to handle $N_{2}$ we first rewrite using \eqref{est 208}, \eqref{est 178}, and \eqref{est 200} 
\begin{align}
w_{q+1}^{\bot}(t,x) =& - \lambda_{q+1}^{-1} \sum_{j,k} \nabla \mathbb{P}_{q+1, k} ( \chi_{j}(t) a_{k,j} (t,x) \psi_{q+1, j, k} (t,x) c_{k} (\lambda_{q+1} x) ) \nonumber \\
&+ \lambda_{q+1}^{-1} \sum_{j,k}  \mathbb{P}_{q+1, k} ( \chi_{j}(t) \nabla (a_{k,j} (t,x) \psi_{q+1, j, k} (t,x)) c_{k} (\lambda_{q+1} x)).  \label{est 367}
\end{align}
Applying \eqref{est 367} to $N_{2}$ in \eqref{est 221} and relying on \eqref{est 200b}, as well as the fact that $\mathcal{B} f = \mathcal{B} \mathbb{P} f$ for any $f$ that is not divergence-free and $\mathbb{P}$ eliminates any gradient give us 
\begin{align}
N_{2}(t,x) =&   \lambda_{q+1}^{-1} \mathcal{B} ( \nabla (\nabla^{\bot} \cdot y_{l})(t,x)   \sum_{j,k}   \Lambda^{2- \gamma_{2}} \mathbb{P}_{q+1, k} (\chi_{j}(t) a_{k,j} (t,x) c_{k} (\lambda_{q+1} \Phi_{j} (t,x))))  \label{est 368}\\
& + \lambda_{q+1}^{-1}\mathcal{B} (( \nabla^{\bot} \cdot y_{l})(t,x)   \sum_{j,k}   \Lambda^{2- \gamma_{2}} \mathbb{P}_{q+1, k} (\chi_{j}(t) \nabla (a_{k,j} (t,x) \psi_{q+1, j, k} (t,x)) c_{k} (\lambda_{q+1} x )) ). \nonumber
\end{align}
Thus, we can compute using \cite[equation (A.11) on p. 1865]{BSV19} and the fact that $\beta < \frac{6- 2 \gamma_{2}}{5}$ from \eqref{est 154b}, for $a \in 5 \mathbb{N}$ sufficiently large 
\begin{align}
&\lVert N_{2} \rVert_{C_{t,x}} \overset{\eqref{est 368}}{\lesssim}  \lambda_{q+1}^{-2} \sum_{j,k}  \lVert \nabla(\nabla^{\bot} \cdot y_{l} ) \rVert_{C_{t,x}} \lVert \Lambda^{2- \gamma_{2}} \mathbb{P}_{q+1, k} ( \chi_{j} a_{k,j} (x) c_{k} (\lambda_{q+1} \Phi_{j} (x))) \rVert_{C_{t,x}}  \label{est 303}  \\
& \hspace{14mm} + \lVert \nabla^{\bot} \cdot y_{l} \rVert_{C_{t,x}} \lVert \Lambda^{2- \gamma_{2}} \mathbb{P}_{q+1, k} (\chi_{j}  \nabla (a_{k,j} (x) \psi_{q+1, j, k} (x)) c_{k} (\lambda_{q+1} x) )  \rVert_{C_{t,x}} \nonumber \\
&\overset{\eqref{est 166} \eqref{est 216} \eqref{est 268}}{\lesssim}  L M_{0}(t) \lambda_{q+2}^{2-\gamma_{2}} \delta_{q+2} a^{b^{q} [(b-2) [(b+2) (-2+ \gamma_{2} + 2 \beta) - \gamma_{2} - \beta] - 6 + 2 \gamma_{2} + 5 \beta]}  \ll L M_{0}(t) \lambda_{q+2}^{2- \gamma_{2}} \delta_{q+2}. \nonumber
\end{align}
Applying \eqref{est 302} and \eqref{est 303} to \eqref{est 304} allows us to conclude 
\begin{equation}\label{est 305} 
\lVert R_{N} \rVert_{C_{t,x}} \ll L M_{0}(t) \lambda_{q+2}^{2-\gamma_{2}} \delta_{q+2}.
\end{equation}  
\subsubsection{Bounds on $R_{L}$, the linear error}
From \eqref{est 191d} let us write
\begin{equation}\label{est 224} 
R_{L} = L_{1} + L_{2} \hspace{1mm} \text{ where } \hspace{1mm} L_{1} \triangleq \mathcal{B} \Lambda^{\gamma_{1}} w_{q+1}\hspace{1mm}  \text{ and } \hspace{1mm} L_{2} \triangleq \mathcal{B} [\Lambda^{2- \gamma_{2}} w_{q+1}^{\bot} (\nabla^{\bot} \cdot z_{l})]. 
\end{equation}  
We estimate from $L_{1}$ for $a \in 5 \mathbb{N}$ sufficiently large by using $\gamma_{1} + 2 \gamma_{2} < 5 - 3 \beta$ from \eqref{est 154d} 
\begin{align}
\lVert L_{1} \rVert_{C_{t,x}} \overset{\eqref{est 163} }{\lesssim}& \lambda_{q+1}^{\gamma_{1} -1} \lVert w_{q+1} \rVert_{C_{t,x}}  \label{est 226} \\
\overset{\eqref{est 204a} \eqref{est 35}}{\lesssim}&  L^{\frac{1}{2}}  M_{0}(t)^{\frac{1}{2}} \lambda_{q+2}^{2-\gamma_{2}} \delta_{q+2} a^{b^{q+1} [(b-2) (-2+ \gamma_{2} + 2 \beta) - 5 + 2 \gamma_{2} + 3 \beta + \gamma_{1} ]} \ll L M_{0}(t) \lambda_{q+2}^{2- \gamma_{2}} \delta_{q+2}. \nonumber
\end{align} 
For $L_{2}$, because $\supp \hat{w}_{q+1}^{\bot} \subset \{ \xi: \frac{\lambda_{q+1}}{2} \leq \lvert \xi \rvert \leq 2 \lambda_{q+1} \}$ due to \eqref{est 163} and $\supp \hat{z}_{l} \subset B(0,  \frac{\lambda_{q}}{4})$, we can rely on \cite[equation (A.11)]{BSV19} to compute for $a \in 5 \mathbb{N}$ sufficiently large  
\begin{align}
\lVert L_{2} \rVert_{C_{t,x}} \overset{\eqref{est 224}}{\lesssim}& \lambda_{q+1}^{-1} \lVert \Lambda^{2- \gamma_{2}} w_{q+1} \rVert_{C_{t,x}} \lVert z \rVert_{C_{t}C_{x}^{1}}  \overset{\eqref{est 163} \eqref{est 204a} \eqref{est 166}}{\lesssim} \lambda_{q+1}^{1- \gamma_{2} - \beta} L^{\frac{1}{2}}  M_{0}(t)^{\frac{1}{2}} L^{\frac{1}{4}}\nonumber \\
\lesssim& M_{0}(t) \lambda_{q+2}^{2- \gamma_{2}} \delta_{q+2} a^{b^{q+1} [(b-2)(-2+ \gamma_{2} + 2 \beta) - 3 + \gamma_{2} + 3 \beta]} \ll L M_{0}(t) \lambda_{q+2}^{2- \gamma_{2}} \delta_{q+2} \label{est 227} 
\end{align}
where we used that \eqref{est 154b} guarantees that $\beta < \frac{3- \gamma_{2}}{3}$. Applying \eqref{est 226} and \eqref{est 227} in \eqref{est 224}, we conclude that 
\begin{equation}\label{est 306} 
\lVert R_{L} \rVert_{C_{t,x}} \ll LM_{0}(t) \lambda_{q+2}^{2- \gamma_{2}} \delta_{q+2}. 
\end{equation} 

\subsubsection{Bounds on $R_{O}$, the oscillation error}
Some of the computations in this subsection are generalizations of those in \cite[Section 5.4]{BSV19}; when the generalizations are straight-forward, we only sketch them and refer to \cite[Section 5.4]{BSV19} for details. We rewrite 
\begin{align}
\divergence R_{O} \overset{\eqref{est 191b} \eqref{est 211}}{=}& \divergence \left( \sum_{j} \chi_{j}^{2} (\mathring{R}_{l} - \mathring{R}_{q,j}) + \chi_{j}^{2} \mathring{R}_{q,j } \right) \nonumber\\
&+ (\Lambda^{2- \gamma_{2}} w_{q+1} \cdot \nabla) w_{q+1} - (\nabla w_{q+1})^{T} \cdot \Lambda^{2- \gamma_{2}} w_{q+1}. \label{est 228} 
\end{align} 
By \eqref{est 208} we can write 
\begin{subequations}\label{est 417}
\begin{align}
& (\Lambda^{2- \gamma_{2}} w_{q+1} \cdot \nabla) w_{q+1} = (\Lambda^{2- \gamma_{2}} \sum_{j,k}  \mathbb{P}_{q+1, k} \tilde{w}_{q+1, j, k}) \cdot \nabla \sum_{j',k'} \mathbb{P}_{q+1, k'} \tilde{w}_{q+1, j', k'}, \\
& (\nabla w_{q+1})^{T} \cdot\Lambda^{2- \gamma_{2}} w_{q+1} = (\nabla \sum_{j,k}  \mathbb{P}_{q+1, k} \tilde{w}_{q+1, j, k})^{T} \cdot \Lambda^{2- \gamma_{2}} \sum_{j',k'} \mathbb{P}_{q+1, k'} \tilde{w}_{q+1, j', k'}. 
\end{align}
\end{subequations} 
We call the cases $k + k' \neq 0$ and $k + k' = 0$  respectively the high and low frequency parts. We have $\frac{1}{2} \leq \lvert k + k' \rvert \leq 2$ due to Lemma \ref{Geometric Lemma} in the high oscillation part that leads to 
\begin{align*}
& \Lambda^{ 2- \gamma_{2}} \mathbb{P}_{q+1, k} \tilde{w}_{q+1, j, k} \cdot \nabla \mathbb{P}_{q+1, k'} \tilde{w}_{q+1, j', k'} = \tilde{P}_{\approx \lambda_{q+1}} [  \Lambda^{ 2- \gamma_{2}} \mathbb{P}_{q+1, k} \tilde{w}_{q+1, j, k} \cdot \nabla \mathbb{P}_{q+1, k'} \tilde{w}_{q+1, j', k'} ], \\
& ( \nabla \mathbb{P}_{q+1, k} \tilde{w}_{q+1, j, k})^{T} \cdot \Lambda^{2- \gamma_{2}} \mathbb{P}_{q+1, k'} \tilde{w}_{q+1, j', k'} = \tilde{P}_{\approx \lambda_{q+1}}[ ( \nabla \mathbb{P}_{q+1, k} \tilde{w}_{q+1, j, k})^{T} \cdot \Lambda^{2- \gamma_{2}} \mathbb{P}_{q+1, k'} \tilde{w}_{q+1, j', k'}]
\end{align*}
and thus we define formally 
\begin{align}
R_{O, \text{high}} \triangleq& \mathcal{B} \tilde{P}_{\approx \lambda_{q+1}} [ \sum_{j, j', k, k': k + k' \neq 0} ( \Lambda^{2- \gamma_{2}} \mathbb{P}_{q+1, k} \tilde{w}_{q+1, j, k}) \cdot \nabla \mathbb{P}_{q+1, k'} \tilde{w}_{q+1, j', k'} \nonumber\\
& \hspace{35mm} - (\nabla \mathbb{P}_{q+1, k} \tilde{w}_{q+1, j, k})^{T} \cdot \Lambda^{2- \gamma_{2}} \mathbb{P}_{q+1, k'} \tilde{w}_{q+1, j', k'} ] \label{est 230}
\end{align}
where $\sum_{j,j',k,k'} \triangleq \sum_{j,j' = 0}^{\lceil \tau_{q+1}^{-1} T_{L} \rceil} \sum_{k\in \Gamma_{j}, k' \in \Gamma_{j'}}$ (recall \eqref{est 401}).  Concerning the low frequency part, Lemma \ref{Geometric Lemma} shows that $k + k' = 0$ implies that $\Gamma_{j} = \Gamma_{j'}$; thus, we symmetrize to define the $(j,k)$-term of low-frequency part of the nonlinear terms as 
\begin{align}
\mathcal{J}_{j,k} &\triangleq \frac{1}{2} [  \Lambda^{2-\gamma_{2}} \mathbb{P}_{q+1,k} \tilde{w}_{q+1, j,k} \cdot \nabla \mathbb{P}_{q+1, -k} \tilde{w}_{q+1, j, -k} \label{est 232} \\
&\hspace{30mm} + \Lambda^{2- \gamma_{2}} \mathbb{P}_{q+1, -k} \tilde{w}_{q+1, j, -k} \cdot \nabla \mathbb{P}_{q+1, k} \tilde{w}_{q+1, j,k} \nonumber\\
& - (\nabla \mathbb{P}_{q+1, k} \tilde{w}_{q+1, j,k})^{T} \cdot \Lambda^{2- \gamma_{2}} \mathbb{P}_{q+1, -k} \tilde{w}_{q+1, j, -k} - (\nabla \mathbb{P}_{q+1, -k} \tilde{w}_{q+1, j, -k})^{T} \cdot \Lambda^{2- \gamma_{2}} \mathbb{P}_{q+1, k} \tilde{w}_{q+1, j, k} ]. \nonumber 
\end{align}
As we will see in \eqref{est 243}, we can decompose $\mathcal{J}_{j,k}$ as 
\begin{equation}\label{est 229}  
\mathcal{J}_{j,k} = \nabla \mathcal{P}_{j,k} + \divergence (\mathcal{L}_{j,k})  = \nabla \left( \mathcal{P}_{j,k} + \frac{ \mathcal{L}_{j,k}^{11} + \mathcal{L}_{j,k}^{22}}{2} \right) + \divergence (\mathring{\mathcal{L}}_{j,k}), 
\end{equation} 
where $\mathring{\mathcal{L}}_{j,k}$ is a trace-free part of $\mathcal{L}_{j,k}$, and both of $\mathcal{L}_{j,k}$ and $\mathcal{P}_{j,k} $ will be defined in \eqref{est 244}. Then \eqref{est 229} allows us to rewrite \eqref{est 228}  as 
\begin{equation}\label{est 291}
R_{O} = R_{O, \text{approx}} + R_{O, \text{low}} + R_{O, \text{high}}
\end{equation}  
where in addition to $R_{O, \text{high}}$ in \eqref{est 230} we defined 
\begin{equation}\label{est 259}
R_{O, \text{approx}} \triangleq \sum_{j} \chi_{j}^{2} (\mathring{R}_{l} -\mathring{R}_{q,j}) \hspace{2mm} \text{ and }  \hspace{2mm} R_{O, \text{low}} \triangleq \sum_{j} \chi_{j}^{2} \mathring{R}_{q,j} + \sum_{j,k}  \mathring{\mathcal{L}}_{j,k}. 
\end{equation} 
We have the identity of $\mathbb{P}_{q+1, k} b_{k} (\lambda_{q+1} x) = b_{k} (\lambda_{q+1} x)$ due to \eqref{est 192} and \eqref{est 179}  which leads to 
\begin{equation}\label{est 273} 
\mathbb{P}_{q+1, k} \tilde{w}_{q+1, j, k}(x) \overset{\eqref{est 200a} \eqref{est 208}}{=} \tilde{w}_{q+1, j, k} (x) +\chi_{j} [\mathbb{P}_{q+1, k}, a_{k,j} (x) \psi_{q+1, j, k} (x) ] b_{k} (\lambda_{q+1} x) 
\end{equation} 
and therefore 
\begin{equation}
w_{q+1}(t,x) \overset{\eqref{est 208}}{=} \sum_{j,k}  \tilde{w}_{q+1, j,k}(t,x) + \chi_{j} (t) [\mathbb{P}_{q+1, k}, a_{k,j} (t,x) \psi_{q+1, j, k}(t,x) ] b_{k} (\lambda_{q+1} x).
\end{equation} 
Next, by \eqref{est 200a} and \eqref{est 208} we can deduce 
\begin{align}
&\Lambda^{2- \gamma_{2}} \mathbb{P}_{q+1, k} \tilde{w}_{q+1, j, k}(x) \nonumber\\
=& \lambda_{q+1}^{2- \gamma_{2}} \tilde{w}_{q+1, j, k}(x) + \chi_{j} [\mathbb{P}_{q+1, k} \Lambda^{2- \gamma_{2}}, a_{k,j}(x) \psi_{q+1, j, k}(x) ] b_{k} (\lambda_{q+1} x)\label{est 274}
\end{align}
which leads to via \eqref{est 208} 
\begin{equation}
\Lambda^{2- \gamma_{2}} w_{q+1}(x) = \sum_{j,k} \lambda_{q+1}^{2- \gamma_{2}}    \tilde{w}_{q+1, j, k}(x) +  \chi_{j} [\mathbb{P}_{q+1, k} \Lambda^{2- \gamma_{2}}, a_{k,j}(x) \psi_{q+1, j, k}(x) ] b_{k} (\lambda_{q+1} x). 
\end{equation}  
We denote the potential vorticity associated to $\mathbb{P}_{q+1, k} \tilde{w}_{q+1, j, k}$ as 
\begin{equation}\label{est 233} 
\vartheta_{j,k} \triangleq \nabla^{\bot} \cdot \mathbb{P}_{q+1, k} \chi_{j} a_{k,j} b_{k} (\lambda_{q+1} \Phi_{j}) 
\end{equation} 
because $\tilde{w}_{q+1, j,k}(t,x) = \chi_{j}(t) a_{k,j}(t,x) b_{k} (\lambda_{q+1} \Phi_{j}(t,x))$ due to \eqref{est 200a}. For any $f$ that is mean-zero and divergence-free, due to \eqref{est 38} we have an identity of 
\begin{equation}\label{est 231}
\Lambda^{2- \gamma_{2}} f \cdot \nabla g - (\nabla g)^{T}\cdot \Lambda^{2- \gamma_{2}} f = \Lambda^{2- \gamma_{2}} f^{\bot} (\nabla^{\bot} \cdot g) = (\Lambda^{1- \gamma_{2}} \mathcal{R} \nabla^{\bot} \cdot f) (\nabla^{\bot} \cdot g). 
\end{equation}  
The identity \eqref{est 231} allows us to define 
\begin{equation}\label{est 235} 
\mathcal{J}(\vartheta_{j,k}, \vartheta_{j, -k}) \triangleq  \frac{1}{2} [ ( \Lambda^{1- \gamma_{2}} \mathcal{R} \vartheta_{j,k} ) \vartheta_{j, -k} + \vartheta_{j,k}(\Lambda^{1- \gamma_{2}} \mathcal{R} \vartheta_{j, -k} ) ] 
\end{equation} 
and deduce from \eqref{est 232} 
\begin{equation}\label{est 234}
\mathcal{J}_{j,k} \overset{\eqref{est 233}  \eqref{est 231}}{=} \mathcal{J} (\vartheta_{j,k}, \vartheta_{j, -k}).
\end{equation} 
Additionally, we define 
\begin{equation}\label{est 236} 
s^{m} (\lambda, \eta) \triangleq \int_{0}^{1} \frac{ i ((1-r) \eta - r \lambda)_{m}}{ \lvert (1-r) \eta - r \lambda \rvert^{2- \gamma_{2}}} dr 
\end{equation} 
so that the Fourier symbol of $\Lambda^{\gamma_{2} -1} \mathcal{R}_{m}$ is 
\begin{align}\label{est 237} 
s^{m} (-\eta, \eta) = \int_{0}^{1} \frac{ i \eta_{m}}{\lvert \eta \rvert^{2- \gamma_{2}}} dr = \lvert \eta \rvert^{\gamma_{2} -1} \frac{ i \eta_{m}}{\lvert \eta \rvert}. 
\end{align}
Denoting the $l$-th component of $\mathcal{J}(f,g)$ by $\mathcal{J}_{l} (f,g)$, we see that 
\begin{equation}\label{est 238}
2 ( \mathcal{J}_{l} (f,g))\hspace{1mm}\hat{}\hspace{1mm}  (\xi) \overset{\eqref{est 235}}{=} \int_{\mathbb{R}^{2}} \left( \frac{ i (\xi - \eta)_{l}}{\lvert \xi - \eta \rvert^{\gamma_{2}}} + \frac{ i \eta_{l}}{\lvert \eta \rvert^{\gamma_{2}} } \right) \hat{f} (\xi - \eta) \hat{g} (\eta) d\eta 
\end{equation} 
 where we can compute 
 \begin{equation}\label{est 239} 
 \frac{ i (\xi - \eta)_{l}}{ \lvert \xi - \eta \rvert^{\gamma_{2}}} + \frac{ i \eta_{l}}{\lvert \eta \rvert^{\gamma_{2}}} \overset{\eqref{est 236}}{=} \frac{ i \xi_{l}}{\lvert \xi - \eta \rvert^{\gamma_{2}}} + \sum_{m=1}^{2} i \xi_{m} \frac{ i \eta_{l}}{\lvert \eta \rvert^{\gamma_{2}}} \frac{1}{\lvert \xi - \eta \rvert^{\gamma_{2}}} \gamma_{2} s^{m} (\xi - \eta, \eta). 
\end{equation} 
Applying \eqref{est 239} to \eqref{est 238} gives 
\begin{align}
& (\mathcal{J}_{l}(f,g))\hspace{1mm}\hat{}\hspace{1mm}  (\xi) \label{est 240} \\
=& \frac{ i \xi_{l}}{2} \int_{\mathbb{R}^{2}} \frac{ \hat{f} (\xi - \eta)}{ \lvert \xi - \eta \rvert^{\gamma_{2}}} \hat{g} (\eta) d \eta + \sum_{m=1}^{2} \frac{ i \xi_{m} \gamma_{2}}{2} \int_{\mathbb{R}^{2}} s^{m} (\xi - \eta, \eta) \frac{ \hat{f} (\xi - \eta)}{ \lvert \xi - \eta \rvert^{\gamma_{2}}} \frac{ i \eta_{l}}{\lvert \eta \rvert^{\gamma_{2}}} \hat{g} (\eta) d \eta. \nonumber  
\end{align}
We define $\mathcal{T}^{m}$ by 
\begin{equation}\label{est 241}
( \mathcal{T}^{m} (f,g))\hspace{1mm}\hat{}\hspace{1mm}  (\xi) \triangleq \int_{\mathbb{R}^{2}} s^{m} (\xi - \eta, \eta) \hat{f} (\xi - \eta) \hat{g} (\eta) d\eta 
\end{equation} 
so that \eqref{est 240} gives us 
\begin{equation}\label{est 242} 
\mathcal{J}_{l} (f,g) = \frac{1}{2} \partial_{l} (( \Lambda^{-\gamma_{2}} f) g) + \frac{\gamma_{2}}{2} \sum_{m=1}^{2} \partial_{m} ( \mathcal{T}^{m} (\Lambda^{-\gamma_{2}} f, \Lambda^{1- \gamma_{2}} \mathcal{R}_{l} g)).
\end{equation} 
Taking Fourier inverse on \eqref{est 241} and writing $K_{s^{m}}$ as the Fourier inverse in $\mathbb{R}^{4}$ of $s^{m}$ give us 
\begin{equation}
\mathcal{T}^{m} (f,g) (x) = \int\int_{\mathbb{R}^{2} \times \mathbb{R}^{2}} K_{s^{m}} (x-y, x-z) f(y) g(z) dy dz. 
\end{equation} 
(see \cite[equations (5.27)-(5.28)]{BSV19} for details). Therefore, 
\begin{equation}\label{est 243}
\mathcal{J}_{j,k}  \overset{\eqref{est 234}\eqref{est 242}}{=} \frac{1}{2} \nabla (( \Lambda^{-\gamma_{2}} \vartheta_{j,k} ) \vartheta_{j, -k}) + \frac{\gamma_{2}}{2} \divergence ( \mathcal{T} ( \Lambda^{-\gamma_{2}} \vartheta_{j,k}, \Lambda^{1- \gamma_{2}} \mathcal{R} \vartheta_{j,-k})). 
\end{equation} 
Thus, we have proven our claim in \eqref{est 229}  with 
\begin{equation}\label{est 244}
\mathcal{L}_{j,k}^{ml} \triangleq \frac{\gamma_{2}}{2} \mathcal{T}^{m} ( \Lambda^{ -\gamma_{2}} \vartheta_{j,k}, \Lambda^{1- \gamma_{2}} \mathcal{R}^{l} \vartheta_{j, -k} ) \hspace{2mm} \text{ and } \hspace{2mm} \mathcal{P}_{j,k} \triangleq \frac{1}{2} ( \Lambda^{-\gamma_{2}} \vartheta_{j,k}) \vartheta_{j, -k}.
\end{equation} 
Next, we see that 
\begin{equation}\label{est 245}
\vartheta_{j,k} \overset{\eqref{est 233} \eqref{est 192} }{=} \nabla^{\bot} \cdot ( \mathbb{P} P_{\approx \lambda_{q+1} k} ) \tilde{w}_{q+1, j, k} = P_{\approx \lambda_{q+1} k} (\nabla^{\bot} \cdot \tilde{w}_{q+1, j, k}). 
\end{equation} 
It follows due to \eqref{est 245}, \eqref{est 200a}, \eqref{est 208}, and \eqref{est 178} that 
\begin{subequations}\label{est 246} 
\begin{align}
& \Lambda^{- \gamma_{2}} \vartheta_{j,k}(x) =  \chi_{j} ik^{\bot} \cdot \Lambda^{1- \gamma_{2}} \mathcal{R}^{\bot} P_{\approx \lambda_{q+1} k} (a_{k,j} (x) \psi_{q+1, j, k} (x) c_{k} (\lambda_{q+1} x)), \label{est 246a} \\
& \Lambda^{1- \gamma_{2}} \mathcal{R}_{l} \vartheta_{j, -k}(x) = - i \chi_{j} k^{\bot}  \cdot \nabla^{\bot}  \Lambda^{1- \gamma_{2}} \mathcal{R}_{l} P_{\approx -k \lambda_{q+1}} ( a_{-k, j} (x) \psi_{q+1, j, -k} (x) c_{-k} (\lambda_{q+1} x) ). \label{est 246b} 
\end{align}
\end{subequations} 
It follows by using the fact that $a_{k,j} = a_{-k,j}$ due to \eqref{est 196} and Lemma \ref{Geometric Lemma} (2) that 
\begin{subequations}\label{est 247}
\begin{align}
& ( \Lambda^{- \gamma_{2}} \vartheta_{j,k} )\hspace{1mm}\hat{}\hspace{1mm} (\xi) = - \chi_{j} \frac{ k \cdot \xi}{\lvert \xi \rvert^{\gamma_{2}}} \hat{K}_{\approx 1} \left( \frac{ \xi}{\lambda_{q+1}} - k \right) (a_{k,j} \psi_{q+1, j, k})\hspace{1mm}\hat{}\hspace{1mm}  (\xi - k \lambda_{q+1}), \\
&( \Lambda^{1- \gamma_{2}} \mathcal{R}_{l} \vartheta_{j, -k})\hspace{1mm}\hat{}\hspace{1mm}  (\xi)= \chi_{j} i \xi_{l} \frac{ k \cdot \xi}{\lvert \xi \rvert^{\gamma_{2}}} \hat{K}_{\approx 1} \left( \frac{\xi}{\lambda_{q+1}} + k \right) (a_{k,j} \psi_{q+1, j, -k})\hspace{1mm}\hat{}\hspace{1mm} (\xi + k \lambda_{q+1}). 
\end{align}
\end{subequations}   
Consequently, if we define 
\begin{align}
M_{k}^{ml} (\xi, \eta) \triangleq& - i s^{m} (\xi + k \lambda_{q+1}, \eta - k \lambda_{q+1}) \frac{ k\cdot ( \xi + k \lambda_{q+1} )}{\lvert \xi + k \lambda_{q+1} \rvert^{\gamma_{2}}} \nonumber\\
& \times \hat{K}_{\approx 1} \left( \frac{\xi}{\lambda_{q+1}} \right) (\eta_{l} - k_{l} \lambda_{q+1}) \frac{ k \cdot (\eta - k \lambda_{q+1})}{\lvert \eta - k \lambda_{q+1} \rvert^{\gamma_{2}}} \hat{K}_{\approx 1} \left( \frac{\eta}{\lambda_{q+1}} \right)  \label{est 248} 
\end{align}
and 
\begin{align}
M_{k,r}^{ml} (\xi, \eta) \triangleq& \frac{ ((1-r) \eta - r \xi - k \lambda_{q+1})_{m}}{\lvert (1-r) \eta - r \xi - k \lambda_{q+1} \rvert^{2- \gamma_{2}}} \nonumber\\
& \times \frac{ k \cdot (\xi + k \lambda_{q+1} )}{\lvert \xi + k \lambda_{q+1} \rvert^{\gamma_{2}}} \hat{K}_{\approx 1} \left( \frac{\xi}{\lambda_{q+1}} \right) (\eta_{l} - k_{l} \lambda_{q+1}) \frac{ k \cdot (\eta - k \lambda_{q+1} )}{\lvert \eta - k \lambda_{q+1} \rvert^{\gamma_{2}}} \hat{K}_{\approx 1} \left( \frac{\eta}{\lambda_{q+1}} \right)  \label{est 249} 
\end{align}
so that 
\begin{equation}\label{est 250} 
M_{k}^{ml} (\xi, \eta) = \int_{0}^{1} M_{k,r}^{ml} (\xi, \eta) dr, 
\end{equation} 
then we obtain from \eqref{est 244}, \eqref{est 241}, and \eqref{est 247}, 
\begin{align}
\mathcal{L}_{j,k}^{ml} (x) =& \frac{\gamma_{2}}{2} \frac{\chi_{j}^{2}(t)}{(2\pi)^{2}} \int \int_{\mathbb{R}^{2} \times \mathbb{R}^{2}} M_{k}^{ml} (\xi, \eta) \nonumber\\
& \times (a_{k,j} \psi_{q+1, j, k})\hspace{1mm}\hat{}\hspace{1mm} (\xi) (a_{k,j} \psi_{q+1, j, -k})\hspace{1mm}\hat{}\hspace{1mm} (\eta) e^{ix\cdot (\xi + \eta)} d\xi d\eta. \label{est 251} 
\end{align}
We observe that if we define for $\xi, \tilde{\xi} \in \mathbb{R}^{2}$
\begin{equation}\label{est 252}
(M_{k,r}^{\ast})^{ml} (\xi, \tilde{\xi}) \triangleq \frac{ (( 1-r) \tilde{\xi} - r \xi - k)_{m}}{\lvert (1-r) \tilde{\xi} - r \xi - k \rvert^{2 -\gamma_{2}}} (\tilde{\xi}_{l} - k_{l}) \frac{ k \cdot (\xi + k) k \cdot (\tilde{\xi} - k)}{\lvert \xi + k \rvert^{\gamma_{2}} \lvert \tilde{\xi} - k \rvert^{\gamma_{2}}} \hat{K}_{\approx 1} (\xi) \hat{K}_{\approx 1} (\tilde{\xi}),
\end{equation} 
then we can write due to \eqref{est 249} and \eqref{est 252}, 
\begin{equation}\label{est 255}
 M_{k,r}^{ml} (\xi, \eta)=\lambda_{q+1}^{2- \gamma_{2}} (M_{k,r}^{\ast})^{ml} \left(\frac{ \xi}{\lambda_{q+1}}, \frac{\eta}{\lambda_{q+1}} \right).  
\end{equation} 
We note that $M_{k,r}^{\ast}$ is independent of $\lambda_{q+1}$ and due to $\supp \hat{K}_{\approx 1}\subset B(0, \frac{1}{8})$ from Section \ref{Subsection 3.2}, we have $\supp (M_{k,r}^{\ast}) \subset B(0, \frac{1}{8}) \times B(0, \frac{1}{8})$. This implies that $M_{k,r}^{\ast}$ is infinitely many times differentiable with bounds that are uniform in $r \in (0,1)$. Next, we observe that 
\begin{subequations}\label{est 253} 
\begin{align}
& M_{k,r}^{ml} (0,0) \overset{\eqref{est 249}}{=} - k_{m} k_{l} \lambda_{q+1}^{2-\gamma_{2}}  \hspace{5mm} \forall \hspace{1mm} k \in \mathbb{S}^{1}, \label{est 253a}  \\
&  \frac{1}{(2\pi)^{2}} \int\int_{\mathbb{R}^{2} \times \mathbb{R}^{2}} (a_{k,j} \psi_{q+1, j, k}) \hspace{1mm} \hat{}\hspace{1mm} (\xi) (a_{k,j} \psi_{q+1, j, -k}) \hspace{1mm} \hat{}\hspace{1mm} (\eta) e^{ix\cdot (\xi + \eta)} d\xi d \eta \overset{\eqref{est 200b}}{=} a_{k,j}^{2} (t,x).  \label{est 253b} 
\end{align}
\end{subequations} 
Thus, if we define 
\begin{align}
\tilde{\mathcal{L}}_{j,k}^{ml} (t,x) \triangleq& \frac{\gamma_{2}}{2}\frac{\chi_{j}^{2}(t)}{(2\pi)^{2}} \int\int_{\mathbb{R}^{2} \times \mathbb{R}^{2}} \int_{0}^{1} [ M_{k,r}^{ml} (\xi,\eta) - M_{k,r}^{ml} (0,0) ] dr \nonumber  \\
& \hspace{20mm} \times  (a_{k,j} \psi_{q+1, j, k})(\xi)\hspace{1mm}\hat{}\hspace{1mm} (a_{k,j} \psi_{q+1, j, -k})\hspace{1mm}\hat{}\hspace{1mm} (\eta) e^{ix\cdot (\xi + \eta)} d\xi d \eta, \label{est 254}
\end{align}
then it follows that 
\begin{align}
\mathcal{L}_{j,k}^{ml} (t, x) 
\overset{\eqref{est 251}\eqref{est 250}}{=}& \frac{\gamma_{2}}{2} \frac{ \chi_{j}^{2}(t)}{(2\pi)^{2}} \int \int_{\mathbb{R}^{2} \times \mathbb{R}^{2}} [ \int_{0}^{1} M_{k,r}^{ml} (\xi,\eta) dr] \label{est 260}\\
& \times (a_{k,j} \psi_{q+1, j, k} )\hspace{1mm}\hat{}\hspace{1mm}  (\xi) (a_{k,j} \psi_{q+1, j, -k})\hspace{1mm}\hat{}\hspace{1mm}  (\eta) e^{ix\cdot (\xi + \eta)} d\xi d\eta \nonumber    \\
\overset{\eqref{est 253} \eqref{est 254} }{=}& \frac{\gamma_{2}}{2} \chi_{j}^{2}(t) \lambda_{q+1}^{2- \gamma_{2}} (k^{\bot} \otimes k^{\bot} -\Id)^{ml} a_{k,j}^{2}(t,x) + \tilde{\mathcal{L}}_{j,k}^{ml} (t,x), \nonumber
\end{align} 
and we note that $\Tr (k^{\bot} \otimes k^{\bot}) = 1$ because $k \in \mathbb{S}^{1}$. Next, we define 
\begin{subequations}\label{est 256} 
\begin{align}
( \tilde{\mathcal{L}}_{j,k}^{(1)})^{ml} (t,x) \triangleq& -i\frac{\gamma_{2}}{2}  \frac{ \chi_{j}^{2}(t)}{(2\pi)^{2}} \int_{0}^{1}\int_{0}^{1}\int \int_{\mathbb{R}^{2} \times \mathbb{R}^{2}} \lambda_{q+1}^{1- \gamma_{2}}  \nabla_{\xi} (M_{k,r}^{\ast})^{ml} \left( \frac{ \bar{r}\xi}{\lambda_{q+1}}, \frac{\bar{r} \eta}{\lambda_{q+1}} \right) \nonumber \\
& \cdot ( \nabla (a_{k,j} \psi_{q+1, j, k} ))\hspace{1mm}\hat{}\hspace{1mm} (\xi) (a_{k,j} \psi_{q+1, j, -k})\hspace{1mm}\hat{}\hspace{1mm}  (\eta) e^{ix\cdot (\xi + \eta)} d\xi d\eta d r d \bar{r}, \\ 
 ( \tilde{\mathcal{L}}_{j,k}^{(2)})^{ml} (t,x) \triangleq& -i \frac{\gamma_{2}}{2} \frac{\chi_{j}^{2} (t)}{(2\pi)^{2}} \int_{0}^{1}\int_{0}^{1} \int\int_{\mathbb{R}^{2} \times \mathbb{R}^{2}} \lambda_{q+1}^{1- \gamma_{2}}  \nabla_{\eta} (M_{k,r}^{\ast})^{ml} \left( \frac{\bar{r} \xi}{\lambda_{q+1}}, \frac{\bar{r} \eta}{\lambda_{q+1}} \right) \nonumber \\
& \cdot ( a_{k,j} \psi_{q+1, j, k} )\hspace{1mm}\hat{}\hspace{1mm} (\xi) (\nabla a_{k,j} \psi_{q+1, j, -k})\hspace{1mm}\hat{}\hspace{1mm}  (\eta) e^{ix \cdot (\xi + \eta)} d\xi d \eta d r d \bar{r}, 
\end{align}
\end{subequations}  
so that we can compute from \eqref{est 254} using \eqref{est 255} 
\begin{equation}\label{est 261}
\tilde{\mathcal{L}}_{j,k}^{ml} (t,x)  = ( \tilde{\mathcal{L}}_{j,k}^{(1)})^{ml} (t,x)  + ( \tilde{\mathcal{L}}_{j,k}^{(2)})^{ml} (t,x). 
\end{equation}
We also define for $z, \tilde{z} \in \mathbb{R}^{2}$
\begin{subequations}\label{est 257} 
\begin{align}
& ( \mathcal{K}_{k, r, \bar{r}}^{(1)})^{ml} (z, \tilde{z}) \triangleq -i \left( \frac{ \lambda_{q+1}}{\bar{r}} \right)^{4}  ( \nabla_{\xi} (M_{k,r}^{\ast} )^{ml})\hspace{1mm}\check{}\hspace{1mm}  \left( \frac{\lambda_{q+1} z}{\bar{r}}, \frac{\lambda_{q+1} \tilde{z}}{\bar{r}}\right), \\
& ( \mathcal{K}_{k, r, \bar{r}}^{(2)})^{ml} (z, \tilde{z}) \triangleq -i\left( \frac{ \lambda_{q+1}}{\bar{r}} \right)^{4}  ( \nabla_{\eta} (M_{k,r}^{\ast} )^{ml})\hspace{1mm}\check{}\hspace{1mm}  \left( \frac{\lambda_{q+1} z}{\bar{r}}, \frac{\lambda_{q+1} \tilde{z}}{\bar{r}}\right), 
\end{align}
\end{subequations} 
and 
\begin{subequations}\label{est 258} 
\begin{align}
& \tilde{\mathcal{T}}_{k}^{(1), ml} (\nabla (a_{k,j} \psi_{q+1, j, k}), a_{k,j} \psi_{q+1, j, -k}) (t,x) \triangleq \frac{\gamma_{2}}{2} \lambda_{q+1}^{1-\gamma_{2}} \int_{0}^{1}\int_{0}^{1} \int\int_{\mathbb{R}^{2} \times \mathbb{R}^{2}} ( \tilde{\mathcal{K}}_{k,r,\bar{r}}^{(1)})^{ml} ( x- z, x - \tilde{z}) \nonumber\\
& \hspace{45mm} \cdot \nabla (a_{k,j} \psi_{q+1, j, k})(z) (a_{k,j} \psi_{q+1, j, -k}) (\tilde{z}) dz d \tilde{z} dr d \bar{r},\\
& \tilde{\mathcal{T}}_{k}^{(2), ml} ( a_{k,j} \psi_{q+1, j, k}, \nabla (a_{k,j} \psi_{q+1, j, -k})) (t,x) \triangleq \frac{\gamma_{2}}{2} \lambda_{q+1}^{1-\gamma_{2}} \int_{0}^{1}\int_{0}^{1} \int\int_{\mathbb{R}^{2} \times \mathbb{R}^{2}} ( \tilde{\mathcal{K}}_{k,r,\bar{r}}^{(2)})^{ml} ( x- z, x - \tilde{z}) \nonumber\\
& \hspace{45mm}\cdot (a_{k,j} \psi_{q+1, j, k})(z) \nabla (a_{k,j} \psi_{q+1, j, -k}) (\tilde{z}) dz d \tilde{z} dr d \bar{r},
\end{align}
\end{subequations} 
so that we can compute from \eqref{est 256} 
\begin{subequations}\label{est 418}
\begin{align}
& (\tilde{\mathcal{L}}_{j,k}^{(1)})^{ml} (t,x)  = \chi_{j}^{2}(t) \tilde{\mathcal{T}}_{k}^{(1), ml} (\nabla (a_{k,j} \psi_{q+1, j, k}), a_{k,j} \psi_{q+1, j, -k})(t,x), \\
& ( \tilde{\mathcal{L}}_{j,k}^{(2)})^{ml} (t,x)  =  \chi_{j}^{2}(t) \tilde{\mathcal{T}}_{k}^{(2), ml} (a_{k,j} \psi_{q+1, j, k}, \nabla (a_{k,j} \psi_{q+1, j, -k}))(t,x), 
\end{align}
\end{subequations}
where as usual, the $\mathbb{T}^{2}$-periodic functions of $z$ and $\tilde{z}$ are identified with their periodic extensions to $\mathbb{R}^{2}$. It follows that for both $i \in \{1,2\}$, all $(z, \tilde{z}) \in \mathbb{R}^{4}$, and all $0 \leq \lvert a \rvert, \lvert b \rvert \leq 1$, uniformly in $r \in (0,1)$, $\mathcal{K}_{k,r, \bar{r}}^{(i)}$ in \eqref{est 257} satisfies 
\begin{equation}\label{est 267}
\left\lVert (z, \tilde{z})^{a} \nabla_{(z, \tilde{z})}^{b} ( \mathcal{K}_{k, r, \bar{r}}^{(i)} )^{ml} \right\rVert_{L_{z, \tilde{z}}^{1}(\mathbb{R}^{2} \times \mathbb{R}^{2} )} \leq C_{a,b} \left( \frac{\lambda_{q+1}}{\bar{r}}\right)^{\lvert b \rvert - \lvert a \rvert}.  
\end{equation} 
Using \eqref{est 259}, \eqref{est 260}, and \eqref{est 261}, we are now prepared to write 
\begin{equation}\label{est 262}
R_{O,\text{low}} = \mathring{O}_{1} + \mathring{O}_{2} 
\end{equation} 
where $\mathring{O}_{1}$ and $\mathring{O}_{2}$ respectively are trace-free parts of 
\begin{subequations}\label{est 263}
\begin{align}
& O_{1}(t,x) \triangleq \sum_{j} \chi_{j}^{2}(t) \mathring{R}_{q,j}(t,x) + \frac{ \gamma_{2} \lambda_{q+1}^{2-\gamma_{2}}}{2}  \sum_{j,k}  (k^{\bot} \otimes k^{\bot} - \Id) \chi_{j}^{2}(t) a_{k,j}^{2}(t,x), \label{est 263a}\\
& O_{2} \triangleq O_{21} + O_{22} \hspace{2mm} \text{ where } \hspace{2mm} O_{21} \triangleq  \sum_{j,k}  \tilde{\mathcal{L}}_{j,k}^{(1)}  \hspace{1mm} \text{ and } O_{22} \triangleq  \sum_{j,k}  \tilde{\mathcal{L}}_{j,k}^{(2)}. \label{est 263b}
\end{align}
\end{subequations} 
We make the key observation that due to our choice of $a_{k,j}$ in \eqref{est 196}, we have 
\begin{align}
O_{1} & \overset{\eqref{est 263a}\eqref{est 196}}{=} \sum_{j} \chi_{j}^{2} \lambda_{q+1}^{2- \gamma_{2}} [ \frac{ \mathring{R}_{q,j}}{\lambda_{q+1}^{2- \gamma_{2}}}  \nonumber\\
& + \frac{\gamma_{2}}{2} \sum_{k}  k^{\bot} \otimes k^{\bot}  \frac{\bar{C} LM_{0} (\tau_{q+1} j)}{\epsilon_{\gamma} \gamma_{2}} \delta_{q+1} \gamma_{k}^{2} \left( \Id - \frac{\epsilon_{\gamma}  \mathring{R}_{q,j}}{\bar{C} \lambda_{q+1}^{2-\gamma_{2}} \delta_{q+1} LM_{0} (\tau_{q+1} j)} \right) - \frac{\gamma_{2}}{2} \sum_{k} \Id a_{k,j}^{2} ] \nonumber  \\
\overset{\eqref{est 265}}{=}& \sum_{j} \chi_{j}^{2} \lambda_{q+1}^{2- \gamma_{2}} [ \frac{ \mathring{R}_{q,j}}{\lambda_{q+1}^{2- \gamma_{2}}} + \frac{\bar{C} LM_{0}(\tau_{q+1} j)}{\epsilon_{\gamma}} \delta_{q+1} \left( \Id - \frac{\epsilon_{\gamma} \mathring{R}_{q,j} }{\bar{C} \lambda_{q+1}^{2-\gamma_{2}} \delta_{q+1} LM_{0}(\tau_{q+1} j)} \right) - \frac{\gamma_{2}}{2} \sum_{k} \Id a_{k,j}^{2} ] \nonumber\\
& \hspace{6mm} = \sum_{j} \chi_{j}^{2} \lambda_{q+1}^{2- \gamma_{2}} \Id [ \frac{ \bar{C} LM_{0}(\tau_{q+1} j)}{\epsilon_{\gamma}} \delta_{q+1} - \frac{\gamma_{2}}{2} \sum_{k} a_{k,j}^{2} ] \label{est 382}
\end{align}
so that $\mathring{O}_{1}$ vanishes because a trace-free part of any multiple of an identity matrix is a zero matrix. Therefore, \eqref{est 262} is simplified to 
\begin{equation}\label{est 269}
R_{O,\text{low}}  = \mathring{O}_{21} + \mathring{O}_{22}
\end{equation} 
due to \eqref{est 263}.  We now come back to estimate $R_{O, \text{approx}}$ in \eqref{est 259}; we recall $D_{t,q}$ from \eqref{est 206} and realize that $D_{t,q} (\mathring{R}_{l} - \mathring{R}_{q,j} )= D_{t,q} \mathring{R}_{l}$ and $(\mathring{R}_{l} - \mathring{R}_{q,j}) (\tau_{q+1}j, x) = 0$ due to \eqref{est 195} and thus we may apply \eqref{est 197a} with ``$f_{0}$'' = 0 and ``$g$'' = $D_{t,q} \mathring{R}_{l}$ to estimate for all $t \in \supp \chi_{j}$, $b> L^{2} + 2$, and $a \in 5 \mathbb{N}$ sufficiently large, using the fact that $\beta < 2- \gamma_{2}$ due to \eqref{est 154b}, 
\begin{align}
\lVert (\mathring{R}_{l} - \mathring{R}_{q,j})(t) \rVert_{C_{x}} \overset{\eqref{est 206}}{\leq}&  \int_{ \tau_{q+1} j}^{t} \lVert \partial_{s} \mathring{R}_{l}(s) \rVert_{C_{x}} + \lVert (\Lambda^{2- \gamma_{2}} y_{l} + \Lambda^{2- \gamma_{2}} z_{l}) (s) \rVert_{C_{x}} \lVert \nabla \mathring{R}_{l} (s) \rVert_{C_{x}} ds \nonumber  \\
\overset{\eqref{est 166}\eqref{est 167} \eqref{est 136}}{\lesssim}& \tau_{q+1} [ l^{-1} L M_{0}(t) \lambda_{q+1}^{2- \gamma_{2}} \delta_{q+1} + L^{\frac{3}{2}} M_{0}(t)^{\frac{3}{2}} \lambda_{q}^{3- \gamma_{2}} \delta_{q}^{\frac{1}{2}} \lambda_{q+1}^{2-\gamma_{2}} \delta_{q+1} ] \nonumber\\
\lesssim& \tau_{q+1} l^{-1} LM_{0}(t) \lambda_{q+1}^{2- \gamma_{2}} \delta_{q+1}. \label{est 266}
\end{align}
This leads to 
\begin{equation}\label{est 292}
\lVert R_{O, \text{approx}} \rVert_{C_{t,x}} \overset{\eqref{est 259}}{\leq} \sum_{j} \lVert 1_{\supp \chi_{j}}  (\mathring{R}_{l} - \mathring{R}_{q,j} ) \rVert_{C_{t,x}} \overset{\eqref{est 266}}{\lesssim} \tau_{q+1} l^{-1} L M_{0}(t) \lambda_{q+1}^{2- \gamma_{2}} \delta_{q+1}. 
\end{equation} 
Next, we compute from \eqref{est 263b} 
\begin{align}
&\lVert O_{21} \rVert_{C_{t,x}} \label{est 270} \\
\overset{\eqref{est 258} \eqref{est 418}}{\lesssim}&  \sum_{j,k}  \chi_{j}^{2}(t) \lambda_{q+1}^{1-\gamma_{2}} \lVert \nabla (a_{k,j} \psi_{q+1, j, k}) \rVert_{C_{t,x}} \lVert a_{k,j} \psi_{q+1, j, -k} \rVert_{C_{t,x}} \sup_{r, \bar{r} \in [0,1]} \lVert ( \mathcal{K}_{k, r, \bar{r}}^{(1)})^{ml} \rVert_{L^{1} (\mathbb{R}^{2} \times \mathbb{R}^{2} )}\nonumber \\
\overset{\eqref{est 267} \eqref{est 268}}{\lesssim}&  \sum_{j,k}  \chi_{j}^{2}(t) \lambda_{q+1}^{1- \gamma_{2}} \lambda_{q} \delta_{q+1}^{\frac{1}{2}} L^{\frac{1}{2}} M_{0}(t)^{\frac{1}{2}} \lVert a_{k,j} \rVert_{C_{t,x}} \lVert \psi_{q+1, j, -k} \rVert_{C_{t,x}}  
\overset{\eqref{est 216a}}{\lesssim}  \lambda_{q+1}^{1-\gamma_{2}} \lambda_{q} \delta_{q+1} L M_{0}(t).\nonumber
\end{align}
Looking at \eqref{est 256}, we realize that an identical bound in \eqref{est 270} applies for $O_{22}$ and thus 
\begin{equation}\label{est 293}
\lVert R_{O, \text{low}} \rVert_{C_{t,x}} \overset{\eqref{est 269}}{\leq} \lVert O_{21} \rVert_{C_{t,x}} + \lVert O_{22} \rVert_{C_{t,x}} \overset{\eqref{est 270}}{\lesssim} \lambda_{q+1}^{1-\gamma_{2}} \lambda_{q} \delta_{q+1} M_{0}(t). 
\end{equation} 
Next, we define 
\begin{subequations}\label{est 272} 
\begin{align}
& O_{3} \triangleq \mathcal{B} \tilde{P}_{\approx \lambda_{q+1}} \left( \sum_{j, j', k, k': k + k' \neq 0} (\Lambda^{2- \gamma_{2}} \mathbb{P}_{q+1, k} \tilde{w}_{q+1, j, k}) \cdot \nabla \mathbb{P}_{q+1, k'} \tilde{w}_{q+1, j', k'} \right), \label{est 272a} \\
& O_{4} \triangleq \mathcal{B} \tilde{P}_{\approx \lambda_{q+1}} \left( \sum_{j, j', k, k': k + k' \neq 0}  (\nabla \mathbb{P}_{q+1, k} \tilde{w}_{q+1, j, k})^{T} \cdot \Lambda^{2-\gamma_{2}} \mathbb{P}_{q+1, k'} \tilde{w}_{q+1, j', k'} \right), \label{est 272b} 
\end{align}
\end{subequations} 
so that \eqref{est 230} gives us  
\begin{equation}\label{est 271} 
R_{O, \text{high}} = O_{3} - O_{4}. 
\end{equation} 
To work on $O_{3}$ from \eqref{est 272a}, we use the identity of $(B\cdot\nabla) A = \divergence (A \otimes B) - A (\nabla\cdot B)$, the fact that $\Lambda^{2-\gamma_{2}} \mathbb{P}_{q+1, k} \tilde{w}_{q+1,j, k}$ is divergence-free, \eqref{est 273}, and \eqref{est 274} to split $O_{3}$ to 
\begin{equation}\label{est 286} 
O_{3} = \sum_{k=1}^{3} O_{3k} 
\end{equation} 
where 
\begin{subequations}\label{est 275} 
\begin{align}
O_{31}(x) \triangleq& \mathcal{B} \tilde{P}_{\approx \lambda_{q+1}} \lambda_{q+1}^{2 - \gamma_{2}} \sum_{j, j',k,k': k + k'\neq 0} \divergence ( \tilde{w}_{q+1, j, k} \otimes \tilde{w}_{q+1, j', k'})(x), \label{est 275a} \\
O_{32}(x) \triangleq& \mathcal{B} \tilde{P}_{\approx \lambda_{q+1}} \lambda_{q+1}^{2 - \gamma_{2}} \sum_{j, j',k, k': k + k'\neq 0}  \nonumber\\
& \hspace{5mm} \divergence ( \tilde{w}_{q+1, j, k}(x) \otimes \chi_{j'} [\mathbb{P}_{q+1, k'}, a_{k', j'}(x) \psi_{q+1, j', k'}(x)] b_{k'} (\lambda_{q+1} x) ), \label{est 275b}\\
O_{33}(x) \triangleq& \mathcal{B} \tilde{P}_{\approx \lambda_{q+1}}  \sum_{j, j', k, k': k + k'\neq 0} \nonumber\\
& \hspace{5mm}  \divergence (\chi_{j} [\mathbb{P}_{q+1} \Lambda^{2-\gamma_{2}}, a_{k, j}(x) \psi_{q+1, j, k}(x)] b_{k} (\lambda_{q+1} x) \otimes \mathbb{P}_{q+1, k'} \tilde{w}_{q+1, j', k'}(x)). \label{est 275c}
\end{align}
\end{subequations} 
We now rewrite $O_{31}$ as follows; it is inspired by \cite[equations (104)-(105)]{Y21c} and different from \cite[p. 1853]{BSV19} due to a technical reason. We first rely on \eqref{est 275a}, \eqref{est 200a}, \eqref{est 208}, and symmetry to write 
 \begin{equation}\label{est 280}
O_{31} = O_{311} + O_{312}
\end{equation} 
where 
\begin{subequations}\label{est 278}
\begin{align}
O_{311}(x) &\triangleq \frac{1}{2} \mathcal{B} \tilde{P}_{\approx \lambda_{q+1}} \lambda_{q+1}^{2- \gamma_{2}} \sum_{j, j', k, k': k + k' \neq 0} \chi_{j}\chi_{j'} a_{k,j} (x) a_{k',j'} (x) \psi_{q+1, j', k'} (x) \psi_{q+1, j, k} (x)\nonumber\\
& \times \divergence (b_{k'} (\lambda_{q+1} x) \otimes b_{k} (\lambda_{q+1} x) + b_{k} (\lambda_{q+1} x) \otimes b_{k'} (\lambda_{q+1} x)), \label{est 278a}\\
O_{312}(x) &\triangleq  \mathcal{B} \tilde{P}_{\approx \lambda_{q+1}} \lambda_{q+1}^{2-\gamma_{2}} \sum_{j,j',k,k': k + k' \neq 0} \chi_{j} \chi_{j'}  \nonumber\\
&\times b_{k'}(\lambda_{q+1} x) \otimes b_{k} (\lambda_{q+1} x) \nabla (a_{k,j} (x) a_{k', j'} (x) \psi_{q+1, j', k'} (x) \psi_{q+1, j, k} (x)). \label{est 278b}
\end{align}
\end{subequations}
Within $O_{311}$, we can rely on the identity 
\begin{equation}\label{est 276} 
(A\cdot\nabla) B + (B\cdot\nabla) A = \nabla (A\cdot B) - A \times \nabla \times B - B \times \nabla \times A, 
\end{equation} 
that was also used on \cite[p. 113]{BV19a}, \eqref{est 179}, and that $k \in \mathbb{S}^{1}$ to rewrite 
\begin{align}
& \divergence (b_{k'} (\lambda_{q+1} x) \otimes b_{k} (\lambda_{q+1} x) + b_{k} (\lambda_{q+1} x) \otimes b_{k'} (\lambda_{q+1} x)) \nonumber\\
=& \nabla \left(b_{k} (\lambda_{q+1} x) \cdot b_{k'} (\lambda_{q+1} x) + e^{i(k+k') \cdot \lambda_{q+1} x} \right).  \label{est 277}
\end{align}
Applying \eqref{est 277} to \eqref{est 278a} and using the fact that $\mathcal{B} f = \mathcal{B} \mathbb{P} f$ for any $f$ that is not divergence-free and $\mathbb{P}$ eliminates any gradient give us 
\begin{align}
O_{311}(x) \overset{\eqref{est 278a}}{=}& \frac{1}{2} \mathcal{B} \tilde{P}_{\approx \lambda_{q+1}} \lambda_{q+1}^{2- \gamma_{2}} \sum_{j, j', k, k': k + k' \neq 0} \chi_{j} \chi_{j'} \nabla (a_{k,j} (x) a_{k', j'} (x) \psi_{q+1, j', k'}(x) \psi_{q+1, j, k}(x))  \nonumber \\
& \hspace{35mm} \times (b_{k}(\lambda_{q+1} x) \cdot b_{k'} (\lambda_{q+1} x) - e^{i(k+k') \cdot \lambda_{q+1} x}). \label{est 279}
\end{align}
From \eqref{est 279} and \eqref{est 278} we are able to immediately derive 
\begin{align}
& \lVert O_{31} \rVert_{C_{t,x}} \overset{\eqref{est 280}}{\leq} \sum_{k=1}^{2} \lVert O_{31k} \rVert_{C_{t,x}}  \label{est 287} \\
&\overset{\eqref{est 279}\eqref{est 278}}{\lesssim} \lambda_{q+1}^{-1} \lambda_{q+1}^{2-\gamma_{2}} \sum_{j, j', k, k': k + k' \neq 0} \lVert  1_{\supp \chi_{j}} \nabla (a_{k,j} \psi_{q+1, j, k}) \rVert_{C_{t,x}} \lVert 1_{\supp \chi_{j'}} a_{k', j'} \psi_{q+1, j', k'} \rVert_{C_{t,x}} \nonumber\\
& \hspace{8mm} + \lVert 1_{\supp \chi_{j'}} \nabla (a_{k', j'} \psi_{q+1, j', k'}) \rVert_{C_{t,x}} \lVert 1_{\supp \chi_{j}} a_{k,j} \psi_{q+1, j, k} \rVert_{C_{t,x}} 
\overset{\eqref{est 268} \eqref{est 216}}{\lesssim}  \lambda_{q+1}^{1-\gamma_{2}} \lambda_{q} \delta_{q+1} LM_{0}(t). \nonumber 
\end{align}
We can estimate $O_{32}$ and $O_{33}$ from \eqref{est 275b}-\eqref{est 275c} more immediately as follows by relying on \cite[equation (A.17)]{BSV19}:
\begin{align}
& \lVert O_{32} \rVert_{C_{t,x}} + \lVert O_{33} \rVert_{C_{t,x}} \overset{\eqref{est 200a} }{\lesssim} \sum_{j, j', k, k': k + k' \neq 0} \lambda_{q+1}^{2-\gamma_{2}} \lVert \chi_{j} a_{k,j}  b_{k} (\lambda_{q+1} \Phi_{j} ) \rVert_{C_{t,x}} \nonumber \\
& \hspace{50mm} \times \lambda_{q+1}^{-1} \lVert 1_{\supp \chi_{j'}} \nabla (a_{k', j'} \psi_{q+1, j', k'}) \rVert_{C_{t,x}} \lVert b_{k'} (\lambda_{q+1} x) \rVert_{C_{t,x}} \nonumber \\
&\hspace{25mm} + \lambda_{q+1}^{1-\gamma_{2}} \lVert 1_{\supp \chi_{j}} \nabla (a_{k,j} \psi_{q+1, j, k}) \rVert_{C_{t,x}} \lVert b_{k} (\lambda_{q+1} x) \rVert_{C_{t,x}} \lVert \chi_{j'}  a_{k', j'}  b_{k'} (\lambda_{q+1} \Phi_{j'}) \rVert_{C_{t,x}} \nonumber \\
& \hspace{60mm} \overset{\eqref{est 216a} \eqref{est 268}}{\lesssim} \lambda_{q+1}^{1- \gamma_{2}} \lambda_{q} \delta_{q+1} L M_{0}(t). \label{est 288} 
\end{align}
Thus, we conclude 
\begin{equation}\label{est 289}
\lVert O_{3} \rVert_{C_{t,x}} \overset{\eqref{est 286}}{\leq} \sum_{k=1}^{3} \lVert O_{3k} \rVert_{C_{t,x}} \overset{\eqref{est 287} \eqref{est 288}}{\lesssim} \lambda_{q+1}^{1-\gamma_{2}} \lambda_{q} \delta_{q+1} L M_{0}(t).
\end{equation} 
Next, we define 
\begin{subequations}\label{est 282} 
\begin{align}
 O_{41}(x) \triangleq& \sum_{j, j', k, k': k + k' \neq 0}  \mathcal{B} \tilde{P}_{\approx \lambda_{q+1}} \nonumber \\
& (\chi_{j} \nabla ( [ \mathbb{P}_{q+1, k}, a_{k,j} (x) \psi_{q+1, j, k} (x) ] b_{k} (\lambda_{q+1} x) )^{T} \cdot \lambda_{q+1}^{2-\gamma_{2}} \tilde{w}_{q+1, j', k'}(x)), \label{est 282a} \\
 O_{42}(x) \triangleq& \sum_{j, j', k, k': k + k' \neq 0}  \mathcal{B} \tilde{P}_{\approx \lambda_{q+1}} \nonumber \\
&((\nabla \mathbb{P}_{q+1, k} \tilde{w}_{q+1, j, k} (x))^{T} \cdot \chi_{j'} [\mathbb{P}_{q+1, k'} \Lambda^{2- \gamma_{2}}, a_{k', j'}(x) \psi_{q+1, j', k'} (x)] b_{k'} (\lambda_{q+1} x) ), \label{est 282b}  
\end{align}
\end{subequations} 
so that $O_{4}$ from \eqref{est 272b} satisfies  
\begin{equation}\label{est 283} 
O_{4} = O_{41} + O_{42}
\end{equation} 
due to \eqref{est 273}, \eqref{est 274}, and 
\begin{equation}
\mathcal{B} (( \nabla \tilde{w}_{q+1, j, k})^{T} \cdot \tilde{w}_{q+1, j', k'} + (\nabla \tilde{w}_{q+1, j', k'})^{T} \cdot \tilde{w}_{q+1, j, k} ) = \mathcal{B} \nabla (\tilde{w}_{q+1, j, k} \cdot \tilde{w}_{q+1, j', k'}) = 0
\end{equation} 
which follows from the fact that $\mathcal{B} f = \mathcal{B} \mathbb{P} f$ for any $f$ that is not divergence-free and $\mathbb{P}$ eliminates any gradient. Let us further write 
\begin{align}
& \nabla ( [ \mathbb{P}_{q+1, k}, a_{k,j} (x) \psi_{q+1, j, k} (x) ] b_{k} (\lambda_{q+1} x) ) \nonumber \\
=& [ \mathbb{P}_{q+1, k}, \nabla (a_{k,j} \psi_{q+1, j, k}) ] b_{k} (\lambda_{q+1} x) + [\mathbb{P}_{q+1, k}, a_{k,j} \psi_{q+1, j, k} ] \nabla b_{k} (\lambda_{q+1} x). \label{est 281}
\end{align}
Applying \eqref{est 281} in \eqref{est 282a} and then\cite[equation (A.17)]{BSV19} to $O_{41}$, as well as using the fact that $\alpha > 1$ due to \eqref{est 154c} give us for 
\begin{align*}
b > 
\begin{cases}
2+ 4L^{2} & \text{ if } \gamma_{2} = 1, \\
2+ L^{2} (\frac{2}{1- \frac{\gamma_{2}}{2}}) & \text{ if } \gamma_{2} > 1, 
\end{cases}
\end{align*}
and $a \in 5 \mathbb{N}$ sufficiently large, 
\begin{align}
& \lVert O_{41} \rVert_{C_{t,x}} \overset{\eqref{est 282a}\eqref{est 281}\eqref{est 200a}}{\lesssim} \lambda_{q+1}^{-1} \sum_{j, j', k, k': k + k' \neq 0} [ \lambda_{q+1}^{-1} \lVert 1_{\supp \chi_{j}} \nabla^{2} (a_{k,j} \psi_{q+1, j, k}) \rVert_{C_{t,x}} \lVert b_{k} (\lambda_{q+1} x) \rVert_{C_{t,x}}\nonumber  \\
& \hspace{35mm} + \lambda_{q+1}^{-1} \lVert 1_{\supp \chi_{j}} \nabla (a_{k,j} \psi_{q+1, j, k} ) \rVert_{C_{t,x}} \lVert \nabla b_{k} (\lambda_{q+1} x) \rVert_{C_{t,x}} ] \lambda_{q+1}^{2-\gamma_{2}} \lVert a_{k', j'} \rVert_{C_{t,x}} \nonumber \\
&\overset{\eqref{est 268} \eqref{est 216}}{\lesssim} \lambda_{q+1}^{1- \gamma_{2}} \delta_{q+1}^{\frac{1}{2}} L^{\frac{1}{2}} M_{0}(t)^{\frac{1}{2}} [ \lambda_{q+1}^{-1} ( \lambda_{q}^{2} \delta_{q+1}^{\frac{1}{2}} L^{\frac{1}{2}} M_{0}(t)^{\frac{1}{2}} + \delta_{q+1}^{\frac{1}{2}} L M_{0}(t) \lambda_{q+1} \tau_{q+1} \lambda_{q}^{4- \gamma_{2}}  \delta_{q}^{\frac{1}{2}}) \nonumber \\
& \hspace{55mm} + \lambda_{q} \delta_{q+1}^{\frac{1}{2}} L^{\frac{1}{2}}M_{0}(t)^{\frac{1}{2}} ] \lesssim \lambda_{q+1}^{1- \gamma_{2}} \lambda_{q} \delta_{q+1} L M_{0}(t).\label{est 284} 
\end{align}
Similarly, we can estimate from \eqref{est 282b} by \cite[equation (A.17)]{BSV19} 
\begin{align}
\lVert  &O_{42} \rVert_{C_{t,x}} 
\lesssim \lambda_{q+1}^{-1} \sum_{j,j',k,k': k + k' \neq 0} \lVert 1_{\supp \chi_{j}} \nabla \mathbb{P}_{q+1, k} \tilde{w}_{q+1, j, k} \rVert_{C_{t,x}} \label{est 285} \\
& \hspace{20mm} \times  \lVert  1_{\supp \chi_{j'}} [ \mathbb{P}_{q+1, k'} \Lambda^{2-\gamma_{2}}, a_{k', j'} \psi_{q+1, j', k'} ] b_{k'} (\lambda_{q+1} x) \rVert_{C_{t,x}} \nonumber \\
&\lesssim \lambda_{q+1}^{-1} \sum_{j,j',k,k': k + k' \neq 0} \lambda_{q+1} \lVert 1_{\supp \chi_{j}} \tilde{w}_{q+1, j, k} \rVert_{C_{t,x}}   \nonumber \\
& \times \lambda_{q+1}^{1-\gamma_{2}} \lVert 1_{\supp \chi_{j'}} \nabla (a_{k', j'} \psi_{q+1, j', k'}) \rVert_{C_{t,x}} \lVert b_{k'} (\lambda_{q+1} x) \rVert_{C_{t,x}} \overset{\eqref{est 200a} \eqref{est 268} \eqref{est 216a}}{\lesssim} \lambda_{q+1}^{1-\gamma_{2}} \lambda_{q} \delta_{q+1} LM_{0}(t). \nonumber 
\end{align}
In sum, we conclude 
\begin{equation}\label{est 290}
\lVert O_{4} \rVert_{C_{t,x}} \overset{\eqref{est 283}}{\leq} \sum_{k=1}^{2} \lVert O_{4k} \rVert_{C_{t,x}} \overset{\eqref{est 284} \eqref{est 285}}{\lesssim} \lambda_{q+1}^{1-\gamma_{2}} \lambda_{q} \delta_{q+1} LM_{0}(t).
\end{equation} 
Therefore, 
\begin{equation}\label{est 294}
\lVert R_{O,\text{high}} \rVert_{C_{t,x}} \overset{\eqref{est 271}}{\leq} \lVert O_{3} \rVert_{C_{t,x}} + \lVert O_{4} \rVert_{C_{t,x}} \overset{\eqref{est 289} \eqref{est 290}}{\lesssim} \lambda_{q+1}^{1-\gamma_{2}} \lambda_{q} \delta_{q+1} LM_{0}(t).
\end{equation} 
At last, we deduce by relying on the fact that $\alpha < 7 - 3 \gamma_{2} - 5 \beta$ and $\beta < \frac{5- 2 \gamma_{2}}{4}$ due to \eqref{est 154b} and \eqref{est 154c}, for $a \in 5 \mathbb{N}$ sufficiently large 
\begin{align}
\lVert R_{O} &\rVert_{C_{t,x}} \overset{\eqref{est 291}}{\leq} \lVert R_{O, \text{approx}} \rVert_{C_{t,x}} + \lVert R_{O, \text{low}} \rVert_{C_{t,x}} + \lVert R_{O, \text{high}} \rVert_{C_{t,x}}  \label{est 307}\\
& \hspace{5mm} \overset{\eqref{est 292} \eqref{est 293} \eqref{est 294}}{\lesssim} \tau_{q+1} l^{-1} LM_{0}(t) \lambda_{q+1}^{2- \gamma_{2}} \delta_{q+1} + \lambda_{q+1}^{1- \gamma_{2}} \lambda_{q} \delta_{q+1} LM_{0}(t) \nonumber \\
& \hspace{5mm} \overset{\eqref{est 210}\eqref{est 32}}{\approx} LM_{0}(t) \lambda_{q+2}^{2- \gamma_{2}} \delta_{q+2} [ a^{b^{q+1} [(b-2) (-2 + \gamma_{2} + 2 \beta) + \frac{ -7 + 3 \gamma_{2} + 5 \beta + \alpha}{2} ]} \nonumber\\
& \hspace{25mm} + a^{b^{q} [(b-2) [(b+2) (-2+ \gamma_{2} + 2 \beta) + 1 - \gamma_{2} - 2 \beta] + 4 \beta - 5 + 2 \gamma_{2} ]} ] \ll LM_{0}(t) \lambda_{q+2}^{2- \gamma_{2}} \delta_{q+2}. \nonumber 
\end{align}

\subsubsection{Bounds on $R_{Com1}$, the first commutator error}
From \eqref{est 181}, we compute 
\begin{equation}
\lVert R_{\text{Com1}} \rVert_{C_{t,x}} \leq \sum_{k=1}^{4} \RomanIII_{k}
\end{equation} 
where 
\begin{subequations}\label{est 182} 
\begin{align}
\RomanIII_{1} \triangleq& \lVert \mathcal{B} [( \lambda^{2-\gamma_{2}} y_{l})^{\bot} \nabla^{\bot} \cdot y_{l} - [ (\Lambda^{2- \gamma_{2}} y_{q})^{\bot} \nabla^{\bot} \cdot y_{q} ] \ast_{x} \phi_{l} \ast_{t} \varphi_{l} ] \rVert_{C_{t,x}}, \\
\RomanIII_{2} \triangleq& \lVert \mathcal{B} [( \lambda^{2-\gamma_{2}} y_{l})^{\bot} \nabla^{\bot} \cdot z_{l} - [ (\Lambda^{2- \gamma_{2}} y_{q})^{\bot} \nabla^{\bot} \cdot z_{q} ] \ast_{x} \phi_{l} \ast_{t} \varphi_{l} ] \rVert_{C_{t,x}}, \\
\RomanIII_{3} \triangleq& \lVert \mathcal{B} [( \lambda^{2-\gamma_{2}} z_{l})^{\bot} \nabla^{\bot} \cdot y_{l} - [ (\Lambda^{2- \gamma_{2}} z_{q})^{\bot} \nabla^{\bot} \cdot y_{q} ] \ast_{x} \phi_{l} \ast_{t} \varphi_{l} ] \rVert_{C_{t,x}}, \\
\RomanIII_{4} \triangleq& \lVert \mathcal{B} [( \lambda^{2-\gamma_{2}} z_{l})^{\bot} \nabla^{\bot} \cdot z_{l} - [ (\Lambda^{2- \gamma_{2}} z_{q})^{\bot} \nabla^{\bot} \cdot z_{q} ] \ast_{x} \phi_{l} \ast_{t} \varphi_{l} ] \rVert_{C_{t,x}}. 
\end{align}
\end{subequations} 
Using $\lVert \mathcal{B} \rVert_{C_{x} \mapsto C_{x}} \lesssim 1$ and the standard commutator estimate (e.g, \cite[equation (5) on p. 88]{CDS12b}) we bound them separately as follows: as $5- \gamma_{2} - 2\beta \leq 9 - 3\gamma_{2} - 4 \beta$ due to \eqref{est 154b}, 
\begin{subequations}\label{est 183} 
\begin{align}
\RomanIII_{1} \lesssim&  l^{2} (\lambda_{q}^{3- \gamma_{2}} \lVert y_{q} \rVert_{C_{x}C_{t}^{1}}^{2} + \lambda_{q}^{1+ \gamma_{2}}  \lVert \Lambda^{2- \gamma_{2}} y_{q} \rVert_{C_{t,x}}^{2}) \nonumber\\
\overset{\eqref{est 166}\eqref{est 169}}{\lesssim}& l^{2} ( \lambda_{q}^{9 - 3 \gamma_{2} - 4 \beta} L^{2} M_{0}(t)^{2}  + \lambda_{q}^{5 - \gamma_{2} - 2 \beta} LM_{0}(t)) \overset{\eqref{est 154b}}{\lesssim} l^{2} \lambda_{q}^{9 - 3 \gamma_{2} - 4 \beta} L^{2}M_{0}(t)^{2}, \label{est 183a} \\
\RomanIII_{2} \lesssim& l^{2(\frac{1}{2} - 2 \delta)} \lVert \Lambda^{2- \gamma_{2}} y_{q} \ast_{x} \phi_{l} \rVert_{C_{x} C_{t}^{\frac{1}{2} - 2 \delta}} \lVert \nabla^{\bot} \cdot z_{q}\rVert_{C_{x}C_{t}^{\frac{1}{2} - 2 \delta}} + l^{2} \lVert \Lambda^{2- \gamma_{2}} y_{q} \rVert_{C_{t}C_{x}^{1}} \lVert \nabla^{\bot} \cdot z_{q}\rVert_{C_{t}C_{x}^{1}}  \nonumber \\
\overset{\eqref{est 136} \eqref{est 166} \eqref{est 169}}{\lesssim}& l^{1- 4\delta} M_{0}(t)^{\frac{3}{4} - \delta} \delta_{q}^{\frac{3}{4} - \delta} \lambda_{q}^{2- \gamma_{2} + (3-\gamma_{2})(\frac{1}{2} - 2\delta)} L^{\frac{5}{4} - \delta} +  l^{2} \lambda_{q}^{3- \gamma_{2} - \beta} M_{0}(t)^{\frac{1}{2}} L^{\frac{3}{4}}, \label{est 183b}  \\
\RomanIII_{3} \overset{\eqref{est 136}}{\lesssim}& l^{1- 4 \delta} L^{\frac{1}{2}} \lambda_{q}^{\gamma_{2} -1} \lVert \Lambda^{2- \gamma_{2}} y_{q} \rVert_{C_{t,x}}^{\frac{1}{2} + 2 \delta} \lVert \Lambda^{2-\gamma_{2}} y_{q} \rVert_{C_{t}^{1}C_{x}}^{\frac{1}{2} - 2 \delta} + l^{2} L^{\frac{1}{4}} \lambda_{q}^{\gamma_{2}} \lVert \Lambda^{2-\gamma_{2}} y_{q} \rVert_{C_{t,x}}   \nonumber\\
\overset{\eqref{est 166} \eqref{est 169}}{\lesssim}& l^{1- 4 \delta} L^{\frac{5}{4} - \delta} \lambda_{q}^{1 + (3- \gamma_{2}) (\frac{1}{2} - 2 \delta)} M_{0}(t)^{\frac{3}{4} - \delta} \delta_{q}^{\frac{3}{4} - \delta} + l^{2} L^{\frac{3}{4}} \lambda_{q}^{2 - \beta} M_{0}(t)^{\frac{1}{2}}, \label{est 183c}   \\
\RomanIII_{4} \lesssim&  l^{2(\frac{1}{2} - 2 \delta)} \lVert \Lambda^{2- \gamma_{2}} z_{q} \ast_{x} \phi_{l} \rVert_{C_{t}^{\frac{1}{2} - 2 \delta}C_{x}} \lVert \nabla^{\bot} \cdot z_{q} \rVert_{C_{t}^{\frac{1}{2} - 2 \delta}C_{x}} + l^{2} \lVert \Lambda^{2- \gamma_{2}} z_{q}\rVert_{C_{t}C_{x}^{1}} \lVert \nabla^{\bot} z_{q} \rVert_{C_{t}C_{x}^{1}} \nonumber\\ 
\overset{\eqref{est 136}}{\lesssim}&l^{1-4\delta} L + l^{2} L^{\frac{1}{2}}.\label{est 183d} 
\end{align}
\end{subequations}  
We first observe that the upper bound of $\RomanIII_{3}$ in \eqref{est 183c} is already worse than that of $\RomanIII_{2}$ in \eqref{est 183b} because $\gamma_{2} \geq 1$. Moreover, because 
\begin{subequations}
\begin{align}
& \lambda_{q+1}^{-2\alpha} \lambda_{q}^{2-\beta} \lesssim \lambda_{q+1}^{-2\alpha} \lambda_{q}^{9 - 3 \gamma_{2} - 4 \beta} \text{ due to } \beta \leq \frac{7-3\gamma_{2}}{3} \text{ guaranteed by \eqref{est 154b}},\\
& \lambda_{q+1}^{-\alpha} \lesssim \lambda_{q+1}^{-2\alpha} \lambda_{q}^{9 - 3 \gamma_{2} - 4\beta} \text{ due to } \alpha \leq \frac{9 - 3 \gamma_{2} - 4 \beta}{2} \text{ guaranteed by \eqref{est 154c}},
\end{align}
\end{subequations} 
we can take $\delta \in (0, \frac{1}{4})$ arbitrarily small and see that for $a \in 5 \mathbb{N}$ sufficiently large 
\begin{align}
&\lVert R_{\text{Com1}} \rVert_{C_{t,x}} \lesssim l^{2} \lambda_{q}^{9 - 3 \gamma_{2} - 4 \beta} L^{2} M_{0}(t)^{2} +  l^{1- 4 \delta}  \lambda_{q}^{1 + (3- \gamma_{2}) (\frac{1}{2} - 2 \delta)} L^{\frac{5}{4} - \delta} M_{0}(t)^{\frac{3}{4} - \delta} \delta_{q}^{\frac{3}{4} - \delta}  \label{est 308} \\
& \hspace{10mm} \overset{\eqref{est 32}}{\lesssim} L M_{0}(t ) \lambda_{q+2}^{2- \gamma_{2}} \delta_{q+2}  [a^{b^{q} [b^{2} (-2+ \gamma_{2} + 2 \beta) - b 2 \alpha + 9 - 3 \gamma_{2} - 4 \beta]} L M_{0}(t) \nonumber\\
& \hspace{30mm} + a^{b^{q} [b^{2} (-2+ \gamma_{2} + 2 \beta) - b \alpha + \frac{5}{2} - \frac{\gamma_{2}}{2} - \frac{3\beta}{2} ]} \lambda_{q+1}^{\alpha 4 \delta} \lambda_{q}^{(-6 + 2 \gamma_{2} + 2 \beta) \delta}]  \ll LM_{0}(t) \lambda_{q+2}^{2- \gamma_{2}} \delta_{q+2} \nonumber
\end{align} 
where the last inequality used that \eqref{est 154c} guarantees that $\frac{1+ \gamma_{2} + 4 \beta}{4} < \alpha$ and $\frac{ - 11 + 7 \gamma_{2} + 13 \beta}{4} < \alpha$, former of which relies on $\beta < \frac{3-\gamma_{2}}{4}$ from \eqref{est 154b} when $\gamma_{2} > 1$. 

\subsubsection{Bounds on $R_{Com2}$, the second commutator error}\label{Section 4.1.6}
We split $R_{\text{Com2}}$ from \eqref{est 191e} to 
\begin{equation}
R_{\text{Com2}} = \sum_{k=1}^{14} R_{\text{Com,2k}}
\end{equation} 
where 
\begin{subequations}\label{est 296} 
\begin{align} 
R_{\text{Com2,1}} \triangleq& \mathcal{B} ( \Lambda^{2-\gamma_{2}} y_{q+1} \cdot \nabla (z_{q+1} - z_{q}) + \Lambda^{2-\gamma_{2}} (z_{q+1} - z_{q}) \cdot \nabla y_{q+1} ), \label{est 296a} \\
R_{\text{Com2,2}} \triangleq& \mathcal{B} ( (\nabla y_{l})^{T} \cdot \Lambda^{2-\gamma_{2}} (z_{q} - z_{q+1}) ), \label{est 296b} \\
R_{\text{Com2,3}} \triangleq& \mathcal{B} ( (\nabla (z_{q} - z_{q+1}))^{T} \cdot \Lambda^{2-\gamma_{2}} y_{l} ), \label{est 296c} \\
R_{\text{Com2,4}} \triangleq& \mathcal{B} ( (\nabla \Lambda^{2-\gamma_{2}} (z_{q+1} - z_{q}))^{T} \cdot w_{q+1} ), \label{est 296d} \\
R_{\text{Com2,5}} \triangleq& -\mathcal{B} ( \Lambda^{2-\gamma_{2}} w_{q+1}^{\bot} \nabla^{\bot} \cdot (z_{l} - z_{q}) ), \label{est 296e} \\
R_{\text{Com2,6}} \triangleq& \mathcal{B} ( (\nabla (z_{q} - z_{q+1}))^{T} \cdot \Lambda^{2-\gamma_{2}} w_{q+1}), \label{est 296f} \\
R_{\text{Com2,7}} \triangleq& \mathcal{B} ( \Lambda^{2-\gamma_{2}} y_{l}^{\bot} \nabla^{\bot} \cdot (z_{q} - z_{l})), \label{est 296g} \\
R_{\text{Com2,8}} \triangleq& \mathcal{B} (  \Lambda^{2-\gamma_{2}} (z_{q} - z_{l}) \cdot \nabla w_{q+1}), \label{est 296ih} \\
R_{\text{Com2,9}} \triangleq& \mathcal{B} ( \Lambda^{2-\gamma_{2}} (z_{q} - z_{l})^{\bot} \nabla^{\bot} \cdot y_{l}), \label{est 296i} \\
R_{\text{Com2,10}} \triangleq& \mathcal{B} ( \Lambda^{2-\gamma_{2}} (z_{q+1} - z_{q})^{\bot} \nabla^{\bot} \cdot z_{q+1}), \label{est 296j} \\
R_{\text{Com2,11}} \triangleq& \mathcal{B} ( \Lambda^{2-\gamma_{2}} (z_{q} - z_{l})^{\bot}\nabla^{\bot} \cdot z_{q+1} ), \label{est 296k} \\
R_{\text{Com2,12}} \triangleq& \mathcal{B} ( \Lambda^{2-\gamma_{2}} z_{l}^{\bot} \nabla^{\bot} \cdot (z_{q+1} - z_{q}) ), \label{est 296l} \\
R_{\text{Com2,13}} \triangleq& \mathcal{B} ( \Lambda^{2-\gamma_{2}} z_{l}^{\bot} \nabla^{\bot} \cdot (z_{q} - z_{l})). \label{est 296m}
\end{align}
\end{subequations} 
The key observation is that for $R_{\text{Com2,k}}$, $k \in \{5, 8\}$,  identically to how we handled $R_{T}$ (recall \eqref{est 219}), $\supp (z_{l} - z_{q})\hspace{1mm}\hat{}\hspace{1mm}  \subset B(0, \frac{\lambda_{q}}{4})$ while $\supp (\tilde{w}_{q+1, j, k})\hspace{1mm}\hat{}\hspace{1mm}  \subset \{ \frac{7}{8} \lambda_{q+1} \leq \lvert \xi \rvert \leq \frac{9}{8} \lambda_{q+1} \}$ and hence the nonlinear terms therein are equivalent to those with $\tilde{P}_{\approx \lambda_{q+1}}$ applied which gives us a factor of $\lambda_{q+1}^{-1}$ due to $\mathcal{B}$. For terms with $z_{q+1} - z_{q}$, we can use the fact that $\supp (z_{q+1} - z_{q}) \hspace{1mm} \hat{} \subset \{\xi: \frac{\lambda_{q}}{4} \leq \lvert \xi \rvert \leq \frac{\lambda_{q+1}}{4}\} $. With these in mind, we can estimate for any $\delta \in (0,\frac{1}{4})$ 
\begin{subequations}\label{est 429}
\begin{align}
&\lVert R_{\text{Com2,1}} \rVert_{C_{t,x}} \lesssim   \lVert \Lambda^{2-\gamma_{2}} y_{q+1} \rVert_{C_{t,x}} \lVert z_{q+1} - z_{q} \rVert_{C_{t,x}} + \lVert \Lambda^{2-\gamma_{2}} (z_{q+1} - z_{q}) \rVert_{C_{t,x}} \lVert y_{q+1} \rVert_{C_{t,x}} \nonumber \\
& \hspace{20mm} \overset{\eqref{est 166} \eqref{est 136}}{\lesssim} M_{0}(t) \lambda_{q+2}^{2-\gamma_{2}} \delta_{q+2} \lambda_{q+2}^{-2 + \gamma_{2} + 2 \beta} \lambda_{q}^{-1-\gamma_{2}}, \\
& \lVert R_{\text{Com2,2}}\rVert_{C_{t,x}} + \lVert R_{\text{Com2,9}} \rVert_{C_{t,x}} \lesssim \lVert \nabla y_{l} \rVert_{C_{t,x}} [  \lVert \lambda^{2- \gamma_{2}} (z_{q} - z_{q+1}) \rVert_{C_{t,x}} + \lVert \Lambda^{2-\gamma_{2}} (z_{q} - z_{l}) \rVert_{C_{t,x}} ]  \nonumber \\
& \hspace{15mm} \overset{\eqref{est 136} \eqref{est 166}\eqref{est 154}}{\lesssim} M_{0}(t) \lambda_{q+2}^{2-\gamma_{2}} \delta_{q+2} a^{b^{q} [(b-2) [(b+2) (-2+ \gamma_{2} + 2 \beta) - \frac{\alpha}{2} ] - 7 + 4 \gamma_{2} + 7 \beta - \alpha + b \alpha 2 \delta]},  \\
& \lVert R_{\text{Com2,3}} \rVert_{C_{t,x}} + \lVert R_{\text{Com2,7}} \rVert_{C_{t,x}}  \overset{\eqref{est 166} \eqref{est 136}}{\lesssim}   M_{0}(t) (\lambda_{q+1}^{-\alpha (\frac{1}{2} - 2 \delta)} + \lambda_{q}^{-2}) \lambda_{q}^{2- \gamma_{2}} \delta_{q}^{\frac{1}{2}} \nonumber \\
& \hspace{10mm} \approx M_{0}(t) \lambda_{q+2}^{2- \gamma_{2}} \delta_{q+2} [a^{b^{q} [ b^{2} (-2 + \gamma_{2} + 2 \beta) - \gamma_{2} - \beta]} + a^{b^{q} [b^{2} (-2 + \gamma_{2} + 2 \beta) + 2 - \gamma_{2} - \beta - b \alpha (\frac{1}{2} - 2 \delta) ]}],  \\
&\lVert R_{\text{Com2,4}} \rVert_{C_{t,x}} \overset{\eqref{est 204a} \eqref{est 136}}{\lesssim} \lambda_{q}^{ - \gamma_{2}}  M_{0}(t) \delta_{q+1}^{\frac{1}{2}}  \nonumber\\
& \hspace{25mm} \lesssim  M_{0}(t) \lambda_{q+2}^{2- \gamma_{2}} \delta_{q+2} a^{b^{q} [(b-2) [(b+2) (-2 + \gamma_{2} + 2 \beta) - \beta] - 8 + 3 \gamma_{2} + 6 \beta]},  \\
&\lVert R_{\text{Com2,5}} \rVert_{C_{t,x}} \lesssim \lambda_{q+1}^{-1} \lVert \Lambda^{2- \gamma_{2}} w_{q+1} \rVert_{C_{t,x}} \lVert \nabla (z_{l} - z_{q} ) \rVert_{C_{t,x}} \nonumber \\
& \hspace{20mm} \overset{\eqref{est 163} \eqref{est 204a} \eqref{est 136}}{\lesssim} M_{0}(t) \lambda_{q+2}^{2-\gamma_{2}} \delta_{q+2} a^{b^{q+1} [(b-2) (-2+ \gamma_{2} + 2 \beta) - 3 + \gamma_{2} + 3 \beta - \alpha (\frac{1}{2} - 2 \delta) ]},  \\
&\lVert R_{\text{Com2,6}} \rVert_{C_{t,x}} \overset{\eqref{est 163}\eqref{est 204a} }{\lesssim} \lambda_{q}^{-2} \lambda_{q+1}^{1- \gamma_{2}} \delta_{q+1}^{\frac{1}{2}} M_{0}(t)  \nonumber\\
& \hspace{25mm} \lesssim M_{0}(t) \lambda_{q+2}^{2- \gamma_{2}}\delta_{q+2}  a^{b^{q} [b^{2} (-2+ \gamma_{2} + 2 \beta) - 2 + b(1- \gamma_{2} - \beta)]},    \\
& \lVert R_{\text{Com2,8}} \rVert_{C_{t,x}} \lesssim \lVert \Lambda^{2-\gamma_{2}} (z_{q} - z_{l}) \rVert_{C_{t,x}}  \lVert w_{q+1} \rVert_{C_{t,x}}  \nonumber \\
& \hspace{20mm} \overset{\eqref{est 204a} \eqref{est 32} \eqref{est 136}}{\lesssim}   M_{0}(t) \lambda_{q+2}^{2-\gamma_{2}} \delta_{q+2} a^{b^{q+1} [(b-2) (-2+ \gamma_{2} + 2 \beta) - 4 + 2 \gamma_{2} + 3 \beta - \frac{\alpha}{2} + 2 \alpha \delta]},   \\
&\sum_{k=10}^{13} \lVert R_{\text{Com2,k}} \rVert_{C_{t,x}} \lesssim L^{\frac{1}{4}} (\lVert \nabla (z_{q+1} - z_{q}) \rVert_{C_{t,x}} + \lVert \nabla (z_{q} - z_{l}) \rVert_{C_{t,x}}) \nonumber\\
& \hspace{20mm} \overset{\eqref{est 136}}{\lesssim} M_{0}(t) \lambda_{q+2}^{2- \gamma_{2}} \delta_{q+2} [a^{b^{q} [(b-2) (b+2) (-2+ \gamma_{2} + 2 \beta) - 10 + 4 \gamma_{2} + 8 \beta]} \nonumber\\
& \hspace{46mm} + a^{b^{q+1} [(b-2)(-2+\gamma_{2} + 2 \beta) - 4 + 2 \gamma_{2} + 4 \beta - \frac{\alpha}{2} + \alpha 2 \delta]}]. 
\end{align}
\end{subequations}
Because 
\begin{align*}
\max\{ -7 + 4 \gamma_{2} + 7 \beta,  -6 + 3 \gamma_{2} + 7 \beta, 2 \gamma_{2} - 6 + 6 \beta,  -8 + 4 \gamma_{2}+ 6 \beta,  -8 + 4 \gamma_{2} + 8 \beta  \} < \alpha 
\end{align*}
all due to \eqref{est 154}, we can take $\beta > 1 - \frac{\gamma_{2}}{2}$ sufficiently close to $1 - \frac{\gamma_{2}}{2}$, $\delta \in (0,\frac{1}{4})$ sufficiently small, and $a\in 5 \mathbb{N}$ sufficiently large to conclude from \eqref{est 429} that for $a \in 5 \mathbb{N}$ sufficiently large 
\begin{equation}\label{est 300} 
\lVert R_{\text{Com2}} \rVert_{C_{t,x}} \ll LM_{0}(t)\lambda_{q+2}^{2-\gamma_{2}} \delta_{q+2}.
\end{equation}  

\subsubsection{Concluding the Proof of Proposition \ref{Proposition 4.8}}
Applying \eqref{est 301}, \eqref{est 305}, \eqref{est 306}, \eqref{est 307}, \eqref{est 308}, and \eqref{est 300} to \eqref{est 190} allows us to conclude 
\begin{equation*}
\lVert \mathring{R}_{q+1} \rVert_{C_{t,x}} \leq \lVert R_{T} + R_{N} + R_{L} + R_{O} + R_{\text{Com1}} + R_{\text{Com2}} \rVert_{C_{t,x}} \ll L M_{0}(t) \lambda_{q+2}^{2-\gamma_{2}} \delta_{q+2}
\end{equation*} 
and hence the Hypothesis \ref{Hypothesis 4.1} \ref{Hypothesis 4.1 (b)} at level $q+1$.

Finally, following previous works, we can readily verify that $y_{q+1}, \mathring{R}_{q+1}$ are $(\mathcal{F}_{t})_{t\geq 0}$-adapted and $y_{q+1}(0,x), \mathring{R}_{q+1}(0,x)$ are deterministic under the assumption that $y_{q}, \mathring{R}_{q}$ are $(\mathcal{F}_{t})_{t\geq 0}$-adapted and $y_{q}(0,x), \mathring{R}_{q}(0,x)$ are deterministic. First, $z(t)$ from \eqref{est 152} is $(\mathcal{F}_{t})_{t\geq 0}$-adapted and consequently, so is $z_{q} = z \ast_{x} \tilde{\phi}_{q}$ and $z_{q+1} = z \ast_{x} \tilde{\phi}_{q+1}$ defined in \eqref{est 150}. Due to the compact support of $\varphi_{l}$ in $(\tau_{q+1}, 2 \tau_{q+1}] \subset \mathbb{R}_{+}$ and because we extended $y_{q}$ and $z_{q}$ for all $q \in \mathbb{N}_{0}$ to $t < 0$ by their respective values at $t = 0$, it follows that 
\begin{equation}\label{est 399} 
y_{l}(t) \overset{\eqref{est 180}}{=} \int_{\tau_{q+1}}^{2 \tau_{q+1}} (y_{q} \ast_{x} \phi_{l})(t-s, x) \varphi_{l}(s) ds,  z_{l}(t) \overset{\eqref{est 180}}{=} \int_{\tau_{q+1}}^{2\tau_{q+1}} (z_{q} \ast_{x} \phi_{l})(t-s,x) \varphi_{l}(s) ds,
\end{equation} 
and $\mathring{R}_{l}$ defined in \eqref{est 180} are all $(\mathcal{F}_{t})_{t\geq 0}$-adapted. Considering $z_{l}$ and $y_{l}$ in \eqref{est 399}, we see that 
$\Phi_{j}$ from \eqref{est 193} and $\mathring{R}_{q,j}$ from \eqref{est 195} are both $(\mathcal{F}_{t})_{t\geq 0}$-adapted. This implies that because $b_{k}$ from \eqref{est 178} is deterministic, we see that $b_{k} (\lambda_{q+1} \Phi_{j}(t,x))$ is $(\mathcal{F}_{t})_{t\geq 0}$-adapted; similarly, because $M_{0}$ from \eqref{est 135}, $\chi_{j}$ from \eqref{est 209}, and $\gamma_{k}$ from Lemma \ref{Geometric Lemma} are all deterministic, $a_{k,j}$ from \eqref{est 196} is $(\mathcal{F}_{t})_{t\geq 0}$-adapted. Therefore,
\begin{align}
w_{q+1}(t,x) &= \sum_{j,k} \chi_{j} (t) \nonumber\\
& \times \mathbb{P}_{q+1, k} \left( \frac{ M_{0}(\tau_{q+1} j)^{\frac{1}{2}}}{\sqrt{\gamma_{2}}} \delta_{q+1}^{\frac{1}{2}} \gamma_{k} \left( \Id - \frac{ \mathring{R}_{q,j} (t,x)}{\lambda_{q+1}^{2-\gamma_{2}} \delta_{q+1} M_{0}(\tau_{q+1} j)} \right) b_{k} (\lambda_{q+1} \Phi_{j}(t,x)) \right)  \label{est 400}
\end{align} 
from \eqref{est 194} and \eqref{est 196}, as well as $\partial_{t} w_{q+1}$ are $(\mathcal{F}_{t})_{t\geq 0}$-adapted so that $y_{q+1}$ in \eqref{est 164} is $(\mathcal{F}_{t})_{t\geq 0}$-adapted. Moreover, all of $R_{T}, R_{O}, R_{N}, R_{L},$ and $R_{\text{Com2}}$ in \eqref{est 191}, as well as $R_{\text{Com1}}$ in \eqref{est 181} are $(\mathcal{F}_{t})_{t\geq 0}$-adapted, implying that $\mathring{R}_{q+1}$ in \eqref{est 190} is $(\mathcal{F}_{t})_{t\geq 0}$-adapted. 

Similarly, due to the compact support of $\varphi_{l}$ in $(\tau_{q+1}, 2 \tau_{q+1} ]$, if $v_{q}(0,x)$ and $\mathring{R}_{q}(0,x)$ are deterministic, then so are $v_{l}(0,x)$, $\mathring{R}_{l}(0,x)$, and $\partial_{t} \mathring{R}_{l} (0,x)$. Similarly, because $z(0,x) \equiv 0$ by \eqref{est 152}, $z_{q}(0,x) \equiv 0$ and hence so is $z_{l}(0,x) \equiv 0$. Moreover, within \eqref{est 400}, because $\supp \chi_{j} \subset (\tau_{q+1} (j-1), \tau_{q+1} (j+1))$, $\chi_{j}(0) = 0$ for all $j \in \{1, \hdots, \lceil \tau_{q+1}^{-1} T_{L} \rceil \}$ leaving only one term when $j = 0$. As $\Phi_{0}(0,x) = x$ due to \eqref{est 193} while $\mathring{R}_{q,0} (0,x) = \mathring{R}_{l}(0,x)$ where $\mathring{R}_{l}(0,x)$ was already verified to be deterministic, we conclude that $w_{q+1}(0,x)$ is deterministic. It follows similarly that $\partial_{t} w_{q+1}(0,x)$ is deterministic. It follows that $y_{q+1}(0,x)$ from \eqref{est 164} is deterministic, and all of $R_{T}(0,x)$, $R_{O}(0,x)$, $R_{N}(0,x)$, $R_{L}(0,x)$, $R_{\text{Com2}}(0,x)$ in \eqref{est 191}, and $R_{\text{Com1}}(0,x)$ in \eqref{est 181} are deterministic, allowing us to conclude that $\mathring{R}_{q+1}(0,x)$ from \eqref{est 190} is deterministic. 

\section{Appendix}\label{Appendix}
\subsection{Proof of Proposition \ref{Proposition 4.1}}
The existence of a martingale solution is relatively standard (e.g., \cite{FR08}); nonetheless, we sketch the proof because the QG momentum equations forced by random noise have not been studied before. In short, we rely on \cite[Theorem 4.6 on p. 1739]{GRZ09} similarly to  \cite[Section 4.2]{Z12a}. We choose in accordance with notations from \cite[p. 1729]{GRZ09} 
\begin{equation}\label{est 55} 
\mathbb{H} = \mathbb{Y} \triangleq \dot{H}_{\sigma}^{\frac{1}{2}} \hspace{1mm} \text{ and } \hspace{1mm} \mathbb{X} \triangleq (H_{\sigma}^{4})^{\ast}.
\end{equation} 
It follows that $\mathbb{Y} \subset \mathbb{H} \subset \mathbb{X}$ continuously and densely and $\mathbb{X}^{\ast} \hookrightarrow \mathbb{Y}$ compactly. Following \cite[p. 152]{Z12a} and \cite[equation (3.1) on p. 1730]{GRZ09} we define an operator $\mathcal{A}: C_{0,\sigma}^{\infty} \mapsto \mathbb{X}$ by 
\begin{equation}\label{est 56} 
\mathcal{A}v \triangleq - (u\cdot\nabla) v + (\nabla v)^{T} \cdot u - \Lambda^{\gamma_{1}} v, \hspace{3mm} u = \Lambda v. 
\end{equation} 
The following proposition is analogous to \cite[Lemma 4.2.3]{Z12a} and \cite[Lemma 6.1]{GRZ09} and relies on the Calder$\acute{\mathrm{o}}$n commutator estimate from Proposition \ref{Proposition 6.3}.

\begin{proposition}\label{Proposition 6.1}
\rm{ (cf. \cite[Lemma 4.2.3]{Z12a} and \cite[Lemma 6.1]{GRZ09})} For any $v, \tilde{v} \in C_{0,\sigma}^{\infty}$  
\begin{subequations}\label{est 57}
\begin{align}
& \lVert \Lambda^{\gamma_{1}} v - \Lambda^{\gamma_{1}} \tilde{v} \rVert_{\mathbb{X}} \lesssim \lVert v - \tilde{v} \rVert_{\dot{H}_{x}^{\frac{1}{2}}}, \label{est 57a}\\
& \lVert - (u \cdot \nabla) v + (\tilde{u} \cdot \nabla)\tilde{v} \rVert_{\mathbb{X}} + \lVert (\nabla v)^{T} \cdot u - (\nabla \tilde{v})^{T} \cdot \tilde{u} \rVert_{\mathbb{X}} \lesssim \lVert v - \tilde{v} \rVert_{\dot{H}_{x}^{\frac{1}{2}}} (\lVert v \rVert_{\dot{H}_{x}^{\frac{1}{2}}} + \lVert \tilde{v} \rVert_{\dot{H}_{x}^{\frac{1}{2}}}).  \label{est 58}
\end{align}
\end{subequations}
Therefore, the operator $\mathcal{A}: C_{0,\sigma}^{\infty} \mapsto \mathbb{X}$ extends to an operator $\mathcal{A}: \dot{H}_{\sigma}^{\frac{1}{2}} \mapsto \mathbb{X}$ by continuity. 
\end{proposition}  

\begin{proof}
First, because $v, \tilde{v} \in C_{0,\sigma}^{\infty}$ are both mean-zero, we can compute 
\begin{equation}
\lVert \Lambda^{\gamma_{1}} v - \Lambda^{\gamma_{2}} \tilde{v} \rVert_{\mathbb{X}} 
\overset{\eqref{est 55}}{=} \lVert \Lambda^{\gamma_{1}}v - \Lambda^{\gamma_{1}} \tilde{v} \rVert_{H_{x}^{-4}}  \lesssim \lVert  v- \tilde{v} \rVert_{\dot{H}_{x}^{\frac{1}{2}}}. 
\end{equation} 
Next, we can compute using \eqref{est 79} and the embedding of $H^{4} (\mathbb{T}^{2}) \hookrightarrow W^{\frac{3}{2}, \infty} (\mathbb{T}^{2})$ 
\begin{align}
& \lVert - (u \cdot \nabla) v + (\tilde{u} \cdot \nabla) \tilde{v} \rVert_{\mathbb{X}} \label{est 407} \\ 
\lesssim& \sup_{ w \in H_{\sigma}^{4}: \lVert w \rVert_{H_{x}^{4}} \leq 1} [ \lVert v - \tilde{v} \rVert_{\dot{H}_{x}^{\frac{1}{2}}} \lVert \nabla w \cdot v \rVert_{\dot{H}_{x}^{\frac{1}{2}}} + \lVert \tilde{v} \rVert_{\dot{H}_{x}^{\frac{1}{2}}} \lVert \nabla w \cdot (v - \tilde{v}) \rVert_{\dot{H}_{x}^{\frac{1}{2}}}] \lesssim \lVert v - \tilde{v} \rVert_{\dot{H}_{x}^{\frac{1}{2}}} ( \lVert v \rVert_{\dot{H}_{x}^{\frac{1}{2}}} + \lVert \tilde{v} \rVert_{\dot{H}_{x}^{\frac{1}{2}}}). \nonumber 
\end{align}
Finally, we compute via \eqref{est 59} and \eqref{est 50} which are crucial as this nonlinear term is not of the divergence form:
\begin{align*} 
&\lVert (\nabla v)^{T} \cdot u - (\nabla \tilde{v})^{T} \cdot \tilde{u} \rVert_{\mathbb{X}}  \nonumber \\
\overset{\eqref{est 59}}{=}& \sup_{w \in H_{\sigma}^{4}: \lVert w \rVert_{H_{x}^{4}} \leq 1} \frac{1}{2} \left\lvert \sum_{i,j=1}^{2} \int_{\mathbb{T}^{2}} (\partial_{i} v_{j} - \partial_{i} \tilde{v}_{j})[ \Lambda, w_{i}] v_{j} + \partial_{i} \tilde{v}_{j} [\Lambda, w_{i}] (v_{j} - \tilde{v}_{j}) dx \right\rvert \nonumber \\
\overset{\eqref{est 50} }{\lesssim}& \sup_{w \in H_{\sigma}^{4}: \lVert w \rVert_{H_{x}^{4}} \leq 1} \sum_{i,j=1}^{2} \lVert v_{j} - \tilde{v}_{j} \rVert_{\dot{H}_{x}^{\frac{1}{2}}} \lVert w_{i} \rVert_{W_{x}^{2,\infty}} \lVert v_{j} \rVert_{\dot{H}_{x}^{\frac{1}{2}}} + \lVert \tilde{v}_{j} \rVert_{\dot{H}_{x}^{\frac{1}{2}}} \lVert w_{i} \rVert_{W_{x}^{2,\infty}} \lVert v_{j} - \tilde{v}_{j} \rVert_{\dot{H}_{x}^{\frac{1}{2}}} \nonumber \\
\lesssim& \lVert v - \tilde{v} \rVert_{\dot{H}_{x}^{\frac{1}{2}}} ( \lVert v \rVert_{\dot{H}_{x}^{\frac{1}{2}}} + \lVert \tilde{v} \rVert_{\dot{H}_{x}^{\frac{1}{2}}}). 
\end{align*} 
\end{proof}
Following \cite[p. 1730 and 1733]{GRZ09} and \cite[p. 154]{Z12a} we define the function $\mathcal{N}_{1}$ and $\mathcal{N}_{p}$ for $p \geq 2$ on $\mathbb{Y}$  by 
\begin{equation}\label{est 62}
\mathcal{N}_{1}(v) \triangleq 
\begin{cases}
\lVert v \rVert_{\dot{H}_{x}^{\frac{1}{2} + \frac{\gamma_{1}}{2}}}^{2} & \text{ if } v \in \dot{H}^{\frac{1}{2} + \frac{\gamma_{1}}{2}}(\mathbb{T}^{2}), \\
+ \infty & \text{ if } v \notin \dot{H}^{\frac{1}{2} + \frac{\gamma_{1}}{2}}(\mathbb{T}^{2}), 
\end{cases} \hspace{5mm} \text{ and } \hspace{5mm}  \mathcal{N}_{p}(v) \triangleq  \lVert v \rVert_{\dot{H}_{x}^{\frac{1}{2}}}^{2(p-1)} \mathcal{N}_{1}(v). 
\end{equation} 
Under these notations, we can show that the criterion of ``(C1) (Demi-Continuity)'' and ``(C3) (Growth Condition)'' on \cite[p. 1733]{GRZ09} hold due to \eqref{est 60} and Proposition \ref{Proposition 6.1} while ``(C2) (Coercivity Condition)'' on \cite[p. 1733]{GRZ09} holds by definition of $\mathcal{N}_{1}$. Therefore, the hypothesis of \cite[Theorem 4.6]{GRZ09} holds, and it allows us to deduce Proposition \ref{Proposition 4.1} (1). 

In order to prove the Proposition \ref{Proposition 4.1} (2), we need the following result first. 
\begin{proposition}\label{Proposition 6.2} 
\rm{(\cite[Lemma A.1]{HZZ19})} Let $\{(s_{l},\xi_{l})\}_{l\in\mathbb{N}} \subset [0,\infty) \times \dot{H}_{\sigma}^{\frac{1}{2}}$ be a family such that $\lim_{l\to\infty} \lVert (s_{l},\xi_{l}) - (s,\xi^{\text{in}}) \rVert_{\mathbb{R} \times \dot{H}_{x}^{\frac{1}{2}}} = 0$ and $\{P_{l} \}_{l\in\mathbb{N}}$ be a family of probability measures on $\Omega_{0}$ satisfying $P_{l} (\{ \xi(t) = \xi_{l} \hspace{1mm} \forall \hspace{1mm} t \in [0, s_{l} ] \}) = 1$ for all $l \in \mathbb{N}$, any $T > 0$, and some $\iota, \kappa > 0$, 
\begin{equation}\label{est 63}  
\sup_{n\in\mathbb{N}} \mathbb{E}^{P_{l}} \left[ \lVert \xi \rVert_{C([0,T]; \dot{H}_{x}^{\frac{1}{2}})} + \sup_{r,t \in [0,T]: r \neq t} \frac{ \lVert \xi(t) - \xi(r) \rVert_{H_{x}^{-4}}}{\lvert t-r \rvert^{\kappa}} + \lVert \xi \rVert_{L^{2}([s_{l}, T]; \dot{H}_{x}^{\frac{1}{2} + \iota })}^{2} \right] < \infty. 
\end{equation} 
Then $\{P_{l} \}_{l\in\mathbb{N}}$ is tight in 
\begin{equation}\label{est 77}
\mathbb{S} \triangleq C_{\text{loc}} ([0,\infty); (H_{\sigma}^{4})^{\ast}) \cap L_{\text{loc}}^{2} ([0,\infty); \dot{H}_{\sigma}^{\frac{1}{2}}).
\end{equation}   
\end{proposition} 

\begin{proof}[Proof of Proposition \ref{Proposition 6.2}]
This is an extension of \cite[Lemma A.1]{HZZ19} on the 3-d NS equations that has already been generalized to 2-d NS equations in \cite[Lemma 6.4]{Y20c} and $n$-d Boussinesq system in \cite[Proposition 6.6]{Y21a}; thus, we only sketch its proof. We recall that a set $K \subset \mathbb{S}$ is compact if  $\{ f \rvert_{[0,T]}: f \in K \} \subset C([0, T]; (H_{\sigma}^{4})^{\ast})) \cap L^{2}(0, T; \dot{H}_{\sigma}^{\frac{1}{2}})$ is compact for every $T > 0$. Now we fix $\epsilon > 0$ and $k \in \mathbb{N}$ such that $k \geq k_{0} \triangleq \sup_{l\in\mathbb{N}} s_{l}$. Due to \eqref{est 63} we may choose $R_{K} > 0$ sufficiently large such that 
\begin{align}
P_{l} ( \{ \xi \in \Omega_{0}: &\sup_{t \in [0, k]} \lVert \xi(t) \rVert_{\dot{H}_{x}^{\frac{1}{2}}}  \nonumber  \\
&+ \sup_{r, t \in [0,k]: r \neq t} \frac{ \lVert \xi(t) - \xi(r) \rVert_{H_{x}^{-4}}}{\lvert t-r \rvert^{\kappa}} + \int_{s_{l}}^{k} \lVert \xi(r) \rVert_{\dot{H}_{x}^{\frac{1}{2} + \iota}}^{2} dr > R_{k} \}) \leq \frac{\epsilon}{2^{k}} \label{est 64} 
\end{align} 
and define 
\begin{equation}\label{est 67}
\Omega_{l} \triangleq \{ \xi \in \Omega_{0}: \xi(t) = \xi_{l} \hspace{1mm} \forall \hspace{1mm} t \in [0, s_{l}] \}) 
\end{equation} 
and 
\begin{align}
K \triangleq& \cup_{q \in \mathbb{N}} \cap_{k \in\mathbb{N}: k \geq k_{0}} \{ \xi \in \Omega_{q}: \sup_{t \in [0,k]} \lVert \xi(t) \rVert_{\dot{H}_{x}^{\frac{1}{2}}} \nonumber \\
& \hspace{10mm}+ \sup_{r, t \in [0, k]: r \neq t} \frac{ \lVert \xi(t) - \xi(r) \rVert_{H_{x}^{-4}}}{\lvert t-r \rvert^{\kappa}}  + \int_{s_{q}}^{k} \lVert \xi(r) \rVert_{\dot{H}_{x}^{\frac{1}{2} + \iota}}^{2} dr \leq R_{k} \}. \label{est 65} 
\end{align}
Then we can compute $\sup_{l \in \mathbb{N}} P_{l} (\Omega_{0} \setminus \bar{K}) \leq \epsilon$ using \eqref{est 64}. By definition of tightness, if we now show that $\bar{K}$ is compact in $\mathbb{S}$, then it implies that $\{P_{l}\}_{l\in\mathbb{N}}$ is tight in $\mathbb{S}$ as desired. We take $\{\xi_{w}\}_{w \in \mathbb{N}} \subset K$. Suppose that for all $N \in \mathbb{N}, \xi_{w} \in \Omega_{N}$ for only finitely many $w \in \mathbb{N}$. Then passing to a subsequence and relabeling if necessary, we can assume that $\xi_{w} \in \Omega_{w}$. Then, for all $k \geq k_{0}$, 
\begin{equation}
\sup_{t \in [0,k]} \lVert \xi_{w} (t) \rVert_{\dot{H}_{x}^{\frac{1}{2}}} + \sup_{r,t \in [0,k]: r \neq t} \frac{ \lVert \xi_{w} (t) - \xi_{w} (r) \rVert_{H_{x}^{-4}}}{\lvert t-r \rvert^{\kappa}} \leq R_{k} 
\end{equation} 
because $\xi_{w} \in K$. Now 
\begin{equation}\label{est 68}
L^{\infty} (0, k; \dot{H}_{\sigma}^{\frac{1}{2}}) \cap C^{\kappa} ([0,k]; (H_{\sigma}^{4})^{\ast}) \hookrightarrow C([0,k]; (H_{\sigma}^{4})^{\ast})
\end{equation} 
is compact (e.g., \cite{S87}). Therefore, we can extract a subsequence $\{\xi_{w_{l}}\}_{l}$ such that 
\begin{equation}\label{est 66}
\lim_{l,q\to\infty} \sup_{t\in [0,k]} \lVert \xi_{w_{l}}(t) - \xi_{w_{q}}(t) \rVert_{H_{x}^{-4}} = 0. 
\end{equation}
It follows from \eqref{est 67}, \eqref{est 65}, and \eqref{est 66} that for all $\delta > 0$, there exists $L \in \mathbb{N}$ such that $w_{l}, w_{q} \geq L$ implies 
\begin{equation}\label{est 69}
\int_{0}^{k} \lVert \xi_{w_{l}}(t) - \xi_{w_{q}}(t) \rVert_{\dot{H}_{x}^{\frac{1}{2}}}^{2} dt = \int_{0}^{s_{w_{l}} \wedge s_{w_{q}}} + \int_{s_{w_{l}} \wedge s_{w_{q}}}^{s_{w_{l}} \vee s_{w_{q}}} + \int_{s_{w_{l}} \vee s_{w_{q}}}^{k} \lVert \xi_{w_{l}} (t) - \xi_{w_{q}}(t) \rVert_{\dot{H}_{x}^{\frac{1}{2}}}^{2} dt  < \delta. 
\end{equation} 
The case in which there exists $N \in \mathbb{N}$ such that $\xi_{w} \in \Omega_{N}$ for infinitely many $w$ is similar and easier; thus, we conclude, along with \eqref{est 66}, that $\{\xi_{w_{l}}\}_{l}$ is convergent in $C([0,k]; (H_{\sigma}^{4})^{\ast}) \cap L^{2}(0, k; \dot{H}_{\sigma}^{\frac{1}{2}})$. By the arbitrariness of $\{\xi_{w}\} \subset K$ we conclude that $K$ is compact. 
\end{proof} 

With Propositions \ref{Proposition 6.1}-\ref{Proposition 6.2}, we are ready to deduce Proposition \ref{Proposition 4.1} (2). Due to similarity with previous works (e.g., \cite[Proof of Theorem 3.1]{HZZ19}, \cite[Proof of Proposition 4.1]{Y20a}) we only sketch its proof. We fix $\{P_{l} \}_{l\in\mathbb{N}} \subset \mathcal{C} ( s_{l}, \xi_{l}, \{C_{t,q} \}_{q\in \mathbb{N}, t \geq s_{l}})$ where $\{(s_{l}, \xi_{l})\}_{l\in\mathbb{N}} \subset [0,\infty) \times \dot{H}_{\sigma}^{\frac{1}{2}}$ satisfies $\lim_{l\to\infty} \lVert (s_{l},\xi_{l}) - (s, \xi^{\text{in}}) \rVert_{\mathbb{R} \times \dot{H}_{x}^{\frac{1}{2}}} = 0$. To verify the hypothesis of Proposition \ref{Proposition 6.2}, we first note that $P_{l} ( \{ \xi(t) = \xi_{l} \hspace{1mm} \forall \hspace{1mm} t \in [0, s_{l} ]\}) = 1$ for all $l \in \mathbb{N}$ due to (M1) of Definition \ref{Definition 4.1}. Second via Proposition \ref{Proposition 6.1} we define $F: \dot{H}_{\sigma}^{\frac{1}{2}} \mapsto (H_{\sigma}^{4})^{\ast}$ by 
\begin{equation}\label{est 71}
F(\xi) \triangleq - \mathbb{P} [(\Lambda \xi \cdot \nabla) \xi - (\nabla \xi)^{T} \cdot \Lambda \xi] - \Lambda^{\gamma_{1}} \xi. 
\end{equation} 
By definition of $\mathcal{C} (s_{l}, \xi_{l}, \{C_{t,q} \}_{q\in\mathbb{N}, t \geq s_{l}})$ and (M2) of Definition \ref{Definition 4.1} we know that for all $l \in \mathbb{N}$ and $t \in [s_{l}, \infty)$, the mapping $t \mapsto M_{t, s_{l}}^{\xi, k}$ where 
\begin{align}
M_{t, s_{l}}^{\xi, k} \triangleq& \langle \xi(t) - \xi_{l}, \psi^{k} \rangle \nonumber \\
&- \int_{s_{l}}^{t} \sum_{i,j=1}^{2} \langle \Lambda \xi_{i}, \partial_{i} \psi_{j}^{k} \xi_{j} \rangle_{\dot{H}_{x}^{-\frac{1}{2}} - \dot{H}_{x}^{\frac{1}{2}}} - \frac{1}{2} \langle \partial_{i} \xi_{j}, [\Lambda, \psi_{i}^{k} ]\xi_{j} \rangle_{\dot{H}_{x}^{-\frac{1}{2}} - \dot{H}_{x}^{\frac{1}{2}}} - \langle \xi, \Lambda^{\gamma_{1}} \psi^{k} \rangle dr \label{est 82} 
\end{align} 
for $\psi^{k} \in C^{\infty} (\mathbb{T}^{2}) \cap \dot{H}_{\sigma}^{\frac{1}{2}}$ and $\xi \in \Omega_{0}$ is a continuous, square-integrable $(\mathcal{B}_{t})_{t\geq s_{l}}$-martingale under $P_{l}$ and $\langle \langle M_{t, s_{l}}^{\xi, k} \rangle \rangle = \int_{s_{l}}^{t} \lVert G(\xi(r))^{\ast} \psi^{k} \rVert_{U}^{2}dr$. We can thus write for all $\kappa \in (0, \frac{1}{2})$ due to (M2) of Definition \ref{Definition 4.1} and \eqref{est 71}
\begin{align}
&\sup_{l\in\mathbb{N}} \mathbb{E}^{P_{l}} [ \sup_{r,t \in [0,T]: r \neq t} \frac{ \lVert \xi(t) - \xi(r) \rVert_{H_{x}^{-4}}}{\lvert t-r \rvert^{\kappa}}] \nonumber\\
=& \sup_{l \in \mathbb{N}} \mathbb{E}^{P_{l}} [ \sup_{r,t \in [0,T]: r \neq t} \frac{ \lVert \int_{r}^{t} F(\xi(\lambda)) d\lambda+ M_{t,s_{l}}^{\xi} - M_{r,s_{l}}^{\xi} \rVert_{H_{x}^{-4}}}{\lvert t-r \rvert^{\kappa}}] \label{est 74}
\end{align}
where we estimate for any $p \in (1, \infty)$ by using  $\lVert \int_{r}^{t} f(l) dl  \rVert_{H_{x}^{-4}}\leq \int_{r}^{t} \lVert f \rVert_{H_{x}^{-4}} dl$ for any $f \in L^{1} ([r,t]; H^{-4} (\mathbb{T}^{2}))$, \eqref{est 57} with $\tilde{v} \equiv 0$, and (M3) of Definition \ref{Definition 4.1},  
\begin{equation}\label{est 72}
\mathbb{E}^{P_{l}} \left[ \sup_{r \in [0,T], t \in (s_{l}, T]: r \neq t} \frac{ \lVert \int_{r}^{t} F(\xi(\lambda)) d\lambda \rVert_{H_{x}^{-4}}^{p}}{\lvert t-r \rvert^{p-1}}\right] \leq T C_{T,p} (1+ \lVert \xi_{l} \rVert_{\dot{H}_{x}^{\frac{1}{2}}}^{2p}). 
\end{equation}
On the other hand, we can compute for any $p \in (1,\infty)$
\begin{equation}\label{est 73}
\mathbb{E}^{P_{l}}[ \lVert M_{t,s_{l}}^{\xi} - M_{r, s_{l}}^{\xi} \rVert_{L_{x}^{2}}^{2p} ]  \lesssim \lVert M_{t,r}^{\xi} \rVert_{L_{x}^{2} L_{\omega}^{2p}}^{2p} \lesssim_{p} \lvert t-r \rvert^{p} C_{t,p} (1+ \lVert \xi_{l} \rVert_{\dot{H}_{x}^{\frac{1}{2}}}^{2p})
\end{equation} 
by Minkowski's inequality, Burkholder-Davis-Gundy inequality, (M2) and (M3) of Definition \ref{Definition 4.1}. Applying Kolmogorov's test (e.g., \cite[Theorem 3.3 on p. 67]{DZ14}) on \eqref{est 73} we obtain for all $\kappa \in (0, \frac{1}{2})$ 
\begin{equation}
 \sup_{l \in \mathbb{N}} \mathbb{E}^{P_{l}} \left[ \sup_{r, t \in [0,T]: r \neq t} \frac{ \lVert \xi(t) - \xi(r) \rVert_{H_{x}^{-4}}}{\lvert t-r \rvert^{\kappa}} \right] 
\overset{\eqref{est 74}\eqref{est 72}}{\leq} C(T, \alpha, \lVert \xi_{l} \rVert_{\dot{H}_{x}^{\frac{1}{2}}}).
\end{equation}  
Together with (M3) of Definition \ref{Definition 4.1}, we now see that \eqref{est 63} is satisfied so that by Proposition \ref{Proposition 6.2}, we can deduce that $\{P_{l}\}_{l\in\mathbb{N}}$ is tight in $\mathbb{S}$. Then we deduce by Prokhorov's theorem (e.g., \cite[Theorem 2.3]{DZ14}) that $P_{l}$ converges weakly to some $P \in \mathcal{P}(\Omega_{0})$ and by Skorokhod's theorem (e.g., \cite[Theorem 2.4]{DZ14}) that there exist $(\tilde{\Omega}, \tilde{\mathcal{F}}, \tilde{P})$ and $\mathbb{S}$-valued random variables $\{\tilde{\xi}_{l}\}_{l\in\mathbb{N}}$ and $\tilde{\xi}$ such that 
\begin{equation}\label{est 76}
\mathcal{L} (\tilde{\xi}_{l}) = P_{l} \hspace{1mm} \forall \hspace{1mm} l \in\mathbb{N}, \hspace{1mm} \tilde{\xi}_{l} \to \tilde{\xi} \text{ in } \mathbb{S} \hspace{1mm} \tilde{P}\text{-a.s.}, \hspace{1mm} \text{ and } \hspace{1mm} \mathcal{L} (\tilde{\xi}) = P. 
\end{equation} 
The rest of this proof consists of showing that $P \in \mathcal{C} (s, \xi^{\text{in}}, \{C_{t,q} \}_{q\in\mathbb{N}, t \geq s})$. First, it follows immediately from \eqref{est 76} that $P(\{ \xi(t) = \xi^{\text{in}} \hspace{1mm} \forall \hspace{1mm} t \in [0,s] \}) = 1$. Next, for all $\psi^{k} \in C^{\infty} (\mathbb{T}^{2}) \cap \dot{H}_{\sigma}^{\frac{1}{2}}$ and $t\geq s$, due to \eqref{est 76} and \eqref{est 77}
\begin{equation}\label{est 80} 
\langle \tilde{\xi}_{l}(t), \psi^{k} \rangle \to \langle \tilde{\xi}(t), \psi^{k} \rangle 
\end{equation} 
as $l\to\infty$ $\tilde{P}$-a.s. Next, to prove that  
\begin{align}
& \mathbb{E}^{\tilde{P}}[ \int_{s_{l}}^{t} \sum_{i,j=1}^{2} \langle \Lambda \tilde{\xi}_{l,i}, \partial_{i} \psi_{j}^{k} \tilde{\xi}_{l,j} \rangle_{\dot{H}_{x}^{-\frac{1}{2}} - \dot{H}_{x}^{\frac{1}{2}}} - \frac{1}{2} \langle \partial_{i} \tilde{\xi}_{l,j}, [\Lambda, \psi_{i}^{k} ] \tilde{\xi}_{l,j} \rangle_{\dot{H}_{x}^{-\frac{1}{2}}-\dot{H}_{x}^{\frac{1}{2}}} - \langle \tilde{\xi}_{l},\Lambda^{\gamma_{1}} \psi^{k} \rangle dr\nonumber \\
& - \int_{s}^{t} \sum_{i,j=1}^{2} \langle \Lambda \tilde{\xi}_{i}, \partial_{i} \psi_{j}^{k} \tilde{\xi}_{j} \rangle_{\dot{H}_{x}^{-\frac{1}{2}} - \dot{H}_{x}^{\frac{1}{2}}} - \frac{1}{2} \langle \partial_{i} \tilde{\xi}_{j}, [\Lambda, \psi_{i}^{k} ]\tilde{\xi}_{j}\rangle_{\dot{H}_{x}^{-\frac{1}{2}}-\dot{H}_{x}^{\frac{1}{2}}} - \langle \tilde{\xi}, \Lambda^{\gamma_{1}} \psi^{k} \rangle dr] \to 0  \label{est 78} 
\end{align}
as $l\to\infty$, we first split the left hand side of \eqref{est 78} as a sum of $\RomanI_{1}$ and $\RomanI_{2}$ defined by 
\begin{align}
\RomanI_{1} \triangleq& \mathbb{E}^{\tilde{P}}[ \int_{s_{l}}^{s} \sum_{i,j=1}^{2} \langle \Lambda \tilde{\xi}_{l,i}, \partial_{i}\psi_{j}^{k} \tilde{\xi}_{l,j} \rangle_{\dot{H}_{x}^{-\frac{1}{2}}-\dot{H}_{x}^{\frac{1}{2}}} - \frac{1}{2} \langle \partial_{i} \tilde{\xi}_{l,j}, [\Lambda, \psi_{i}^{k}]\tilde{\xi}_{l,j} \rangle_{\dot{H}_{x}^{-\frac{1}{2}} - \dot{H}_{x}^{\frac{1}{2}}} - \langle \tilde{\xi}_{l},\Lambda^{\gamma_{1}}\psi^{k} \rangle dr], \nonumber \\
\RomanI_{2} \triangleq& \mathbb{E}^{\tilde{P}} [\int_{s}^{t} \sum_{i,j=1}^{2} \langle \Lambda \tilde{\xi}_{l,i}, \partial_{i} \psi_{j}^{k} \tilde{\xi}_{l,j} \rangle_{\dot{H}_{x}^{-\frac{1}{2}} - \dot{H}_{x}^{\frac{1}{2}}} - \frac{1}{2} \langle \partial_{i} \tilde{\xi}_{l,j}, [\Lambda, \psi_{i}^{k}]\tilde{\xi}_{l,j}\rangle_{\dot{H}_{x}^{-\frac{1}{2}} - \dot{H}_{x}^{\frac{1}{2}}} - \langle \tilde{\xi}_{l},\Lambda^{\gamma_{1}}\psi^{k} \rangle dr \nonumber\\
& - \sum_{i,j=1}^{2} \langle \Lambda \tilde{\xi}_{i}, \partial_{i} \psi_{j}^{k} \tilde{\xi}_{j} \rangle_{\dot{H}_{x}^{-\frac{1}{2}} - \dot{H}_{x}^{\frac{1}{2}}} + \frac{1}{2} \langle \partial_{i} \tilde{\xi}_{j}, [\Lambda, \psi_{i}^{k} ] \tilde{\xi}_{j} \rangle_{\dot{H}_{x}^{-\frac{1}{2}}-\dot{H}_{x}^{\frac{1}{2}}} + \langle \tilde{\xi}, \Lambda^{\gamma_{1}} \psi^{k} \rangle dr]. \label{est 81} 
\end{align}
First, we estimate  
\begin{align}
\RomanI_{1} \overset{\eqref{est 79}\eqref{est 50}  }{\lesssim}& \mathbb{E}^{\tilde{P}} [ \int_{s_{l}}^{s} \lVert \tilde{\xi}_{l} \rVert_{\dot{H}_{x}^{\frac{1}{2}}} ( \lVert \Lambda^{\frac{1}{2}} \nabla \psi^{k} \rVert_{L_{x}^{\infty}} \lVert \tilde{\xi}_{l} \rVert_{L_{x}^{2}} + \lVert \nabla \psi^{k} \rVert_{L_{x}^{\infty}} \lVert \tilde{\xi}_{l} \rVert_{\dot{H}_{x}^{\frac{1}{2}}}) \label{est 81a} \\
&  + \lVert \tilde{\xi}_{l} \rVert_{\dot{H}_{x}^{\frac{1}{2}}} \lVert \psi \rVert_{W_{x}^{2,\infty}} \lVert\tilde{\xi}_{l} \rVert_{\dot{H}_{x}^{\frac{1}{2}}} + \lVert\tilde{\xi}_{l} \rVert_{\dot{H}_{x}^{\frac{1}{2}}} dr]  \lesssim \mathbb{E}^{\tilde{P}}[ \int_{s_{l}}^{s}\lVert \tilde{\xi}_{l} \rVert_{\dot{H}_{x}^{\frac{1}{2}}}^{2} + \lVert \tilde{\xi}_{l} \rVert_{\dot{H}_{x}^{\frac{1}{2}}} dr] \overset{\eqref{est 76}}{\to} 0 \nonumber
\end{align}
as $l\to\infty$ due to the hypothesis that $s_{l} \to s$ as $l\to\infty$. For $\RomanI_{2}$ of \eqref{est 81}, we can estimate 
\begin{align}
\RomanI_{2} & \overset{\eqref{est 79}\eqref{est 50}}{\lesssim}  \mathbb{E}^{\tilde{P}}[ \int_{s}^{t} \sum_{i,j=1}^{2} \lVert \tilde{\xi}_{l,i} - \tilde{\xi}_{i} \rVert_{\dot{H}_{x}^{\frac{1}{2}}} ( \lVert \Lambda^{\frac{1}{2}} \partial_{i} \psi_{j}^{k} \rVert_{L_{x}^{\infty}}\lVert \tilde{\xi}_{l,j} \rVert_{L_{x}^{2}}+ \lVert \partial_{i} \psi_{j}^{k} \rVert_{L_{x}^{\infty}} \lVert \Lambda^{\frac{1}{2}} \tilde{\xi}_{l,j} \rVert_{L_{x}^{2}}) \label{est 81b} \\
& \hspace{3mm} + \lVert \tilde{\xi}_{i} \rVert_{\dot{H}_{x}^{\frac{1}{2}}} (\lVert \Lambda^{\frac{1}{2}} \partial_{i} \psi_{j}^{k} \rVert_{L_{x}^{\infty}} \lVert \tilde{\xi}_{l,j} - \tilde{\xi}_{j} \rVert_{L_{x}^{2}} + \lVert \partial_{i} \psi_{j}^{k} \rVert_{L_{x}^{\infty}} \lVert \tilde{\xi}_{l,j} - \tilde{\xi}_{j} \rVert_{\dot{H}_{x}^{\frac{1}{2}}}) \nonumber\\
&\hspace{3mm} + \lVert \tilde{\xi}_{l,j} - \tilde{\xi}_{j} \rVert_{\dot{H}_{x}^{\frac{1}{2}}} \lVert \psi_{i}^{k} \rVert_{W_{x}^{2,\infty}} \lVert \tilde{\xi}_{l,j} \rVert_{\dot{H}_{x}^{\frac{1}{2}}} + \lVert \tilde{\xi}_{j} \rVert_{\dot{H}_{x}^{\frac{1}{2}}} \lVert \psi_{i}^{k} \rVert_{W_{x}^{2,\infty}} \lVert \tilde{\xi}_{l,j} - \tilde{\xi}_{j} \rVert_{\dot{H}_{x}^{\frac{1}{2}}} + \lVert \tilde{\xi}_{l} - \tilde{\xi} \rVert_{L_{x}^{2}} dr] \overset{\eqref{est 76}}{\to} 0 \nonumber
\end{align} 
as $l\to\infty$. Applying \eqref{est 81a}-\eqref{est 81b} to \eqref{est 81} leads to \eqref{est 78}. Next, for every $t > r \geq s, p \in (1,\infty)$, we estimate from \eqref{est 82} 
\begin{equation}\label{est 84}
 \sup_{l \in \mathbb{N}} \mathbb{E}^{\tilde{P}} [ \lvert M_{t, s_{l}}^{\tilde{\xi}_{l}, k} \rvert^{2p} ] \overset{\eqref{est 79} \eqref{est 50} \eqref{est 76}}{\lesssim_{p}}  \sup_{l\in\mathbb{N}} \mathbb{E}^{P_{l}} [ \lVert \xi(t) \rVert_{L_{x}^{2}}^{2p}  + 1 + t^{2p}(1+ \lVert \xi \rVert_{C([0,t]; \dot{H}_{x}^{\frac{1}{2}})}^{4p}) ] \lesssim_{p} 1.
 \end{equation}  
Next, we can compute starting from \eqref{est 82} 
\begin{align}
\lim_{l\to\infty} \mathbb{E}^{\tilde{P}} [ \lvert M_{t,s_{l}}^{\tilde{\xi}_{l}, k} - M_{t,s}^{\tilde{\xi}, k} \rvert] = 0 \label{est 83} 
\end{align}
due to \eqref{est 80}, \eqref{est 78}, and the hypothesis that $\lim_{l\to\infty} \lVert \xi_{l} - \xi^{\text{in}} \rVert_{\dot{H}_{x}^{\frac{1}{2}}} = 0$. Consequently, for any $g$ that is $\mathbb{R}$-valued, $\mathcal{B}_{r}$-measurable, and continuous on $\mathbb{S}$, 
\begin{equation}
\mathbb{E}^{P} [ (M_{t,s}^{\xi, k} - M_{r,s}^{\xi, k}) g(\xi)]\overset{\eqref{est 76} \eqref{est 83}}{=} \lim_{l\to\infty} \mathbb{E}^{\tilde{P}}[ ( M_{t,s_{l}}^{\tilde{\xi}_{l}, k} - M_{r, s_{l}}^{\tilde{\xi}, k}) g(\tilde{\xi}_{l}) ] \overset{\eqref{est 76}}{=} 0. 
\end{equation} 
This implies that the mapping $t \mapsto M_{t,s}^{k}$ is a $(\mathcal{B}_{t})_{t\geq s}$-martingale under $P$. Next, we can estimate by H$\ddot{\mathrm{o}}$lder's inequality and \eqref{est 84}-\eqref{est 83} 
\begin{equation}\label{est 85} 
 \lim_{l\to\infty} \mathbb{E}^{\tilde{P}} [ \lvert M_{t, s_{l}}^{\tilde{\xi}_{l}, k} - M_{t,s}^{\tilde{\xi}, k} \rvert^{2} ] \leq \lim_{l\to\infty} \left( \mathbb{E}^{\tilde{P}}[ \lvert M_{t,s_{l}}^{\tilde{\xi}_{l}, k} - M_{t,s}^{\tilde{\xi}, k} \rvert ] \right)^{\frac{1}{2}} \left( \mathbb{E} [ \lvert M_{t,s_{l}}^{\tilde{\xi}_{l},k} - M_{t,s}^{\tilde{\xi}, k} \rvert^{3} ] \right)^{\frac{1}{2}} =0. 
\end{equation} 
It follows that 
\begin{equation}\label{est 408}
\mathbb{E}^{P} \left[ \left(( M_{t,s}^{\xi,k})^{2} - (M_{r,s}^{\xi,k})^{2} - \int_{r}^{t} \lVert G(\xi(\lambda))^{\ast} \psi^{k} \rVert_{U}^{2} d \lambda \right) g(\xi) \right]  \overset{\eqref{est 76} \eqref{est 85} \eqref{est 60}} = 0 
\end{equation} 
which implies that $(M_{t,s}^{\xi,k})^{2} - \int_{s}^{t} \lVert G(\xi(\lambda))^{\ast} \psi^{k} \rVert_{U}^{2} d\lambda$ is a $(\mathcal{B}_{t})_{t\geq s}$-martingale under $P$ so that 
\begin{equation}\label{est 86}
\langle \langle M_{t,s}^{\xi,k} \rangle \rangle = \int_{s}^{t} \lVert G(\xi(\lambda))^{\ast} \psi^{k} \rVert_{U}^{2} d\lambda. 
\end{equation} 
This leads to by Burkholder-Davis-Gundy inequality (e.g., \cite[p. 166]{KS91}) 
\begin{equation}
\mathbb{E}^{P} [ \lvert M_{t,s}^{\xi,k} \rvert^{2} ] \overset{\eqref{est 86}}{\lesssim}  \mathbb{E}^{P} [ \int_{s}^{t} \lVert G(\xi(\lambda))^{\ast} \psi^{k} \rVert_{U}^{2} d\lambda ] \overset{\eqref{est 60} }{\lesssim} \lvert t-s \rvert \mathbb{E}^{P} [ 1 + \lVert \xi \rVert_{C([s,t]; \dot{H}_{x}^{\frac{1}{2}})}^{2} ] \lesssim 1
\end{equation} 
and hence $M_{t,s}^{k}$ is square-integrable. Finally, the proof of (M3) follows from defining $R(t,s,\xi) \triangleq \sup_{r \in [0,t]} \lVert \xi(r) \rVert_{\dot{H}_{x}^{\frac{1}{2}}}^{2q} + \int_{s}^{t} \lVert \xi(r) \rVert_{\dot{H}_{x}^{\frac{1}{2} + \iota}}^{2} dr$ for any fixed $q \in \mathbb{N}$ and $t\geq s$, and relying on the fact that the mapping $\xi \mapsto R(t,s,\xi)$ is lower semicontinuous on $\mathbb{S}$.

\subsection{Proof of Proposition \ref{Proposition 4.4}}
Following \cite[p. 84]{DZ14} we define 
\begin{equation}\label{est 116}
Y(s) \triangleq \frac{\sin (\pi \alpha)}{\pi} \int_{0}^{s} e^{- (-\Delta)^{\frac{\gamma_{1}}{2}} (s-r)}(s-r)^{-\alpha} \mathbb{P} d B(r), \hspace{3mm} \alpha \in \left(0, \frac{1}{2}\right), 
\end{equation} 
so that due to \eqref{est 114}  
\begin{equation}\label{est 117} 
\int_{0}^{t} (t-s)^{\alpha -1} e^{- (-\Delta)^{\frac{\gamma_{1}}{2} } (t-s)} Y(s) ds  = z(t)
\end{equation} 
(see \cite[p. 131]{DZ14}). Then, for any $l \in \mathbb{N}$ and $\eta \geq 0$ such that 
\begin{equation}\label{est 120}
-\frac{4\eta}{\gamma_{1}} + \frac{8}{\gamma_{1}} + \frac{4\sigma}{\gamma_{1}} > 2 \alpha, 
\end{equation}
we can estimate by the Gaussian hypercontractivity theorem (e.g., \cite[Theorem 3.50]{J97}), and It$\hat{\mathrm{o}}$'s isometry (e.g., \cite[equation (4) on p. 28]{D13}), 
\begin{equation}\label{est 118}
\mathbb{E}^{\textbf{P}}[ \lVert (-\Delta)^{\eta} Y(s) \rVert_{L_{x}^{2}}^{2l} ] \lesssim_{l} ( \mathbb{E}^{\textbf{P}} [ \lVert (-\Delta)^{\eta} Y(s) \rVert_{L_{x}^{2}}^{2} ] )^{l} 
\lesssim_{l} 1.
\end{equation} 
We integrate \eqref{est 118} over $[0,T]$ and use Fubini's theorem to deduce for all $l \in \mathbb{N}$ and $\eta \geq 0$ that satisfies \eqref{est 120}, 
\begin{equation}\label{est 119} 
\mathbb{E}^{\textbf{P}} \left[ \int_{0}^{T} \lVert (-\Delta)^{\eta} Y(s) \rVert_{L_{x}^{2}}^{2l} ds \right] \lesssim_{l} 1. 
\end{equation} 
This allows us to take $l > \frac{1}{2\alpha}$ and deduce by \eqref{est 117} 
\begin{equation}\label{est 427} 
\mathbb{E}^{\textbf{P}}[  \lVert (-\Delta)^{\eta} z \rVert_{C_{T}L_{x}^{2}}^{2l}] = \mathbb{E}^{\textbf{P}} \left[ \sup_{t \in [0,T]} \lVert (-\Delta)^{\eta} [ \int_{0}^{t} (t-s)^{\alpha -1} e^{- (-\Delta)^{\frac{\gamma_{1}}{2}}(t-s)} Y(s) ds] \rVert_{L_{x}^{2}}^{2l} \right] \overset{\eqref{est 119}}{\lesssim}_{l} 1.
\end{equation} 
Then, by taking $\alpha \in \left(0, \frac{3\sigma}{2\gamma_{1}} \wedge \frac{1}{2} \right)$, we can choose $\eta = 2 + \frac{\sigma}{4}$ in \eqref{est 427} and still satisfy \eqref{est 120} which implies $\mathbb{E}^{\textbf{P}}[ \lVert z \rVert_{C_{T} \dot{H}_{x}^{4+ \frac{\sigma}{2}}}^{2l}] < \infty$, which is the first claim in \eqref{est 115}. Next, in order to prove the second claim in \eqref{est 115}, we take $\eta = \frac{8+  \sigma - \gamma_{1}}{4}$ in \eqref{est 119} so that \eqref{est 120} is satisfied for any $\alpha \in (0, \frac{1}{2})$, granting us 
\begin{equation}\label{est 121} 
\mathbb{E}^{\textbf{P}} \left[ \int_{0}^{T} \lVert (-\Delta)^{\frac{8 + \sigma - \gamma_{1}}{4}} Y(s) \rVert_{L_{x}^{2}}^{2l} ds\right] < \infty.
\end{equation} 
We take $l > \frac{1}{2\alpha}$, rely on \cite[Proposition A.1.1 (ii)]{DZ92} similarly to the proof of \cite[Proposition 34]{D13} to deduce that the mapping $Y \mapsto \int_{0}^{t} (t-s)^{\alpha -1} e^{- (-\Delta)^{\frac{\gamma_{1}}{2}}(t-s)} Y(s) ds \overset{\eqref{est 117}}{=} z(t)$ is a bounded linear operator from $L^{2l} (0,T; H_{x}^{\frac{8 + \sigma - \gamma_{1}}{2}})$ to $C^{\delta} ([0,T]; H_{x}^{\frac{8 + \sigma - \gamma_{1}}{2}})$ for any $\delta \in (0, \alpha - \frac{1}{2l})$; this, together with \eqref{est 121}, leads to the second claim in \eqref{est 115} identically to previous works. 

\subsection{Proof of Proposition \ref{Proposition 4.5}}
The stopping time $\mathfrak{t}$ in the statement of Theorem \ref{Theorem 2.1} is $T_{L}$ from \eqref{est 128} for $L > 0$ sufficiently large and thus by Theorem \ref{Theorem 2.1} we know that there exists a process $v$ that is a weak solution to \eqref{est 30} on $[0, T_{L}]$ such that \eqref{est 129} holds. Hence, we see that $v(\cdot\wedge T_{L}) \in \Omega_{0}$. Now due to \eqref{est 130}, \eqref{est 122}, \eqref{est 30}, and \eqref{est 114}
\begin{equation}\label{est 131}
Z^{v}(t) = z(t) \hspace{3mm} \forall \hspace{1mm} t \in [0, T_{L}] \hspace{1mm} \textbf{P}\text{-a.s.}
\end{equation} 
By Proposition \ref{Proposition 4.4} we know that $z \in C_{T} \dot{H}_{x}^{4+ \frac{\sigma}{2}} \cap C_{\text{loc}}^{\frac{1}{2} - \delta} \dot{H}_{x}^{4- \frac{\gamma_{1}}{2} + \frac{\sigma}{2}}$ $\textbf{P}$-a.s. and thus the trajectories $t \mapsto \lVert z(t) \rVert_{\dot{H}_{x}^{4+ \frac{\sigma}{2}}}$ and $t \mapsto \lVert z \rVert_{C_{t}^{\frac{1}{2} - 2 \delta} \dot{H}_{x}^{4-\frac{\gamma_{1}}{2} + \frac{\sigma}{2}}}$ where $\delta \in (0,\frac{1}{4})$ are $\textbf{P}$-a.s. continuous. It follows by \eqref{est 126b}, \eqref{est 128}, and \eqref{est 131} that 
\begin{equation}\label{est 132} 
\tau_{L}(v) = T_{L} \hspace{3mm} \textbf{P}\text{-a.s.}
\end{equation}  
Next, we verify that $P$ is a martingale solution to \eqref{est 30} on $[0, \tau_{L}]$ according to Definition \ref{Definition 4.2}. The verification of (M1) follows from \eqref{est 60} and \eqref{est 129} while that of (M3) follows by writing $v= y + z$, choosing $C_{t,q}, q \in \mathbb{N}, $ in the Definition \ref{Definition 4.2} to satisfy $(2\pi C_{L} + 2 \pi  L^{\frac{1}{4}})^{2q} + (t \wedge \tau_{L}) (2\pi C_{L} + 2 \pi L^{\frac{1}{4}})^{2} \leq C_{t,q}$, and relying on \eqref{est 137} and \eqref{est 136}. Finally, in order to verify (M2), we let $t \geq s$ and $g$ be bounded, $\mathbb{R}$-valued, $\mathcal{B}_{s}$-measurable, and continuous on $\Omega_{0}$. By Theorem \ref{Theorem 2.1} we know that $v(t\wedge \tau_{L})$ is $(\mathcal{F}_{t})_{t\geq 0}$-adapted so that $g(v(\cdot\wedge\tau_{L}(v)))$ is $\mathcal{F}_{s}$-measurable. This leads to $\mathbb{E}^{P} [ M_{t \wedge \tau_{L}, 0}^{i} g] =\mathbb{E}^{P} [ M_{s \wedge \tau_{L}, 0}^{i} g]$ where $M_{t \wedge \tau_{L},0}^{i} = \langle M_{t \wedge \tau_{L}, 0}, \psi^{i} \rangle$  and hence $M_{t \wedge \tau_{L}, 0}^{i}$ is a $(\mathcal{B}_{t})_{t\geq 0}$-martingale under $\textbf{P}$. Similarly, using the fact that $\langle \langle M_{t\wedge \tau_{L}(v), 0}^{v,i} \rangle \rangle = (t\wedge \tau_{L}(v)) \lVert G \psi^{i} \rVert_{L_{x}^{2}}^{2}$ for all $\psi^{i} \in C^{\infty} (\mathbb{T}^{2}) \cap \dot{H}_{\sigma}^{\frac{1}{2}}$, we can show $\mathbb{E}^{P} [ ( ( M_{t \wedge \tau_{L}, 0}^{i})^{2} - (t \wedge \tau_{L}) \lVert G \psi^{i} \rVert_{L_{x}^{2}}^{2} ) g]  = \mathbb{E}^{P} \left[ \left( (M_{s \wedge \tau_{L}, 0}^{i})^{2} - (s \wedge \tau_{L}) \lVert G \psi^{i} \rVert_{L_{x}^{2}}^{2} \right) g \right]$ which implies that $(M_{t\wedge \tau_{L}, 0}^{i})^{2} - (t \wedge \tau_{L}) \lVert G \psi^{i} \rVert_{L_{x}^{2}}^{2}$ is a $(\mathcal{B}_{t})_{t\geq 0}$-martingale under $P$. It also has a consequence that $\langle \langle M_{t \wedge \tau_{L}, 0}^{i} \rangle \rangle = \int_{0}^{t \wedge \tau_{L}} \lVert G \psi^{i} \rVert_{L_{x}^{2}}^{2} dr$, and this also shows that $M_{t \wedge \tau_{L}, 0}^{i}$ is square-integrable, completing the proof that  $P$ satisfies (M2). Hence, $P$ is a martingale solution to \eqref{est 30} on $[0, \tau_{L}]$. 

\subsection{Proof of Proposition \ref{Proposition 4.6}}
Because $\tau_{L}$ is a $(\mathcal{B}_{t})_{t\geq 0}$-stopping time that is bounded by $L$ due to \eqref{est 126b}  while $P$ is a martingale solution on $[0, \tau_{L}]$ due to Proposition \ref{Proposition 4.5}, we see that Lemma \ref{Lemma 4.3} completes this proof once we verify \eqref{est 94}. First, by relying on \eqref{est 131}, \eqref{est 132}, and \eqref{est 115}, we can show that there exists a $P$-measurable set $\mathcal{N} \subset \Omega_{0,\tau_{L}}$ such that $P(\mathcal{N}) = 0$ and for all $\omega \in \Omega_{0} \setminus \mathcal{N}$ and $\delta \in (0, \frac{1}{2})$, $Z^{\omega} (\cdot \wedge \tau_{L}(\omega)) \in C_{T} \dot{H}_{x}^{4+ \frac{\sigma}{2}} \cap C_{\text{loc}}^{\frac{1}{2} - \delta} \dot{H}_{x}^{4- \frac{\gamma_{1}}{2} + \frac{\sigma}{2}}$. For every $\omega' \in \Omega_{0}$ and $\omega \in \Omega_{0} \setminus \mathcal{N}$ we define 
\begin{equation}\label{est 138}
\mathbb{Z}_{\tau_{L} (\omega)}^{\omega'} (t) \triangleq M_{t,0}^{\omega'} - e^{- (t- t \wedge \tau_{L} (\omega)) \Lambda^{\gamma_{1}}} M_{t \wedge \tau_{L} (\omega), 0}^{\omega'} - \int_{t \wedge \tau_{L}(\omega)}^{t} \mathbb{P} \Lambda^{\gamma_{1}} e^{-(t-s) \Lambda^{\gamma_{1}}} M_{s,0}^{\omega'} ds 
\end{equation} 
so that because $\nabla \cdot M_{t,0}^{\omega} = 0$ for any $\omega \in \Omega_{0}$ due to \eqref{est 122}, we see that 
\begin{equation}\label{est 139} 
 \mathbb{Z}_{\tau_{L} (\omega)}^{\omega'}(t) = M_{t,0}^{\omega'} - M_{t \wedge \tau_{L}(\omega), 0}^{\omega'} - \int_{t \wedge \tau_{L}(\omega)}^{t} \mathbb{P} \Lambda^{\gamma_{1}} e^{-(t-s) \Lambda^{\gamma_{1}}} \left( M_{s,0}^{\omega'} - M_{s \wedge \tau_{L}(\omega), 0}^{\omega'} \right) ds.  
\end{equation} 
Together with \eqref{est 123}, this leads us to 
\begin{equation}\label{est 140}
Z^{\omega'}(t) - Z^{\omega'}(t \wedge \tau_{L}(\omega)) = \mathbb{Z}_{\tau_{L}(\omega)}^{\omega'} (t) + \left( e^{- (t - t \wedge \tau_{L} (\omega)) \Lambda^{\gamma_{1}}} - \Id \right) Z^{\omega'} (t \wedge \tau_{L} (\omega)). 
\end{equation}  
It follows from \eqref{est 139} that $\mathbb{Z}_{\tau_{L} (\omega)}^{\omega'}$ is $\mathcal{B}^{\tau_{L} (\omega)}$-measurable because $M_{t, 0}^{\omega'} - M_{t\wedge \tau_{L}(\omega), 0}^{\omega'}$ is $\mathcal{B}^{\tau_{L} (\omega)}$-measurable, and that $Q_{\omega}$ from Lemma \ref{Lemma 4.2} satisfies 
\begin{align}
&  Q_{\omega} ( \{ \omega' \in \Omega_{0}: Z^{\omega'} (\cdot) \in C_{T} \dot{H}_{x}^{4+ \frac{\sigma}{2}} \cap C_{\text{loc}}^{\frac{1}{2} - \delta} \dot{H}_{x}^{4 - \frac{\gamma_{1}}{2} + \frac{\sigma}{2}} \}) \nonumber \\
=& \delta_{\omega} ( \{ \omega' \in \Omega_{0}: Z^{\omega'} (\cdot \wedge \tau_{L} (\omega)) \in C_{T} \dot{H}_{x}^{4+ \frac{\sigma}{2}} \cap C_{\text{loc}}^{\frac{1}{2} - \delta} \dot{H}_{x}^{4 - \frac{\gamma_{1}}{2} + \frac{\sigma}{2}} \}) \nonumber\\
& \otimes_{\tau_{L} (\omega)} R_{\tau_{L} (\omega), \xi(\tau_{L} (\omega), \omega)} ( \{ \omega' \in \Omega_{0}: \mathbb{Z}_{\tau_{L} (\omega)}^{\omega'} (\cdot) \in C_{T} \dot{H}_{x}^{4+ \frac{\sigma}{2}} \cap C_{\text{loc}}^{\frac{1}{2} - \delta} \dot{H}_{x}^{4 - \frac{\gamma_{1}}{2} + \frac{\sigma}{2}} \}) \label{est 142}
\end{align}
where for all $\omega \in \Omega \setminus \mathcal{N}$, $\delta_{\omega} ( \{ \omega' \in \Omega_{0}: Z^{\omega'} ( \cdot \wedge \tau_{L} (\omega)) \in C_{T} \dot{H}_{x}^{4+ \frac{\sigma}{2}} \cap C_{\text{loc}}^{\frac{1}{2} - \delta} \dot{H}_{x}^{4 - \frac{\gamma_{1}}{2} + \frac{\sigma}{2}} \})  = 1$. Next, we rewrite 
\begin{equation}
 \int_{0}^{t} \mathbb{P} e^{- (t-s) \Lambda^{\gamma_{1}}} d( M_{s,0}^{\omega'} - M_{s \wedge \tau_{L} (\omega), 0}^{\omega'}) \overset{\eqref{est 139} }{=} \mathbb{Z}_{\tau_{L} (\omega)}^{\omega'} (t)
\end{equation} 
and observe that Proposition \ref{Proposition 4.4} implies that for any $\delta \in (0, \frac{1}{2})$,  
\begin{equation}\label{est 141}
R_{\tau_{L} (\omega), \xi(\tau_{L} (\omega), \omega)} ( \{ \omega' \in \Omega_{0}: \mathbb{Z}_{\tau_{L} (\omega)}^{\omega'} (\cdot) \in C_{T} \dot{H}_{x}^{4+ \frac{\sigma}{2}} \cap C_{\text{loc}}^{\frac{1}{2} - \delta} \dot{H}_{x}^{4- \frac{\gamma_{1}}{2} + \frac{\sigma}{2}} \}) = 1. 
\end{equation} 
Hence, applying this to \eqref{est 142} now allows us to conclude that for all $\omega \in (\Omega_{0} \setminus \mathcal{N}) \cap \{ \xi(\tau) \in \dot{H}_{\sigma}^{\frac{1}{2}}\}$, $Q_{\omega} ( \{ \omega' \in \Omega_{0}: Z^{\omega'} (\cdot) \in C_{T} \dot{H}_{x}^{4+ \frac{\sigma}{2}} \cap C_{\text{loc}}^{\frac{1}{2} - \delta} \dot{H}_{x}^{4 - \frac{\gamma_{1}}{2} + \frac{\sigma}{2}} \}) = 1$; i.e., for all $\omega \in (\Omega_{0} \setminus \mathcal{N}) \cap \{\xi(\tau) \in \dot{H}_{\sigma}^{\frac{1}{2}}\}$ there exists a measurable set $N_{\omega}$ such that $Q_{\omega} (N_{\omega}) = 0$ and for all $\omega' \in \Omega_{0} \setminus N_{\omega}$, the mapping $t \mapsto Z^{\omega'} (t)$ lies in $C_{T} \dot{H}_{x}^{4+ \frac{\sigma}{2}} \cap C_{\text{loc}}^{\frac{1}{2} - \delta} \dot{H}_{x}^{4 - \frac{\gamma_{1}}{2} + \frac{\sigma}{2}}$. This implies by \eqref{est 126b} that for all $\omega \in \Omega_{0} \setminus \mathcal{N}$, 
\begin{equation}\label{est 145} 
\tau_{L} (\omega') = \bar{\tau}_{L} (\omega') \hspace{3mm} \forall \hspace{1mm} \omega' \in \Omega_{0} \setminus N_{\omega} 
\end{equation}   
if we define for $\delta \in (0, \frac{1}{4})$  
\begin{align}
\bar{\tau}_{L} (\omega') \triangleq& \inf \{t \geq 0: C_{S_{1}} \lVert Z^{\omega'}(t) \rVert_{\dot{H}_{x}^{4+ \frac{\sigma}{2}}} \geq L^{\frac{1}{4}} \} \nonumber\\
\wedge& \inf\{t \geq 0: C_{S_{1}} \lVert Z^{\omega'} \rVert_{C_{t}^{\frac{1}{2} - 2 \delta} \dot{H}_{x}^{4 - \frac{\gamma_{1}}{2} + \frac{\sigma}{2}}} \geq L^{\frac{1}{2}} \} \wedge L. \label{est 144} 
\end{align}
This leads us to, for all $\omega \in (\Omega_{0} \setminus \mathcal{N}) \cap \{\xi(\tau) \in \dot{H}_{\sigma}^{\frac{1}{2}}\}$ with $P(\{ \xi(\tau) \in \dot{H}_{\sigma}^{\frac{1}{2}} \}) = 1$, 
\begin{equation}\label{est 143} 
Q_{\omega} (\{ \omega ' \in \Omega_{0}: \tau_{L} (\omega') = \tau_{L} (\omega) \}) = 1 
\end{equation} 
by the identical arguments to \cite{HZZ19}. 

\subsection{Further preliminaries}
We list various previous results which played crucial roles in the proofs of Theorems \ref{Theorem 2.1}-\ref{Theorem 2.2}. The following standard product estimate is convenient and can be readily proved via Fourier analysis:
\begin{equation}\label{est 79}
\lVert \Lambda^{s} (fg) \rVert_{L_{x}^{2}}\lesssim \lVert \Lambda^{s}f \rVert_{L_{x}^{\infty}} \lVert g \rVert_{L_{x}^{2}} + \lVert f \rVert_{L_{x}^{\infty}} \lVert \Lambda^{s} g \rVert_{L_{x}^{2}} \hspace{3mm} \text{ for any } s \geq 0. 
\end{equation} 

\begin{proposition}\label{Proposition 6.3}
\rm{(Calder$\acute{\mathrm{o}}$n commutator; e.g., \cite[Lemma A.5]{BSV19}, \cite[Lemma 2.2]{M08}, and \cite[Theorem 10.3]{L02})} Let $p \in (1,\infty)$. 
\begin{enumerate} 
\item If $\phi \in W^{1,\infty}(\mathbb{T}^{2})$, then for any $f \in L^{p}(\mathbb{T}^{2})$ that is mean-zero, 
\begin{equation}\label{est 49}
\lVert [\Lambda, \phi] f \rVert_{L_{x}^{p}} \lesssim_{p} \lVert \phi \rVert_{W_{x}^{1,\infty}} \lVert f \rVert_{L_{x}^{p}};
\end{equation} 
moreover, if $s \in [0,1], \phi \in W^{2,\infty}(\mathbb{T}^{2})$, then for any $f \in H^{s}(\mathbb{T}^{2})$ that is mean-zero, 
\begin{equation}\label{est 50} 
\lVert [\Lambda, \phi] f \rVert_{H_{x}^{s}} \lesssim_{s} \lVert \phi \rVert_{W_{x}^{2,\infty}} \lVert f \rVert_{H_{x}^{s}}. 
\end{equation} 
\item Let $\gamma_{2} \in (1, 2)$. If $\phi \in C^{2-\gamma_{2}} (\mathbb{T}^{2})$ such that $\Lambda^{2-\gamma_{2}} \phi \in L^{\infty} (\mathbb{T}^{2})$, then for all $f \in L^{p}(\mathbb{T}^{2})$ that is mean-zero, 
\begin{equation}\label{est 51} 
\lVert [\Lambda^{2-\gamma_{2}}, \phi] f \rVert_{L_{x}^{p}} \lesssim_{p} ( \lVert \phi \rVert_{C_{x}^{2-\gamma_{2}}} + \lVert \Lambda^{2- \gamma_{2}} \phi \rVert_{L_{x}^{\infty}}) \lVert f \rVert_{L_{x}^{p}};
\end{equation} 
moreover, if $s \in [0,1], \phi \in C^{1, 2-\gamma_{2}} (\mathbb{T}^{2})$ such that $\Lambda^{2-\gamma_{2}} \phi \in W^{1, \infty} (\mathbb{T}^{2})$, then for any $f \in H^{s}(\mathbb{T}^{2})$ that is mean-zero, 
\begin{equation}\label{est 52}
\lVert [\Lambda^{2-\gamma_{2}}, \phi] f \rVert_{H_{x}^{s}} \lesssim_{s} (\lVert \phi \rVert_{C_{x}^{1, 2-\gamma_{2}}} + \lVert \Lambda^{2-\gamma_{2}} \phi \rVert_{W_{x}^{1, \infty}}) \lVert f \rVert_{H_{x}^{s}}.
\end{equation} 
\end{enumerate} 
\end{proposition} 

\begin{proof}[Proof of Proposition \ref{Proposition 6.3}]
Proposition \ref{Proposition 6.3} (1) is a statement from \cite[Lemma A.5]{BSV19}; thus, we only sketch the proof for Proposition \ref{Proposition 6.3} (2) following the argument of \cite[Lemma A.5]{BSV19} and \cite[Lemma 2.2]{M08}. First, we work on the case $x \in \mathbb{R}^{2}$ because it leads to the case $x \in \mathbb{T}^{2}$ by the same argument on \cite[p. 1866]{BSV19}. We fix $\phi \in C^{2-\gamma_{2}}(\mathbb{T}^{2})$ such that $\Lambda^{2- \gamma_{2}} \phi \in L^{\infty} (\mathbb{T}^{2})$ and define $T \triangleq [\Lambda^{2-\gamma_{2}}, \phi]$. We will verify the hypothesis of T(1) theorem from \cite[Theorem 1 on p. 373]{DJ84} (see also \cite[Theorem 10.2 on p. 95]{L02}), which consists of 
\begin{enumerate}
\item $T$ is a continuous operator from $\mathcal{S}(\mathbb{R}^{2})$ to its dual $\mathcal{S}'(\mathbb{R}^{2})$ where $\mathcal{S}(\mathbb{R}^{2})$ is the Schwartz space, associated with some singular integral operator (SIO) $K$; i.e., $K$ is a continuous function defined on $\mathbb{R}^{2} \times \mathbb{R}^{2} \setminus \triangle$ where $\triangle \triangleq \{(x,y) \in \mathbb{R}^{2}\times \mathbb{R}^{2}: x= y \}$ for which there exist two constants $\delta \in (0,1]$ and $C_{K} > 0$ such that 
\begin{subequations}
\begin{align}
&\langle Tf, g \rangle = \int_{\mathbb{R}^{2}} \int_{\mathbb{R}^{2}} K(x,y) f(y) g(x) dy dx \hspace{1mm} \forall \hspace{1mm} f, g \in C_{c}^{\infty} (\mathbb{R}^{2}) \text{ with disjoint supports}, \label{est 426} \\
& \lvert K(x,y) \rvert \leq C_{K} \lvert x-y \rvert^{-2} \hspace{1mm}\forall \hspace{1mm} (x,y) \in \mathbb{R}^{2} \times \mathbb{R}^{2} \setminus \triangle, \label{est 425} \\ 
&\lvert K(x',y) - K(x,y) \rvert + \lvert K(y, x') - K(y,x) \rvert \leq C_{K} \frac{ \lvert x' - x \rvert^{\delta}}{\lvert x-y \rvert^{2+ \delta}}  \label{est 45} \\
&\hspace{32mm}  \forall \hspace{1mm} x, x', y \text{ such that }  \lvert x' - x \rvert < \frac{1}{2} \lvert x-y \rvert, \label{est 44} 
\end{align}
\end{subequations} 
\item  $T1, T^{\ast} 1 \in BMO(\mathbb{R}^{2})$ where $T^{\ast}$ is defined by $\langle T^{\ast} f, g \rangle = \langle Tg, f\rangle$ (see \cite[p. 372]{DJ84}), 
\item $T$ has the weak boundedness property (WBP) (see \cite[p. 373]{DJ84}); i.e., for any bounded subset $\mathcal{B} \subset C_{c}^{\infty} (\mathbb{R}^{2})$, there exists a constant $C$ such that for all $\psi, \tilde{\psi} \in \mathcal{B}$, all $x_{0} \in \mathbb{R}^{2}$, and $R> 0$, $\lvert \int_{\mathbb{R}^{2}} T\psi \left( \frac{y-x_{0}}{R} \right) \tilde{\psi} \left(\frac{y-x_{0}}{R} \right) dy \rvert \leq C R^{2}$.
\end{enumerate} 
First, due to \cite[Proposition 2.1 on p. 513]{CC04} we may write for all $f \in C_{c}^{\infty} (\mathbb{R}^{2})$  
\begin{equation}\label{est 403}
Tf(x)= [ \Lambda^{2-\gamma_{2}}, \phi] f(x) = C_{\gamma_{2}} \pv \int_{\mathbb{R}^{2}} \frac{ [\phi(x) - \phi(y)]}{\lvert x-y \rvert^{4-\gamma_{2}}} f(y) dy. 
\end{equation} 
Therefore, off the diagonal $\triangle$, the kernel $K(x,y)$ associated to $T = [\Lambda^{2-\gamma_{2}}, \phi]$ is 
\begin{equation}\label{est 404} 
K(x,y) \triangleq C_{\gamma_{2}} \frac{ \phi(x) - \phi(y)}{\lvert x-y \rvert^{4-\gamma_{2}}}, 
\end{equation} 
from which \eqref{est 426} can be immediately verified by \eqref{est 403}-\eqref{est 404}; moreover, we can immediately estimate 
\begin{equation*}
\lvert K(x,y) \rvert \leq \lvert C_{\gamma_{2}} \rvert \lVert \phi \rVert_{C_{x}^{2- \gamma_{2}}} \lvert x-y \rvert^{-2} \text{ for all } (x,y) \in \mathbb{R}^{2} \times \mathbb{R}^{2} \setminus \triangle 
\end{equation*}
and conclude \eqref{est 425}. Next, to prove \eqref{est 45}, we first notice that $\lvert K(x',y) - K(x,y) \rvert =\lvert K(y, x') - K(y,x) \rvert$ by \eqref{est 404}. Now we take $\delta \in (0, 2-\gamma_{2})$ so that $\delta < 1$ and estimate using \eqref{est 44}
\begin{equation}\label{est 46}
\lvert K(x',y) - K(x,y) \rvert \leq \lvert C_{\gamma_{2}} \rvert [ \lvert \phi(x') - \phi(y) \rvert \frac{(4-\gamma_{2}) 2^{5-\gamma_{2}} \lvert x'-x \rvert}{\lvert x-y\rvert^{5-\gamma_{2}}} + \frac{ \lVert \phi \rVert_{C_{x}^{2-\gamma_{2}}} \lvert x' - x \rvert^{2-\gamma_{2}}}{\lvert x-y \rvert^{4-\gamma_{2}}} ].
\end{equation}  
Now using \eqref{est 44}, one can consider the following three cases and readily verify \eqref{est 45} starting from \eqref{est 46}:
\begin{enumerate} 
\item $1 \leq \lvert x' - x \rvert < \frac{1}{2} \lvert x-y \rvert$, 
\item $\lvert x' - x \rvert \leq 1 < \frac{1}{2} \lvert x-y \rvert$ or $\lvert x' - x \rvert < 1 \leq \frac{1}{2} \lvert x-y \rvert$, 
\item $\lvert x' - x \rvert < \frac{1}{2} \lvert x-y \rvert \leq 1$. 
\end{enumerate}  
This completes the proof that $T$ is a SIO and thus the first hypothesis of the T(1) theorem. 

To verify the second hypothesis, we observe that $T^{\ast} = -[\Lambda^{2-\gamma_{2}}, \phi] = -T$ which implies $-T^{\ast}(1) = T(1) = -\sum_{k=1}^{2} \mathcal{R}_{k}^{2} \Lambda^{2-\gamma_{2}} \phi$ where $\mathcal{R}_{k}^{2}$ are Calder$\acute{\mathrm{o}}$n-Zygmund operators (e.g., \cite[p. 106]{L02}) and thus bounded from $L^{\infty}(\mathbb{R}^{2})$ to BMO$(\mathbb{R}^{2})$ (e.g., \cite[Theorem 6.8 on p. 54]{L02}) indicating that $\mathcal{R}_{k}^{2}\Lambda^{2-\gamma_{2}} \phi \in BMO(\mathbb{R}^{2})$ so that $T^{\ast}(1), T(1) \in BMO(\mathbb{R}^{2})$. Finally, the fact that $T$ has the WBP can be readily proven via a criteria in \cite[Section IV]{DJ84} because $K(x,y) = -K(y,x)$. By T(1) theorem from \cite{DJ84} we now conclude that $T$ is bounded from $L^{2}(\mathbb{R}^{2})$ to itself. Thus, by definition on \cite[p. 372]{DJ84} (see also \cite[Definition 6.4 on p. 53]{L02}), we deduce that it is a Calder$\acute{\mathrm{o}}$n-Zygmund operator with a norm less than or equal to $\lVert \phi \rVert_{C_{x}^{2-\gamma_{2}}} + \lVert \Lambda^{2-\gamma_{2}} \phi \rVert_{L_{x}^{\infty}}$ (e.g., see \cite[Theorem 8.3.3]{G09} concerning the norm). By \cite[Theorem 6.8 on p. 54]{L02} this allows us to deduce that $T$ is bounded from $L^{p}(\mathbb{R}^{2})$ to $L^{p}(\mathbb{R}^{2})$ for all $p \in (1,\infty)$ and hence conclude \eqref{est 51}. To prove \eqref{est 52}, we observe that 
\begin{equation}
\nabla [\Lambda^{2-\gamma_{2}}, \phi] f = [\Lambda^{2-\gamma_{2}}, \phi] \nabla f + [\Lambda^{2-\gamma_{2}}, \nabla \phi] f.
\end{equation} 
Because $\phi \in C^{1, 2-\gamma_{2}} (\mathbb{T}^{2})$ such that $\Lambda^{2-\gamma_{2}} \phi \in W^{1, \infty} (\mathbb{T}^{2})$ by hypothesis, we can apply \eqref{est 51} with $p=2$ to $\lVert [\Lambda^{2-\gamma_{2}}, \phi] \nabla f \rVert_{L_{x}^{2}}$ and $\lVert [\Lambda^{2-\gamma_{2}}, \nabla \phi] f \rVert_{L_{x}^{2}}$ and interpolate to deduce \eqref{est 52}. 
\end{proof}  

\begin{lemma}\label{Lemma 6.4}
\rm{(\cite[Lemma A.2]{BSV19}, \cite[Proposition D.1]{BDIS15})} 
Within this lemma, we denote $D_{t}\triangleq \partial_{t} + u \cdot \nabla$. Consider a smooth solution $f$ to \eqref{est 402} with a given $u, f_{0}$, and $g$, and recall that $\Phi =X^{-1}$ with $X$ defined by \eqref{flux}. Then, for $t > t_{0}$,
\begin{subequations}\label{est 197}
\begin{align}
& \lVert f (t) \rVert_{C_{x}} \leq \lVert f_{0} \rVert_{C_{x}} + \int_{t_{0}}^{t} \lVert g(\tau) \rVert_{C_{x}} d\tau, \label{est 197a} \\
& \lVert Df (t) \rVert_{C_{x}} \leq \lVert D f_{0} \rVert_{C_{x}} e^{( t - t_{0}) \lVert Du \rVert_{C_{t, x}}} + \int_{t_{0}}^{t} e^{(t- \tau)  \lVert D u \rVert_{C_{t, x}}} \lVert Dg(\tau) \rVert_{C_{x}} d\tau;\label{est 197b} 
\end{align}
\end{subequations}  
more generally, for any $N \in \mathbb{N} \setminus \{1\}$, there exists a constant $C = C(N)$ such that 
\begin{align}
\lVert D^{N} f(t) \rVert_{C_{x}}  \leq& ( \lVert D^{N} f_{0} \rVert_{C_{x}} + C(t-t_{0})  \lVert D^{N} u \rVert_{C_{t,x}} \lVert Df_{0} \rVert_{C_{x}} ) e^{C(t-t_{0})  \lVert D u \rVert_{C_{t,x}}} \nonumber\\
&+ \int_{t_{0}}^{t} e^{C(t-\tau)  \lVert Du \rVert_{C_{t,x}}} (  \lVert D^{N} g(\tau) \rVert_{C_{x}} + C(t-\tau)  \lVert D^{N} u \rVert_{C_{t,x}} \lVert Dg(\tau) \rVert_{C_{x}} )d\tau. \label{est 198} 
\end{align} 
Moreover,
\begin{subequations}\label{est 199}
\begin{align}
&\lVert D\Phi(t) - \Id \rVert_{C_{x}} \leq e^{(t-t_{0}) \lVert Du \rVert_{C_{t,x}}} - 1 \leq (t- t_{0}) \lVert Du \rVert_{C_{t,x}} e^{(t- t_{0}) \lVert Du \rVert_{C_{t,x}}}, \label{est 199a}\\
& \lVert D^{N} \Phi(t) \rVert_{C_{x}} \leq C(t-t_{0})  \lVert D^{N} u \rVert_{C_{t,x}} e^{C(t-t_{0})  \lVert Du \rVert_{C_{t,x}}} \hspace{3mm} \forall \hspace{1mm} N \in \mathbb{N}\setminus \{1\}. \label{est 199b}
\end{align}
\end{subequations}
\end{lemma} 

\section*{Acknowledgments}
The author expresses deep gratitude to Prof. Andrei Tarfulea and Prof. Zachary Bradshaw for valuable discussions and Dr. Raj Beekie for insightful discussions on convex integration. This work was supported by a grant from the Simons Foundation (962572, KY).

\end{document}